\pdfoutput=1
\RequirePackage{ifpdf}
\ifpdf 
\documentclass[pdftex]{sigma}
\else
\documentclass{sigma}
\fi

\numberwithin{equation}{section}

\newtheorem{Theorem}{Theorem}[section]
\newtheorem*{main}{Main Theorem}
\newtheorem{Lemma}[Theorem]{Lemma}
\theoremstyle{definition}
\newtheorem{Question}[Theorem]{Question}
\newtheorem{Definition}[Theorem]{Definition}
\newtheorem{Remark}[Theorem]{Remark}

\usepackage{tikz}
\newcommand{\del}[1]{\frac{\partial}{\partial #1}}
\newcommand{\indel}[1]{\partial/\partial #1}

\newcommand{\oneline}[6]{
\tikz{
\tikzstyle{mycir}=[circle,minimum size=0.8cm,draw]
\node[mycir, label=center:#1](u1) at (-3,0) {} ;
\node[mycir, label=center:#2](u2) at (0,0) {} ;
\node[mycir, label=center:#3](u3) at (3,0) {} ;
\node[mycir, label=center:#4](v1) at (-3,1.2) {} ;
\node[mycir, label=center:#5](v2) at (0,1.2) {} ;
\node[mycir, label=center:#6](v3) at (3,1.2) {} ;
\draw[thick] (u1)--(u2);
\draw[thick] (u2)--(u3);
\draw[thick] (u1)--++(-1,0);
\draw[thick] (u3)--++(1,0);
\draw[thick, dotted] (v1)--(u1);
\draw[thick, dotted] (v2)--(u2);
\draw[thick, dotted] (v3)--(u3);
}
}

\newcommand{\onelinwlab}[7]{
\tikz{
\tikzstyle{mycir}=[circle,minimum size=0.8cm,draw]
\node[label=center:#7] at (-5,0) {} ;
\node[mycir, label=center:#1](u1) at (-3,0) {} ;
\node[mycir, label=center:#2](u2) at (0,0) {} ;
\node[mycir, label=center:#3](u3) at (3,0) {} ;
\node[mycir, label=center:#4](v1) at (-3,1.2) {} ;
\node[mycir, label=center:#5](v2) at (0,1.2) {} ;
\node[mycir, label=center:#6](v3) at (3,1.2) {} ;
\draw[thick] (u1)--(u2);
\draw[thick] (u2)--(u3);
\draw[thick] (u1)--++(-1,0);
\draw[thick] (u3)--++(1,0);
\draw[thick, dotted] (v1)--(u1);
\draw[thick, dotted] (v2)--(u2);
\draw[thick, dotted] (v3)--(u3);
}
}

\begin{document}
\allowdisplaybreaks

\newcommand{\arXivNumber}{1804.10664}

\renewcommand{\PaperNumber}{122}

\FirstPageHeading

\ShortArticleName{Quadratic Differential Equations without Multivalued Solutions}

\ArticleName{Quadratic Differential Equations in Three Variables\\ without Multivalued Solutions: Part~I}

\Author{Adolfo GUILLOT}

\AuthorNameForHeading{A.~Guillot}

\Address{Instituto de Matem\'aticas, Universidad Nacional Aut\'onoma de M\'exico,\\ Ciudad Universitaria, Mexico City 04510, Mexico}
\Email{\href{mailto:adolfo.guillot@im.unam.mx}{adolfo.guillot@im.unam.mx}}
\URLaddress{\url{http://www.matem.unam.mx/~guillot/}}

\ArticleDates{Received May 01, 2018, in final form November 05, 2018; Published online November 11, 2018}

\Abstract{For ordinary differential equations in the complex domain, a central problem is to understand, in a given equation or class of equations, those whose solutions do not present multivaluedness. We consider autonomous, first-order, quadratic homogeneous equations in three variables, and begin the classification of those which do not have multivalued solutions.}


\Classification{34M55; 34M45; 34M35}

{\small
\setcounter{tocdepth}{2}
\tableofcontents}

\section{Introduction}

In the complex domain, the solutions of an ordinary differential equation may be multivalued. At the beginning of the XX\textsuperscript{th} century, partly motivated by the success of the theory of elliptic functions and by the work of Fuchs and Poincar\'e on equations of the first order, Painlev\'e posed the problem ``to determine all the algebraic differential equations of the first order, then of the second, then of the third order etc., whose general solution is uniform''~\cite{painleve}. We will consider autonomous quadratic homogeneous ordinary differential equations in dimension three, equations of the form~$z_i'=P_i(z_1,z_2,z_3)$, $i=1,\ldots, 3$, for~$P_i$ a quadratic homogeneous polynomial; following Painlev\'e's behest, we will address the problem of classifying those which have exclusively single-valued solutions.

One motivation for this study comes from investigating the ``Painlev\'e property'' (absence of movable critical points) for nonautonomous ordinary differential equations. In one variable, after the works of Fuchs and Poincar\'e on first-order equations (see~\cite{pan-sebastiani}), Painlev\'e's achievement for equations of the second order~\cite{painleve} (corrected and completed by Gambier~\cite{gambier}) was followed by the works of Chazy~\cite{chazy} for third-order equations (completed by Cosgrove~\cite{cosgrove-chazy}) and by those of Cosgrove~\cite{cosgrovep2, cosgrovep1} for fourth-order ones. First-order equations in more than one variable have not received the same attention. For two variables, we have Garnier's work~\cite{garnier} as well as Kimura and Matuda's announcement~\cite{kimura-matuda}. For three variables, there is, among others, Exton's work~\cite{exton}, but, to our best knowledge, there has been no attempt to systematically study first-order equations in three variables.

If we have an equation of the form~$z_i'=f_i(z_1,\ldots,z_n;t)$, with~$f_i$ a polynomial of degree~$d+1$ in~$z_1,\ldots,z_n$ whose coefficients are analytic functions of~$t$, and~$f_i=\sum\limits_{j=0}^{d+1} P_{i,j}$, for~$P_{i,j}$ homogeneous of degree~$j$ in~$z_1,\ldots,z_n$ then, for~$\alpha\neq 0$, $\zeta_i=\alpha^{-1}z_i$ and~$\tau=\alpha^{d}t$, ${\rm d}\zeta_i/{\rm d}\tau=\sum\limits_{j=0}^{d+1}\alpha^{d+1-j} P_{i,j}\big(\zeta;\alpha^{d}\tau\big)$. In the limit, as~$\alpha\to 0$, we obtain the homogeneous autonomous system~${\rm d}\zeta_i/{\rm d}\tau= P_{i,d+1}(\zeta;0)$. If all the solutions to the original equation are single-valued or if it has the Painlev\'e property, the solutions to the latter will be single-valued~\cite[Section~6]{pain-BSMF}. Thus, the homogeneous and autonomous cases provide obstructions for the presence of the Painlev\'e property or the absence of multivalued solutions in the inhomogeneous and nonautonomous ones. Although simpler, these may still exhibit some complex phenomena, like the presence of \emph{essential movable singularities} (whose existence is, according to Painlev\'e~\cite{painleve}, one of the major difficulties in the subject). In this way, the analysis of quadratic homogeneous systems contributes to the fulfillment of Painlev\'e's program for first-order equations in three variables.

Homogeneous autonomous equations without multivalued solutions are also important for understanding holomorphic flows and, more generally, holomorphic actions of complex Lie groups on complex manifolds. For instance, their study is a central ingredient in the results of Ghys and Rebelo~\cite{ghys-rebelo}, from which we can extract the following statement: \emph{if~$X$ is a non-identically zero semicomplete holomorphic vector field on the surface~$M$ having an isolated singularity at~$p\in M$ and if, in some chart centered at~$p$, we have that~$X=X^{2}+X^{3}+\cdots$ for~$X^{j}$ homogeneous of degree~$j$, $X$ is, in a neighborhood of~$p$, conjugated to~$X^2$ up to a reparametrization of the solutions}. Halphen systems, which occupy a distinguished place among homogeneous autonomous equations, are essential to the understanding of holomorphic actions of~$\mathrm{SL}(2,\mathbb{C})$ on complex threefolds~\cite{guillot-halphen}.

Lastly, some quadratic homogeneous systems are historically relevant, like Euler's equation of the top~\cite[Chapter~VI]{mechanics} or Halphen's system~\cite{halphen1}. In the late~XIX\textsuperscript{th} century, Hoyer investigated, within a class containing Euler's equations, the quadratic homogeneous equations that could be solved by elliptic functions~\cite{hoyer}, and Kowalevski studied some Lotka--Volterra homogeneous equations~\cite[Chapter~VI]{audin}, a subject that foreshadowed her renowned work on the top~\cite{kowalevski}.

We will study the quadratic differential equations under their vector field form. We will consider vector fields on~$\mathbb{C}^3$ of the form~$\sum\limits_{i=1}^3 P_i(z_1,z_2,z_3)\indel{z}_i$ with~$P_i$ a quadratic homogeneous polynomial. We will say that a vector field is \emph{univalent} or \emph{semicomplete} if it does not have multivalued solutions (these notions will be made precise in Section~\ref{sec:uni-semi}).

We aspire to have a classification of the univalent quadratic homogeneous vector fields on~$\mathbb{C}^3$. There are two situations where such a classification exists. A classification of the univalent quadratic homogeneous systems that are symmetric with respect to the full permutation group in three variables follows essentially from Chazy's work (see Section~\ref{chazy-sym}). We have given a~classification of the univalent quadratic homogeneous vector fields which are divergence-free~\cite[Theorem~C]{guillot-fourier}.

In this article we present the first part of a potentially comprehensive classification. In order to state our results, let us mention some facts that will be later explained with detail. To a univalent quadratic homogeneous vector field on~$\mathbb{C}^3$ having an isolated singularity at the origin we may associate seven unordered pairs of integers called \emph{Kowalevski exponents} (Section~\ref{sec:kow}), which may be used to group them in six \emph{families} (Section~\ref{sec:families}). Our main result is a classification of the univalent vector fields in the first family (those for which the product of the two exponents in one of these pairs is equal to~$1$) and in the second one (those for which the product of the two exponents in two of these pairs is equal to~$2$).

\begin{main} The quadratic homogeneous vector fields on~$\mathbb{C}^3$ which have an isolated singularity at~$0$, are univalent, and are part of either the first or of the second family belong, up to a linear change of coordinates, to the following list $($equations~{\rm \ref{chazyalt2}} and~{\rm \ref{chazyalt1}--\ref{abel:burn8-1}} are two inequivalent equations each, one for each determination of the square root$)$.

Vector fields of Halphen type:
\begin{gather} \big[\alpha_1 x^2+(1-\alpha_1)(xy-yz+zx)\big]\del{x} + \big[\alpha_2 y^2+(1-\alpha_2)(xy+yz-zx)\big]\del{y}\nonumber\\
\qquad{} + \big[\alpha_3z^2+(1-\alpha_3)(-xy+yz+zx)\big]\del{z}, \qquad m_i=\frac{\alpha_1+\alpha_2+\alpha_3-2}{\alpha_i},\tag{I}\label{halpheneq}\\
 x(x-y)\del{x}+ y[y+3x+(n+2)z]\del{y}+\big(z^2+\beta xy+2xz\big)\del{z}, \nonumber\\
 \qquad{} \beta=\frac{8m^2n}{(mn+2n+2m)(mn-2n+2m)}, \tag{II}\label{pseudobemol}\\
 x[2x-(k+6)y]\del{x} +y[(k+6)y-kz]\del{y} +\big[kz^2+(k-2)(3y-2z)x\big] \del{z},\tag{III} \label{t4} \\
\big[(5m-6)x^2+2(m-6)xy-(7m-6)(x+y)z\big]\del{x} +\big[(m-6)y^2-(m+6)(x+y)z\big] \del{y}\nonumber\\
\qquad{} +z[4mz-(5m-6)x+3(m-6)y]\del{z}.\tag{IV}\label{t1} \end{gather}

Vector fields all of whose solutions are rational:
\begin{gather} x^2\del{x}+y[m y-(m-1)x]\del{y}\nonumber\\
\qquad{} +\frac{1}{4}\big[\big(1-p^2\big)x^2- \big[r^2-p^2+m^2\big(1-q^2\big)\big]xy+m^2\big(1-q^2\big)y^2\big]\del{z},\tag{V} \label{fq-1m-m}\\
 x[x-y+(1-n)z]\del{x}+y(y-x+nz)\del{y} +\big(z^2+\beta^{-1} xy\big)\del{z},\nonumber\\
 \qquad{} \beta=\frac{1}{8}(2n-1+q)(2n-1262-q),\tag{VI} \label{ricell-1} \\
x[x-y+(1-n)z]\del{x}+y[y-x+(n+3)z]\del{y} +\big(z^2+\beta^{-1} xy\big)\del{z}, \nonumber\\
\qquad{} \beta=-\frac{1}{4}(n+1+2q)(n+1-2q).\tag{VII} \label{ricell+2} \end{gather}

Vector fields with a first integral reducing to Riccati equations with elliptic coefficients:
\begin{gather} x(x-2y)\del{x}+y(y-2x)\del{y}\nonumber\\
\qquad{} +\frac{1}{4}\big[\big(1-p^2\big)x^2+\big(p^2+q^2-1 -r^2\big)xy+\big(1-q^2\big)y^2\big]\del{z}, \tag{VIII} \label{fq333}\\
\big(2y^2-x^2\big)\del{x}+xy\del{y}+\frac{1}{8}\big[2\big(1-p^2\big)x^2+\big(q^2-r^2\big)xy+\big(2p^2-q^2-r^2\big)y^2\big]\del{z},\tag{IX} \label{fq442} \\
x(2x-5y)\del{x}+y(y-4x)\del{y}\nonumber\\
\qquad{} +\frac{1}{4}\big[4\big(1-q^2\big)x^2+\big(4+4q^2+r^2-9p^2\big)xy+\big(1-r^2\big)y^2\big]\del{z}, \tag{X} \label{fq236} \\
x[x-y+(1-n)z]\del{x}+y[y-x+(n-2)z]\del{y} +\big(z^2+\beta^{-1} xy\big)\del{z}, \nonumber\\
\qquad{} \beta=\frac{1}{24}(2n-3+q)(2n-3-q), \tag{XI} \label{ricell-3}\\
x[x-y+(1-n)z]\del{x}+y[y-x+(n-5)z]\del{y} +\big(z^2+\beta^{-1} xy\big)\del{z}, \nonumber\\ \qquad{} \beta=\frac{1}{12}(n-3+2q)(n-3-2q),\tag{XII} \label{ricell-6} \\
x[x-y+(1-n)z]\del{x}+y[y-x+(n-3)z]\del{y} +\big(z^2+\beta^{-1} xy\big)\del{z}, \nonumber\\ \qquad{} \beta=\frac{1}{8}(n-2+q)(n-2-q), \tag{XIII} \label{ricell-4}\\
x[x-(1-q)y] \del{x}+y[(1-q)y-5x-z ]\del{y} \nonumber\\
\qquad{} + \big[z^2-(q+5)(q-3)xy+(7+q)xz\big] \del{z}, \tag{XIV} \label{t3}\\
\big[(k-1) z(x+y)-(k+1)x^2\big]\del{x} \nonumber\\
\qquad{} +\big[(k+5)y^2-(k+7)(x+y)z\big]\del{y}+z[4z+(k+1) x -(k+5)y]\del{z}.\tag{XV} \label{t2}
\end{gather}

Vector fields with a first integral reducing to constant vector fields on Abelian surfaces:
\begin{gather} \big[\big(3x^2-y^2-2xz\big)+2\beta y(x-z)\big]\del{x}+\big[2y(3x-2z)+\beta\big(y^2-3x^2\big)\big]\del{y} \nonumber\\
\qquad{}+2[(z-3x)-\beta y]z\del{z}, \qquad \beta\in\mathbb{C},\tag{XVI} \label{linsnetovf}\\
x(x -y)\del{x}+y(y-4x-z)\del{y}+\big[z^2-\big(1+2\sqrt{5}\big) xy +4xz\big]\del{z}, \tag{XVII} \label{chazyalt2}\\
 x(x -y)\del{x}+y(y-9x-z)\del{y}+\big(z^2-6 xy +14xz\big)\del{z},\tag{XVIII} \label{chazyalt4} \\
x(x -y)\del{x}+y(y-4x-z)\del{y}+\left[z^2+\frac{1}{11}\big(17+12\sqrt{3}\big) xy +4xz\right]\del{z}, \tag{XIX}\label{chazyalt1}\\
x(x -y)\del{x}+y(y-3x-z)\del{y}+\left[z^2-\frac{3}{11}\big(9+7\sqrt{3}\big) xy + 2xz\right]\del{z}, \tag{XX}\label{chazyalt3}\\
x(x -y)\del{x}+y(y-11x-z)\del{y}+\big[z^2-\big(1+3\sqrt{3}\big) xy +18xz\big]\del{z},\tag{XXI}\label{chazyalt5}\\
\left[x^2+\left(1-\frac{\sqrt{5}}{2}\right)yz\right]\del{x}+y(2x+y-2z)\del{y}+z(z-2x-y)\del{z},\tag{XXII}\label{chazyixbis}\\
\left[x^2+\left(1+\frac{3\sqrt{3}}{4}\right)yz\right]\del{x}+y(2x+y-2z)\del{y}+z(z-2x-y)\del{z},\tag{XXIII}\label{chazyxbis}\\
\frac{1}{2}\big[2x^2-\big(3\sqrt{3}+9\big)yz+\big(3+3\sqrt{3}\big)zx\big]\del{x}\nonumber\\
\qquad{} + \frac{1}{2}\big[2\big(4+\sqrt{3}\big)y^2-\big(21+9\sqrt{3}\big)yz+\big(7\sqrt{3}+9\big)zx\big]\del{y}\nonumber\\
\qquad{} + \big[\big(3+\sqrt{3}\big)z-x-\big(4+\sqrt{3}\big)y\big]z\del{z},\tag{XXIV}\label{abel:square12-1}\\
\big[3\big(1+\sqrt{3}\big)zx-4yz\sqrt{3}-x^2\big]\del{x}+ \big[3\sqrt{3}xz+\big(2\sqrt{3}-4\big)y^2+6\big(1-\sqrt{3}\big)yz\big]\del{y} \nonumber\\
\qquad{} +\big[x+\big(4-2\sqrt{3}\big)y+\big(\sqrt{3}-3\big)z \big]z\del{z},\tag{XXV}\label{abel:square12-2}\\
\big[2x^2+\big(\sqrt{3}-6\big)yz\big]\del{x}+\big[\big(5-2\sqrt{3}\big)y+\big(3\sqrt{3}-9\big)z\big]y\del{y} \nonumber\\
\qquad{} +\big[\big(3-\sqrt{3}\big)z-2x+\big(2\sqrt{3}-5\big)y\big]z\del{z},\tag{XXVI}\label{abel:square12-3}\\
\big[2\big(2-\sqrt{2}\big)yz+2\big(\sqrt{2}-1\big)zx-x^2\big]\del{x}+\big[\big(8-5\sqrt{2}\big)yz+\big(4\sqrt{2}-5\big)zx+\big(\sqrt{2}-2\big)y^2\big]\del{y}\nonumber\\
\qquad{}+ \big[x+\big(2-\sqrt{2}\big)(y-z)\big]z\del{z},\tag{XXVII}\label{abel:burn8-1}\\
\big[2\big(2+\sqrt{2}\big)x^2-\big(6+5\sqrt{2}\big)zx-\big(2-3\sqrt{2}\big)zy \big]\del{x}\nonumber\\
\qquad{}+\big[2\big(2-\sqrt{2}\big)y^2-\big(2+3\sqrt{2}\big)zx-\big(6-5\sqrt{2}\big)zy \big]\del{y}\nonumber\\
\qquad{}+\big[4z^2-2\big(2+\sqrt{2}\big)zx-2\big(2-\sqrt{2}\big)zy\big]\del{z}.\tag{XXVIII}\label{abel:burn8-2}
\end{gather}
The seven pairs of Kowalevski exponents for each one of these appear in Table~{\rm \ref{table:kow}}. For the equations to be univalent, it is necessary that the Kowalevski exponents are integers. Equations~{\rm \ref{linsnetovf}} to~{\rm \ref{abel:burn8-2}} are always univalent; for equations~{\rm \ref{halpheneq}} to~{\rm \ref{t2}}, sufficient conditions for univalence appear in Table~{\rm \ref{table:conds}} $($these give, in each case, a Zariski-dense subset of the space of parameters$)$.
\end{main}

\begin{table}\centering
\begin{tabular}{cl}
Equation & Kowalevski exponents \\ \hline \hline
\ref{halpheneq} & $(-1,-1)$, $(1,m_1)$, $(1,-m_1)$, $(1,m_2)$, $(1,-m_2)$, $(1,m_3)$, $(1,-m_3)$ \\
\ref{pseudobemol} & $(-1,-2)$, $(1,2)$, $(1,n)$, $(1,-n-1)$, $(-n,n+1)$, $(1,m)$, $(1,-m)$ \\
\ref{t4} & $(1,2)$, $(1,2)$, $(-2,-3)$, $(-2,3)$, $(1,k-1)$, $(k,1-k)$, $(1,-k)$ \\
\ref{t1} & $(-1,-2)$, $(1,2)$, $(1,3)$, $(1,-3)$, $(1,-m)$, $(1,m-1)$, $(1-m,m)$ \\ \hline
\ref{fq-1m-m} & $(1,1)$, $(1,q)$, $(1,-q)$, $(m,p)$, $(m,-p)$, $(-m,r)$, $(-m,-r)$ \\
\ref{ricell-1} & $(1,2)$, $(1,2)$, $(n,1-n)$, $(1,-n)$, $(1,n-1)$, $(-2,q)$, $(-2,-q)$ \\
\ref{ricell+2} & $(1,2)$, $(1,2)$, $(n,-2-n)$, $(-2,-n)$, $(-2,n+2)$, $(1,q)$, $(1,-q)$ \\
\hline
\ref{fq333} & $(1,1)$, $(3,p)$, $(3,-p)$, $(3,q)$, $(3,-q)$, $(3,r)$, $(3,-r)$\\
\ref{fq442} & $(1,1)$, $(2,p)$, $(2,-p)$, $(4,q)$, $(4,-q)$, $(4,r)$, $(4,-r)$ \\
\ref{fq236} & $(1,1)$, $(2,p)$, $(2,-p)$, $(3,q)$, $(3,-q)$, $(6,r)$, $(6,-r)$ \\
\ref{ricell-3} & $(1,2)$, $(1,2)$, $(3,-n)$, $(n,3-n)$, $(3,n-3)$, $(6,q)$, $(6,-q)$ \\
\ref{ricell-6} & $(1,2)$, $(1,2)$, $(6,-n)$, $(n,6-n)$, $(6,n-6)$, $(3,q)$, $(3,-q)$ \\
\ref{ricell-4} & $(1,2)$, $(1,2)$, $(4,-n)$, $(n,4-n)$, $(4,n-4)$, $(4,q)$, $(4,-q)$ \\
\ref{t3} & $(1,2)$, $(1,2)$, $(-2,3)$, $(2,3)$, $(6,q)$, $(6,-6-q)$, $(6+q,-q)$ \\
\ref{t2} & $(1,2)$, $(1,2)$, $(-1,3)$, $(1,3)$, $(6,k)$, $(6,-6-k)$, $(6+k,-k)$ \\ \hline
\ref{linsnetovf} & $(1,2)$, $(1,2)$, $(2,3)$, $(2,3)$, $(2,3)$, $(-1,3)$, $(-1, 6)$ \\
\ref{chazyalt2} & $(1,2)$, $(1,2)$, $(-2,3)$, $(-3,5)$, $(3,2)$, $(2,5)$, $(-3,10)$ \\
\ref{chazyalt4} & $(1,2)$, $(1,2)$, $(-2,3)$, $(-13,10)$, $(13,-3)$, $(2,5)$, $(2,5)$ \\
\ref{chazyalt1} & $(1,2)$, $(1,2)$, $(-2,3)$, $(-3,5)$, $(3,2)$, $(3,4)$, $(-5,12)$ \\
\ref{chazyalt3} & $(1,2)$, $(1,2)$, $(-2,3)$, $(-1,4)$, $(1,3)$, $(2,5)$, $(-5,12)$ \\
\ref{chazyalt5} & $(1,2)$, $(1,2)$, $(-2,3)$, $(-17,12)$, $(17,-5)$, $(2,5)$, $(3,4)$ \\
\ref{chazyixbis} & $(1,2)$, $(1,2)$, $(-1,3)$, $(1,3)$, $(-3,5)$, $(2,5)$, $(-3,10)$\\
\ref{chazyxbis} & $(1,2)$, $(1,2)$, $(-1,3)$, $(1,3)$, $(-3,5)$, $(3,4)$, $(-5,12)$ \\
\ref{abel:square12-1} & $(1,2)$, $(1,2)$, $(-1,3)$, $(1,4)$, $(2,3)$, $(-2,7)$, $(-7,12)$ \\
\ref{abel:square12-2} & $(1,2)$, $(1,2)$, $(-1,3)$, $(1,4)$, $(2,3)$, $(-2,7)$, $(-7,12)$ \\
\ref{abel:square12-3} & $(1,2)$, $(1,2)$, $(-1,3)$, $(1,4)$, $(2,3)$, $(-2,7)$, $(-7,12)$ \\
\ref{abel:burn8-1} & $(1,2)$, $(1,2)$, $(-1,3)$, $(1,4)$, $(2,3)$, $(-3,8)$, $(-3,8)$ \\
\ref{abel:burn8-2} & $(1,2)$, $(1,2)$, $(-1,3)$, $(1,4)$, $(2,3)$, $(-3,8)$, $(-3,8)$ \\ \hline
\end{tabular}
\caption{The seven pairs of Kowalevski exponents for each one of the equations appearing in the Main Theorem. }\label{table:kow}
\end{table}

\begin{table}\centering
\begin{tabular}{cl}
Equation & Conditions \\ \hline \hline
\ref{halpheneq} & $m_i\geq 2$ \\
\ref{pseudobemol} & $n\geq 2$ and $m\geq 2$ \\
\ref{t4} & $k\geq 2$ \\
\ref{t1} & $m \geq 2$ \\
\ref{fq-1m-m} & $m\nmid p$, $m\nmid r$, $p+r$ or~$p-r$ equals~$(q-1-2j)m$ for~$j\in\{0,\ldots,q-1\}$ \\
\ref{ricell-1} & $2\nmid q$ and $q<2n$ \\
\ref{ricell+2} & $2\nmid n$ and $n<2q$ \\
\ref{fq333} & $3\nmid p$, $3\nmid q$ and $3\nmid r$ \\
\ref{fq442} & $2\nmid p$ \\
\ref{fq236} & $6 \nmid r$, $3\nmid q$ and $2\nmid p$ \\
\ref{ricell-3} & $6\nmid q$ and $3\nmid n$ \\
\ref{ricell-6} & $6\nmid n$ and $3\nmid q$ \\
\ref{ricell-4} & $2\nmid n$ and $2\nmid q$ \\
\ref{t3} & $6\nmid q$ \\
\ref{t2} & $6\nmid k$ \\ \hline
\end{tabular}
\caption{Some sufficient conditions for univalence.}\label{table:conds}
\end{table}

The sufficient conditions for univalence in Table~\ref{table:conds} need not be necessary ones. We will give other sufficient conditions and mention the cases when these are necessary as we integrate these equations (we will not give necessary and sufficient conditions for univalence for all equations).

We have grouped the equations according to their main geometric feature (which also determines its integration). In the above equations, there is no phenomenon that is not already present in the integration of the Chazy homogeneous equations (Section~\ref{sec:chazy}).

Equations of Halphen type will be defined in Section~\ref{sec:halphen}; they include equation~\ref{halpheneq}, the family introduced by Halphen in~\cite{halphen3}. (In family~\ref{t4}, the case~$k=6$ is not of Halphen type; it has a first integral and reduces to a Riccati equation with elliptic coefficients.) In equations~\ref{halpheneq}--\ref{t1}, for most cases (the ``hyperbolic'' ones), the solutions have a natural boundary. By Corollary~D in~\cite{guillot-rebelo}, these do not have rational first integrals. Some of the equations~\ref{fq333}--\ref{t2} (which have a first integral and reduce to Riccati equations with elliptic coefficients) have a supplementary first integral; for these, the solutions are given by elliptic functions. For the vector fields that reduce to constant vector fields on Abelian surfaces (Section~\ref{sec:abelian}), there is an automorphism of the surface associated to the vector field. The essential information appears in Table~\ref{table:abel}. Although the occurrence in Chazy's classification of vector fields that reduce to constant vector fields on Abelian surfaces (Section~\ref{sec:chazy}) may be seen as anecdotal, the present classification exhibits the ubiquity of this phenomenon in dimension three.

Halphen's equations have, both in the ``hyperbolic'' and ``parabolic'' cases, movable singularities. For all other equations, the solutions are given by meromorphic functions defined in the whole complex line. Only the equations of Halphen type and those having exclusively rational solutions have one (and only one) radial orbit with two positive exponents. It corresponds to an asymptotic direction of escape for solutions defined in a neighborhood of infinity.

\begin{table}\centering
\begin{tabular}{cccccc}
Abelian surface & automorphism & order~$d$ & equations & $k$ & comments \\ \hline \hline
\begin{tabular}{c} Jacobian of \\ $\zeta^2=\xi^5-1$ \end{tabular} & \begin{tabular}{c} induced by \\ $(\zeta,\xi)\mapsto(-\zeta,\omega \xi)$ \\ $\omega^5=1$ \end{tabular} & $10$ & \begin{tabular}{c} \ref{chazyalt2} \\ \ref{chazyalt4} \\ \ref{chazyixbis} \end{tabular} & $3$ & \begin{tabular}{c} reduced \\ Chazy IX \end{tabular} \\ \hline
$E_i\times E_i$ & $(z,w)\mapsto ({\rm i}w,-{\rm i}(z+w))$ & $12$ & \begin{tabular}{c} \ref{chazyalt1} \\ \ref{chazyalt3} \\ \ref{chazyalt5} \\ \ref{chazyxbis} \end{tabular} & $5$ & \begin{tabular}{c} reduced \\ Chazy X \end{tabular} \\ \hline
$E_\rho\times E_\rho$ & $(z,w)\mapsto (-\rho z,-\rho w)$ & $6$ & \ref{linsnetovf} & $5$ & \begin{tabular}{c} Lins Neto's \\ equations \end{tabular} \\ \hline
$E_\rho\times E_\rho$ & $(z,w)\mapsto\big(\rho^2w, -\rho^2z\big)$ & 12 & \begin{tabular}{c} \ref{abel:square12-1} \\
\ref{abel:square12-2} \\ \ref{abel:square12-3} \end{tabular} & $7$ & \\ \hline
\begin{tabular}{c} Jacobian of \\ $\zeta^2=\xi\big(\xi^4-1\big)$ \end{tabular} & \begin{tabular}{c} induced by \\ $(\zeta,\xi)\mapsto \big(\lambda \zeta,\lambda^2 \xi\big)$ \\ $\lambda^8=1$ \end{tabular} & $8$ & \begin{tabular}{c} \ref{abel:burn8-1} \\ \ref{abel:burn8-2} \end{tabular} & $3$ & \begin{tabular}{c} particular case \\ of Cosgrove's \\ reduced F-V \end{tabular} \\ \hline
\end{tabular}
\caption{The Abelian surfaces with automorphisms related to equations~\ref{linsnetovf}--\ref{abel:burn8-2}. Here, $E_\mu=\mathbb{C}/\langle 1,\mu\rangle$, $\rho^3=1$. At the fixed point, the eigenvalues are~$\lambda$, a primitive~$d$\textsuperscript{th} root of unity and~$\lambda^k$.}\label{table:abel}
\end{table}

\looseness=-1 In broad lines, the strategy of the proof is, first, to use some arithmetical information to find equations that are potentially univalent and then to try to integrate them. In Section~\ref{sec:pre} we give the details on the first part and establish the terminology, conventions and notations that will be used throughout the article. At various points we will need to solve some systems of Diophantine equations. The nature of these equations makes it unreasonable to do it by hand but suitable to do it with the help of the computer. We have written some programs that solve them in the SageMath software~\cite{sagemath}. They are available as ancillary files at the \texttt{arXiv} page of this article.\footnote{\href{https://arxiv.org/abs/1804.10664}{https://arxiv.org/abs/1804.10664}.}

A complete classification of all the univalent quadratic homogeneous vector fields on~$\mathbb{C}^3$ which have an isolated singularity at~$0$ can, in principle, be achieved along the lines of this article (classify those in the third, fourth and fifth families, and then those in the seventh one) although it seems that, for this, a better knowledge of the algebraic relations between the Kowalevski exponents is needed. On the other hand, it is possible (and this would certainly condition the success of such a classification) that there is some vector field that has no evident obstruction for being univalent but whose univalence status might prove very difficult to establish (see the discussion around the history of the integration of the Chazy~IX and Chazy~X equations in Section~\ref{sec:chazy}). After all, we do not have an \emph{a priori} knowledge of the ways in which a quadratic homogeneous vector field on~$\mathbb{C}^3$ may be univalent (except in the case where it has an algebraic first integral, where our joint work with J.~Rebelo~\cite[Theorem~B]{guillot-rebelo} gives a geometric description of each level surface). This said, we do not know examples of univalent polynomial vector fields on~$\mathbb{C}^3$ without an algebraic first integral which are of not of Halphen type, and it may very well happen that all of these are.

There are univalent quadratic vector fields which do not have an isolated singularity at~$0$, and which do not fall within the scope of the Main Theorem. Among them we find historically relevant ones, like the previously mentioned examples of Euler's equations of the top, or Halphen's system~\cite{halphen1}
\begin{gather}\label{halphenclassic} (xy-yz+zx)\del{x} + (yz-zx+xy)\del{y}+ (zx-xy+yz)\del{z}. \end{gather}
Some of these appear as degenerations of some families (the above equation is the limit of equation~\ref{halpheneq} as~$m_i\to\infty$), but not all of them do. A full classification of the univalent quadratic vector fields would have to include these cases, where the obstructions for univalence given by the Kowalevski exponents will be less strong and which may require other tools and techniques.

\section{Preliminaries}\label{sec:pre}

We begin by giving some background as well as outlining the strategy of the proof of the Main Theorem. Further details appear in~\cite{guillot-fourier}. For standard facts on differential equations in the complex domain, including those concerning Riccati equations, we refer the reader to Hille's~\cite{hille} and Ince's books~\cite{ince}. (The integration of some of the Riccati equations with elliptic coefficients that appear here has been carried out in the companion article~\cite{guillot-riccati}.) The elementary facts on elliptic functions that we will use may be found in~\cite{lawden}.

\subsection{Univalence, semicompleteness}\label{sec:uni-semi}

We will say that the holomorphic vector field~$X$ on the manifold~$M$ is~\emph{univalent} (following Palais~\cite[Definition~VI, Chapter~III]{palais}) or~\emph{semicomplete} (following Rebelo~\cite[Definition~2.3]{rebelo-sing}) if for each $p\in M$ there is a domain~$\Omega_p\subset\mathbb{C}$, $0\in\Omega_p$ and a solution~$\phi\colon \Omega_p\to M$ of~$X$, $\phi(0)=p$, such that~$\phi\times \mathbb{I}\colon \Omega_p \to M\times \mathbb{C} $, $t\mapsto (\phi(t),t)$ is proper. (For autonomous systems of ordinary differential equations, where all critical points are movable, most notions related to the absence of multivaluedness of the solutions coincide.)

\subsection{Radial orbits and Kowalevski exponents}
Let~$V_n$ denote the space of quadratic homogeneous vector fields on~$\mathbb{C}^n$. For~$X\in V_n$, a \emph{radial orbit} of~$X$ is a line through the origin of~$\mathbb{C}^n$ that is invariant by~$X$; equivalently a radial orbit of~$X$ is an orbit of Euler's vector field~$E=\sum_i z_i\indel{z_i}$, different from the origin, that is invariant by~$X$. A radial orbit~$\rho$ of a quadratic homogeneous vector field~$X$ is said to be \emph{nondegenerate} if~$X$ does not vanish at any point of~$\rho$ other than the origin. A quadratic homogeneous vector field on~$\mathbb{C}^n$ is said to be \emph{nondegenerate} if all its radial orbits are nondegenerate (equivalently, if its singularity at~$0$ is an isolated one). A generic quadratic homogeneous vector field on~$\mathbb{C}^n$ is nondegenerate and has~$2^n-1$ radial orbits.

\subsubsection{The Kowalevski exponents of a radial orbit}\label{sec:kow}

There are two complex numbers that one can associate to a nondegenerate radial orbit which give a first-order approximation to the dynamics of~$X$ in the neighborhood of it and in which we can read an arithmetical obstruction for semicompleteness. These numbers, central in the ``Painlev\'e analysis'' of homogeneous equations, are referred to as~\emph{Kowalevski's exponents} (a term coined by Yoshida~\cite{yoshida}, motivated by the relevant role that they play in Kowalevski's analysis of the top~\cite{kowalevski}), \emph{Fuchs indices}, \emph{resonances}, or \emph{eigenvalues}. They are defined in a number of ways and there are several interpretations to them. We refer to~\cite{goriely} and~\cite{conte} for an overview as well as to our own account~\cite[Section~2.1]{guillot-fourier}.

Let~$X\in V_n$, $X=\sum_i P_i(z_1,\ldots,z_n)\indel{z_i}$. Let~$\rho$ be a nondegenerate radial orbit of~$X$. Let~$a=(a_1,\ldots,a_n)$ in~$\rho\setminus\{0\}$ be such that~$P_i(a)=a_i$ (it is unique for, since~$X$ is quadratic, $P_i(c a)=c^2 a_i$). A parametrization of~$\rho$ as a solution of~$X$ is~$t\mapsto (-a_1/t,\ldots,-a_n/t)$. The \emph{equation of variations} associated to this solution reads
\begin{gather*}t\frac{{\rm d}v_i}{{\rm d}t}=-\sum\limits_{j=1}^n \frac{\partial P_i}{\partial z_j}(a_1,\ldots,a_n)v_j, \end{gather*}
which may be written, for~$M_{ij}=\partial P_i/\partial z_j(a_1,\ldots,a_n)$, as~$t\xi'=-M\xi$. If~$\lambda$ is an eigenvalue of~$-M$ and $v\in\mathbb{C}^n$ is an eigenvector associated to it, $t^{\lambda}v$ is a solution to this equation. (By Euler's relation, $\sum\limits_{j=1}^n a_j \partial P_i/\partial z_j(a_1,\ldots,a_n)=2P_i(a_1,\ldots,a_n)=2a_i$, $(\mathbb{I}-M)a=-a$, and thus~$-1$ is always an eigenvalue of~$\mathbb{I}-M$.)

\begin{Definition} With the preceding notations, the \emph{Kowalevski exponents} of the nondegenerate radial orbit~$\rho$ of the vector field~$X$ are the roots of~$\det([\mathbb{I}-M]-\lambda \mathbb{I})/(\lambda+1)$.\footnote{Warning! The \emph{eigenvalues} of a radial orbit, as defined in~\cite{guillot-pointfixe} and used throughout~\cite{guillot-fourier}, differ from the Kowalevski exponents by a sign.}
\end{Definition}

The Kowalevski exponents (or \emph{exponents}, for short) of a nondegenerate radial orbit of a~semicomplete vector field are integers (for otherwise the equation of variations would have a~multivalued solution). This result also follows from the fact that, for~$E=\sum_i z_i\indel{z_i}$, \emph{if~$X$ is semicomplete then for every~$\lambda\in\mathbb{C}$ the vector field~$E-\lambda X$ is semicomplete and all of its solutions are, like those of~$E$, $2{\rm i}\pi$-periodic}~\cite[Section~2.1]{guillot-fourier}: if~$a$ is such that~$P_i(a)=a_i$ then~$E-X$ vanishes at~$a$, where its linear part is given by~$\mathbb{I}-M$ and where the periodicity of its solutions implies that the eigenvalues of this linear part are integers.

As an example, consider the vector field~$x_1^2\indel{x_1}+\sum\limits_{i=2}^n \lambda_ix_1x_i\indel{x_i}$ together with its nondegenerate radial orbit~$[1:0:\cdots:0]$. It has exponents~$-\lambda_2$, \ldots, $-\lambda_n$ and the solution~$\big({-}t^{-1},t^{-\lambda_2},\ldots,t^{-\lambda_n}\big)$, which is univalent only when~$\lambda_i\in\mathbb{Z}$ for all~$i$.

In a nondegenerate semicomplete vector field in~$V_n$ all the radial orbits are simple~\cite[Corollary~2.7]{guillot-fourier} and, in particular, for each radial orbit, the Kowalevski exponents are both nonzero (for otherwise the radial orbit would not be a simple one). A nondegenerate semicomplete vector field on~$\mathbb{C}^3$ has seven couples of nonzero integers for exponents.

\subsubsection{The associated foliations on projective spaces}\label{sec:fol} \looseness=-1 A quadratic homogeneous vector field~$X$ on~$\mathbb{C}^n$ induces a foliation~$\mathcal{F}$ on~${\mathbb{CP}}^{n-1}$. Radial orbits for~$X$ give singular points for~$\mathcal{F}$. For a nondegenerate radial orbit~$\rho$ of~$X$, the \emph{characteristic numbers} of the singularity of~$\mathcal{F}$ induced by~$\rho$ are the exponents of~$\rho$ (these characteristic numbers are defined only up to multiplication by a constant). Some obstructions for~$X$ to be semicomplete may be read directly from~$\mathcal{F}$. For example, \emph{if~$X$ is a semicomplete vector field and~$\rho$ a~nondegenerate radial orbit of~$X$, $\mathcal{F}$ is linearizable and has rational characteristic numbers at the singularity induced by~$\rho$}~\cite[Corollary~2.7]{guillot-fourier}. (Lemma~\ref{sing12} will be a very useful consequence of this.)

For~$E=\sum_iz_i\indel{z_i}$ and~$\ell$ a linear form in~$\mathbb{C}^3$, the vector fields~$X$ and~$X+\ell E$ induce the same foliation on~${\mathbb{CP}}^2$. In dimension three, for a radial orbit~$\rho$ which is nondegenerate for both~$X$ and~$X+\ell E$, the exponents with respect to~$X$ and with respect to~$X+\ell E$ coincide if and only if~$\ell(\rho)\equiv 0$ (see~\cite[Lemma~5]{guillot-pointfixe} and~\cite[Corollary~2.7]{guillot-fourier}).

For~$X\in V_3$ inducing the foliation~$\mathcal{F}$ on~${\mathbb{CP}}^2$, if~$\rho$ is a nondegenerate radial orbit of~$X$ with exponents~$(u,v)$, the number~$(u+v)^2/(uv)$ is called the \emph{Baum--Bott index} of the singular point of~$\mathcal{F}$ associated to~$\rho$. See~\cite[Chapter~3]{brunella-birational} for definitions and details.

\subsubsection{The Kowalevski exponents as functions on the space of vector fields}\label{sec:families}

In dimension three, a generic quadratic homogeneous vector field is locally characterized by its Kowalevski exponents:

\begin{Theorem}\label{eigen} Let~$\Psi\colon V_3\dashrightarrow \mathrm{Sym}^7\big(\mathrm{Sym}^2(\mathbb{C})\big)$ be the map that associates to a generic vector field on~$\mathbb{C}^3$ the two exponents of its seven radial orbits. The map~$\Psi$ has finite fibers. In other words, for a generic vector field in~$V_3$ there are, up to linear equivalence, only finitely many vector fields having the same set of exponents.
\end{Theorem}

A proof of this theorem was allegedly given in~\cite[Proposition~3.2]{guillot-fourier}. However, as J.V.~Pereira brought to our attention, the proof given there is incomplete. Let us complete the proof here.

It is sufficient to exhibit a vector field such that any deformation preserving the exponents of its radial orbits comes from a linear change of coordinates. Consider the family
\begin{gather*}X_\alpha= X_0+\left(\sum \alpha_j z_j\right)\sum z_i\del{z_i},\end{gather*}
for
\begin{gather*}X_0=z_1(2z_2-z_3)\del{z_1}+z_2(2z_3-z_1)\del{z_2}+z_3(2z_1-z_2)\del{z_3},\end{gather*}
$\alpha_i\in\mathbb{C}$. Let~$X'$ be given by~$\alpha_i=1$ for all~$i$. It is a nondegenerate vector field. Let us prove that any deformation of~$X'$ within~$V_3$ that preserves its exponents is made of vector fields linearly equivalent to~$X'$. All the vector fields~$X_\alpha$ induce the same foliation~$\mathcal{F}$ on~${\mathbb{CP}}^2$. The arguments given in the proof of Proposition~3.2 in~\cite{guillot-fourier} do prove that any deformation of~$\mathcal{F}$ that preserves its Baum--Bott indexes (like one induced by a deformation of~$X'$ that preserves the exponents) is given by linear changes of coordinates. This implies that if~$X'$ admits a deformation preserving its exponents then this deformation belongs, up to a linear change of coordinates, to the family~$X_\alpha$. For the radial orbits $\rho_1 = [1 : 0 : 0]$, $\rho_2 = [0 : 1 : 0]$ and~$\rho_3 = [0 : 0 : 1]$ of~$X_\alpha$, the sum of the exponents of~$\rho_i$ is~$-1/\alpha_i$. This completes the proof of Theorem~\ref{eigen}.

\begin{Remark} Not all fibers of~$\Psi$ are finite: equation~\ref{linsnetovf} gives a one-parameter family of linearly inequivalent and nondegenerate vector fields having the same exponents.
\end{Remark}

\begin{Question} What is the degree of~$\Psi$? Given a generic set of seven complex numbers whose sum is one, Lins Neto proved that there are~$7!/21=240$ different linear equivalent classes of foliations on~${\mathbb{CP}}^2$ of degree two having them as Baum--Bott indexes~\cite{lins-fibers}. (The condition on the sum is necessary by the theorem of Baum and Bott.) Since two vector fields in~$V_3$ which induce the same foliation and which have the same eigenvalues are linearly equivalent~\cite[Lemma~5]{guillot-pointfixe}, \cite[Corollary~2.7]{guillot-fourier}, \emph{the degree of~$\Psi$ is at most~$240$}.
\end{Question}

A quadratic homogeneous vector field on~$\mathbb{C}^3$ depends upon 18 parameters. The general linear group of~$\mathbb{C}^3$ has nine parameters and acts on this space. Since a generic vector field has no linear symmetries, a quadratic homogeneous vector field depends essentially upon nine parameters. Since each one of the seven radial orbits has two exponents, there are at least five independent relations between them (actually, by Theorem~\ref{eigen}, exactly five). We can give explicitly some of these relations. Let~$X\in V_3$ have an isolated singularity at~$0$ and simple radial orbits~$\rho_1, \ldots, \rho_7$. The radial orbit~$\rho_i$ has two exponents in~$\mathbb{C}$, that will be denoted by~$u_i$ and~$v_i$. They are bound by the relations~$\mathrm{R}_0$, $\mathrm{R}_1$ and~$\mathrm{R}_2$ given respectively by
\begin{gather}\label{systBB}\sum_{i=1}^7\frac{1}{u_iv_i} = 1,\qquad \sum_{i=1}^7\frac{u_i+v_i}{u_iv_i} = 4,\qquad \sum_{i=1}^7\frac{(u_i+v_i)^2}{u_iv_i} = 16.\end{gather}
Relation~$\mathrm{R}_2$ is the Baum--Bott relation for the foliation on~${\mathbb{CP}}^2$ that a vector field in~$V_3$ naturally induces. As explained in~\cite[Example~14]{guillot-pointfixe}, the other two relations may be obtained through some Camacho--Sad and Baum--Bott relations of the foliation that a vector field in~$V_3$ naturally induces on~${\mathbb{CP}}^3$. (See also~\cite[Corollary~12]{guillot-pointfixe} for a unified proof of these relations.) These are all the relations that we know of: we are missing two independent ones. This is one of the main difficulties in the classification of the univalent vector fields on~$\mathbb{C}^3$ with an isolated singularity. Recent results by Kudryashov and Ram\'{\i}rez on a closely related subject~\cite{kud-ram} encourage a moderate optimism towards the possibility of giving these missing relations explicitly (see also~\cite{guillot-ramirez}).

A strategy for classifying nondegenerate semicomplete vector fields in~$V_3$ would be to classify the vector fields having integral exponents and then to try to integrate each one of these vector fields. For the first part, we may start by describing the sets of integers that may be exponents of a nondegenerate semicomplete vector field. The first problem we face is that we do not know the complete system of Diophantine equations we must solve, that we only know the relations~(\ref{systBB}), and that a set of seven couples of integers satisfying this system needs not be the set of exponents of a vector field. Despite this, we may still try to start by solving system~(\ref{systBB}). Let~$\xi_i=u_iv_i$ so that~$\mathrm{R}_0$ reads
\begin{gather}\label{BB0}\sum_{i=1}^7\frac{1}{\xi_i}=1.\end{gather}
A nondegenerate semicomplete quadratic homogeneous vector field on~$\mathbb{C}^3$ gives an integer solution to this ``Egyptian fractions'' problem, and in order to classify nondegenerate quadratic semicomplete vector fields on~$\mathbb{C}^3$, one may start by solving this Diophantine equation. There are infinitely many integer solutions to equation~(\ref{BB0}), which may be grouped in \emph{families}. Families give subvarieties of~$V_3$ where all nondegenerate semicomplete vector fields are to be found. As in~\cite[Section~3.3]{guillot-fourier}, a nondegenerate quadratic homogeneous vector field is said to belong to the \emph{first family} if there is a radial orbit~$\rho_i$ such that~$\xi_i=1$; it is said to belong to the \emph{second family} if there exist~$i$ and~$j$, $i\neq j$ such that~$\xi_i=2$ and~$\xi_j=2$, (so that~$1/\xi_1+1/\xi_2=1$) and, more generally, a solution~$\xi_1,\ldots,\xi_7$ to~(\ref{BB0}) is said to belong to the~$n$\textsuperscript{th} family if~$n$ is the smallest cardinality of the subsets~$J$ of~$\{1,\ldots, 7\}$ such that~$\sum\limits_{i\in J}1/\xi_i=1$. There is no sixth family. The seventh family contains only finitely many solutions of equation~(\ref{BB0}).

\subsubsection{Invariant planes} A vector field~$X$ in~$V_3$ may leave invariant (be tangent to) a plane~$\Pi$ in~$\mathbb{C}^3$. Three (out of the seven) radial orbits of~$X$ will be within this invariant plane. (We will represent invariant planes schematically by diagrams like the one appearing in Fig.~\ref{fig:base_line}.) For each one of these radial orbits, one of the two exponents corresponds to the exponent of the radial orbit of the restriction of~$X$ to~$\Pi$ (we will denote it by~$u_i$). Denoting the other exponent by~$v_i$, we have the relations
\begin{figure}
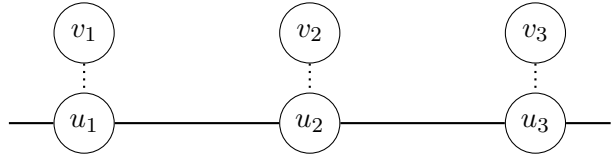
 \centering \oneline{$u_1$}{$u_2$}{$u_3$}{$v_1$}{$v_2$}{$v_3$} \vspace{-2mm} \caption{Diagram representing an invariant plane. The invariant plane contains three radial orbits. For each one, the exponent tangent to the plane is~$u_i$ and the other one~$v_i$.}\label{fig:base_line}\end{figure}
\begin{gather}\label{resv2}\frac{1}{u_1}+\frac{1}{u_2}+\frac{1}{u_3}=1,\\
\label{fakecs}\frac{v_1}{u_1}+\frac{v_2}{u_2}+\frac{v_3}{u_3}=1.\end{gather}
The first one is the relation for quadratic homogeneous vector fields on~$\mathbb{C}^2$. The second one is the Camacho--Sad relation~\cite{CS} for the foliation induced on~${\mathbb{CP}}^2$ with respect to the invariant line induced by~$\Pi$. Notice that the exponents of two of the radial orbits within~$\Pi$ determine the exponents of the third.

In order to classify the nondegenerate quadratic homogeneous vector fields on~$\mathbb{C}^2$ we may list the solutions to the Diophantine equation~(\ref{resv2}). They are
\begin{gather}\label{tess} (2,3,6), \quad (2,4,4),\quad (3,3,3), \quad (1,m,-m),\quad m\in\mathbb{Z}\setminus\{0\}. \end{gather}
Each one of these triples can be realized by a unique (up to linear changes of coordinates) vector field which is, moreover, semicomplete, and which may be explicitly solved (by elliptic functions in the first three cases, by rational ones in the other); see~\cite{briot-bouquet, garnier, ghys-rebelo}, and \cite[Example~2.12]{guillot-fourier}. We will come back to this integration in Section~\ref{im1f}.

Some of the equations that we will consider have two invariant planes. We will represent them by diagrams like the one in Fig.~\ref{fig:dia2lines}. The union of the two invariant planes will contain five radial orbits. In the diagram, the five pairs of circles joined by a dotted line give the Kowalevski exponents of each one of these five radial orbits. The two continuous lines represent the two invariant lines, and the three circles on top of each one of these lines, the three Kowalevski exponents tangent to each one of them. For example, for the diagram in Fig.~\ref{fig:dia2lines} the first invariant plane has the diagram in Fig.~\ref{fig:base_line}, the second one the diagram in Fig.~\ref{fig:secinvplane}; these intersect at the radial orbit with exponents~$(u_2,v_2)$.

\begin{figure}
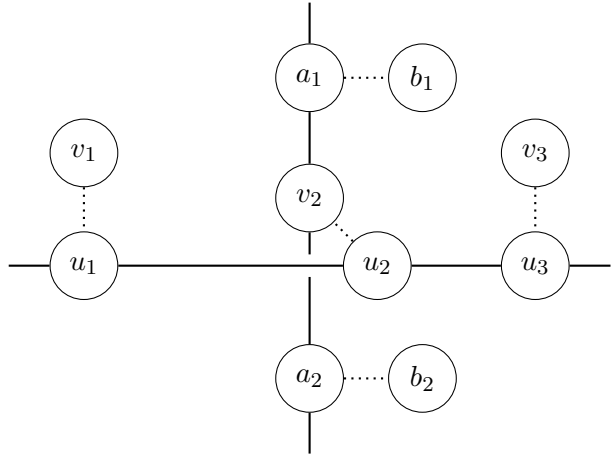
\centering
\tikz{
\tikzstyle{mycir}=[circle,minimum size=0.9cm,draw]
\node[mycir, label=center:$u_1$](u1) at (0,0) {};
\node[mycir, label=center:$u_2$](u2) at (3.9,0) {};
\node[mycir, label=center:$u_3$](u3) at (6,0) {};
\node[mycir, label=center:$v_1$](v1) at (0,1.5) {};
\node[mycir, label=center:$v_2$](v2) at (3,0.9) {};
\node[mycir, label=center:$v_3$](v3) at (6,1.5) {};
\node[mycir, label=center:$a_1$](a1) at (3,2.5) {};
\node[mycir, label=center:$a_2$](a2) at (3,-1.5) {};
\node[mycir, label=center:$b_1$](b1) at (4.5,2.5) {};
\node[mycir, label=center:$b_2$](b2) at (4.5,-1.5) {};
\draw[thick] (u1)--(u2);
\draw[thick] (u2)--(u3);
\draw[thick] (a1)--(v2);
\draw[thick] (v2)--(3,0.15);
\draw[thick] (3,-0.15)--(a2);
\draw[thick] (u1)--++(-1,0);
\draw[thick] (u3)--++(1,0);
\draw[thick] (a1)--++(0,1);
\draw[thick] (a2)--++(0,-1);
\draw[thick, dotted] (v1)--(u1);
\draw[thick, dotted] (v3)--(u3);
\draw[thick, dotted] (a1)--(b1);
\draw[thick, dotted] (a2)--(b2);
\draw[thick, dotted] (u2)--(v2);
}
\vspace{-2mm}
\caption{Diagrams for vector fields having two invariant planes.}\label{fig:dia2lines}
\end{figure}

\begin{figure}
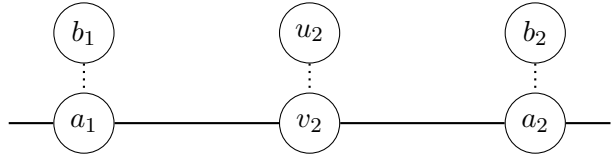
\centering \oneline{$a_1$}{$v_2$}{$a_2$}{$b_1$}{$u_2$}{$b_2$}
\caption{The second invariant plane.}\label{fig:secinvplane}
\end{figure}

Within the space of quadratic homogeneous vector fields on~$\mathbb{C}^3$, those having an invariant plane form a codimension one subvariety. From the discussion in Section~\ref{sec:families} we know that, in restriction to it, the fourteen Kowalevski exponents depend upon eight parameters: there must be at least six independent relations. We have the three relations~(\ref{systBB}) plus the relations~(\ref{resv2}) and~(\ref{fakecs}), that hold in the presence of an invariant line. Thus, even in this restricted case, there are still missing relations. In collaboration with V.~Ram\'{\i}rez, we have recently investigated the missing relations in this setting~\cite{guillot-ramirez}.

\subsection{Vector fields of Halphen type}\label{sec:halphen} The only known examples of semicomplete quadratic homogeneous vector fields on~$\mathbb{C}^3$ that do not have a first integral belong to this class. We have studied their geometry and dynamics in see~\cite[Section~3]{guillot-halphen}, where details on what follows may be found. Let~$E=\sum_i z_i\indel{z_i}$.

\begin{Definition} A quadratic homogeneous vector field~$X$ on~$\mathbb{C}^3$ is said to be \emph{of Halphen type}
if there exists a rational vector field~$C$ on~$\mathbb{C}^3$, homogeneous of degree~$0$, such that~$[C,X]=2E$.
\end{Definition}

The homogeneity of the vector fields involved in this definition implies, by Euler's relation, that~$[E,X]=X$ and~$[E,C]=-C$. The vector fields give thus a representation of the Lie algebra~$\mathfrak{sl}(2,\mathbb{C})$ into the Lie algebra of rational vector fields on~$\mathbb{C}^3$. The prototypical example is given by Halphen's equation~(\ref{halphenclassic}), introduced in~\cite{halphen1}, which is of Halphen type with respect to the vector field~$\sum_i\indel{z_i}$, this is, $\big[\sum_i\indel{z_i},X\big]=2E$.

For a vector field~$X$ of Halphen type with respect to a vector field~$C$, an \emph{adapted function} is a primitive common first integral of~$E$ and~$C$. Recall that the \emph{Schwarzian derivative} of~$f(t)$ with respect to~$t$, $\{f,t\}$ is given by
\begin{gather*}\{f,t\}=\frac{f'''}{f'}-\frac{3}{2}\left(\frac{f''}{f'}\right)^2.\end{gather*}
If~$X$ is a vector field of Halphen type, $\xi$ is an adapted function and~$\xi(t)$ the evaluation of~$\xi$ on a solution of~$X$, we may define the \emph{projective invariant}
\begin{gather}\label{projinv}\{t,\xi\}=-\frac{1}{(\xi')^2}\{\xi,t\},\end{gather}
which is a function of~$\xi(t)$ and is, moreover, independent of the chosen solution of~$X$. For example, for equation~\ref{halpheneq} we have the adapted function~$\xi=(x-y)/(y-z)$, and
\begin{gather}\label{triangle}-\frac{1}{(\xi'(t))^2}\{\xi(t),t\}=\frac{1}{2}\frac{\big(1-\frac{1}{m_1^2}\big)(\xi-1)\xi- \big(1-\frac{1}{m_2^2}\big)(\xi-1)+\big(1-\frac{1}{m_3^2}\big)\xi}{\xi^2(\xi-1)^2}, \end{gather}
for~$m_i=(\alpha_1+\alpha_2+\alpha_3-2)/\alpha_i$. We have given a criterion for univalence of vector fields of Halphen type in terms of the projective invariant~(\ref{projinv}); see~\cite[Proposition~4]{guillot-halphen}. For systems where it is given by~(\ref{triangle}), a sufficient condition for semicompleteness is that~$m_i\in\mathbb{Z}$, $m_i^2\neq 1$. In this case, a key fact is that, when restricted to the upper half plane, $t(\xi)$ uniformizes a triangle bounded by circular arcs having internal angles~$\pi/m_i$~\cite[Section~10.2]{hille}.

In the hyperbolic case ($\sum 1/m_i<1$), undoubtedly the most interesting one, the triangle can be seen as a geodesic triangle in the hyperbolic plane which, moreover, tessellates it. The associated equation of Halphen type is related to the group of direct symmetries of this tessellation, the \emph{triangle group}~$T(m_1,m_2,m_3)$. The domains where the solutions of the equation are defined have a circular natural boundary and the equations do not have a meromorphic first integral (see~\cite[Theorem~A]{guillot-halphen} or~\cite[Corollary~D]{guillot-rebelo} for different proofs of this last fact).

There are semicomplete quasihomogeneous vector fields of Halphen type in~$\mathbb{C}^3$ which are not associated to hyperbolic geometry and whose solutions are defined in domains with a fractal boundary~\cite{guillot-genhalphen} or whose complement is a Cantor set~\cite[Section~4.11]{guillot-lectures}. We do not know of any algebraic autonomous differential equation in dimension three having a univalent solution that has a natural boundary and which is not of Halphen type.

\subsection{Chazy and the homogeneous equations of the third order}\label{sec:chazy}

A classification closely related to the one that concerns us lies at the beginning of Chazy's effort to parallel Painlev\'e's studies on second-order equations for equations of the third order~\cite{chazy}. As one of the first steps in this program, Chazy sought to establish the equations of the form
\begin{gather}\label{eqchazy}\phi'''=a\phi^4+b\phi^2\phi'+c(\phi')^2+d\phi\phi'',\end{gather}
with~$a,b,c,d\in\mathbb{C}$, which have only single-valued solutions. The vector field
\begin{gather*}X=y\del{x}+z\del{y}+\big(ax^4+bx^2y+cy^2+dxz\big)\del{z},\end{gather*}
equivalent to~(\ref{eqchazy}), is quasihomogeneous when~$x$, $y$ and~$z$ are given, respectively, the weights~$1$,~$2$ and~$3$: for~$L=x\indel{x}+2y\indel{y}+3z\indel{z}$, $[L,X]=X$. Our problem is thus analogous to Chazy's.

Equation~(\ref{eqchazy}) has three \emph{similarity} solutions (the natural analogue of radial orbits, orbits where~$X$ and~$L$ are collinear). From each one of these solutions may extract one pair of exponents, which must satisfy the relations
\begin{gather*}\sum_{i=1}^3 \frac{1}{u_iv_i}=\frac{1}{6}, \qquad \sum_{i=1}^3 \frac{u_i+v_i}{u_iv_i}=-\frac{7}{6}, \qquad \sum_{i=1}^3 \frac{(u_i+v_i)^2}{u_iv_i}=\frac{49}{6}.\end{gather*}
A differential equation~(\ref{eqchazy}) depends essentially upon three parameters and thus, in this case, we know all the relations between the exponents. In order to classify the semicomplete equations in family~(\ref{eqchazy}), we may first solve this system of Diophantine equations, and then, for each solution, study each one of the differential equations it determines, in order to assess if all of its solutions are single-valued. It is essentially in this way that Chazy obtained some families of equations (some of them featuring a parameter) usually bearing the names Chazy~I through~Chazy~XII. Chazy integrated most of these equations except those labeled~IX and~X, for which we had to wait almost 90 years until Cosgrove achieved their integration~\cite{cosgrove-chazy}. Let us briefly describe the geometry of those which will be related to our classification (see~\cite{guillot-chazy} for details).

\paragraph*{Chazy~IX.} The equation is
\begin{gather}\label{chazyix}\phi'''=54\phi^4+72\phi^2\phi'+12(\phi')^2.\end{gather}
When considered as a polynomial vector field on~$\mathbb{C}^3$, it has a polynomial quasihomogeneous first integral of degree~$10$ and commutes with a quasihomogeneous polynomial vector field of degree~$4$ having the same first integral. A generic level surface~$S$ of this first integral is an affine surface endowed with two commuting vector fields. Furthermore, by the quasihomogeneity of the first integral, $S$ has a natural action of~$\mathbb{Z}/10\mathbb{Z}$ preserving the restriction of the vector field up to a constant factor. On the other hand, consider the genus-two curve $x^2-y^5-1$. It has an automorphism of order~$10$ given by~$(x,y)\mapsto (-x,\omega y)$ for~$\omega$ a fifth root of unity. This induces an automorphism of order~$10$ on the Jacobian~$J$ of the curve. This automorphism preserves two constant vector fields up to a constant factor. The generic level surface~$S$ embeds into~$J$, by mapping the vector field of the Chazy~IX equation to one of these two constant vector fields (and the one that commutes with it to the other one). This mapping is equivariant with respect to the corresponding actions of~$\mathbb{Z}/10\mathbb{Z}$.

\paragraph*{Chazy X.} Given by the two equations
\begin{gather}\label{chazyx}\phi'''=6\phi^2\phi'+\frac{3}{11}\big(9+7\sqrt{3}\big)\big(\phi^2+\phi'\big)^2,\end{gather}
one for each determination of~$\sqrt{3}$. It was integrated by Cosgrove through an ``unexpectedly complicated'' formula~\cite{cosgrove-chazy}. We may describe the accompanying geometry as follows. When seen as a quasihomogeneous polynomial vector field on~$\mathbb{C}^3$, it has a quasihomogeneous polynomial first integral of degree 12 and commutes with a quasihomogeneous polynomial vector field of degree~$6$ having the same first integral, so that the generic level surface~$S$ is endowed with two commuting vector fields and an action of~$\mathbb{Z}/12\mathbb{Z}$ that preserves each one of them up to a constant factor. On the other hand, consider the elliptic curve~$E_{\rm i}=\mathbb{C}/\Lambda$, for~$\Lambda=\langle 1,{\rm i}\rangle$. It has an order four automorphism induced by multiplication by~${\rm i}$. The square~$E_{\rm i}\times E_{\rm i}$ has an automorphism of order~$12$, given by the order four automorphism of~$E_{\rm i}$ acting diagonally and by the order-three map~$(z,w)\mapsto (w,-z-w)$, which commutes with the previous one. The cyclic group generated by these symmetries preserves two vector fields up to a constant factor each. The generic level surface~$S$ embeds equivariantly into~$E_{\rm i}\times E_{\rm i}$ by mapping the vector field of the Chazy~X equation into one of these two vector fields.

\paragraph*{Chazy XI.} The one-parameter family of equations
\begin{gather*}\phi'''=\frac{1-p^2}{2}\phi\phi''+\left(\frac{1-p^2}{2}+6\right) (\phi')^2-3\big(1-p^2\big)\phi^2\phi'+\frac{3\big(1-p^2\big)^2}{8}\phi^4.\end{gather*}
These can be reduced to Riccati equations with elliptic coefficients. We have the first integral~$g_3=4z^3-(z')^2$ for~$z=\phi'-\frac{1}{4}\big(1-p^2\big)\phi^2$ and thus~$\frac{1}{4}\big(1-p^2\big)\phi$ is a solution to the Riccati equation~$f'=f^2+\frac{1}{4}\big(1-p^2\big)\wp$, the Riccati form of Lam\'e's equation, for the Weierstrass function~$\wp$ such that~$(\wp')^2=4\wp^3-g_3$. The equation is univalent if~$p$ is not a multiple of~$6$.

\paragraph*{Chazy XII.} The one-parameter family
\begin{gather*}\phi'''=2\phi\phi''-3(\phi')^2+\frac{4}{36-k^2}\big(6\phi'-\phi^2\big)^2,\end{gather*}
probably the most famous of Chazy's equations; it is of Halphen type. It corresponds to the vector field
\begin{gather*}X=y\del{x}+z\del{y}+\left[2xz-3y^2+\frac{4}{36-k^2}\big(6y-x^2\big)^2\right]\del{z},\end{gather*}
which is quasihomogeneous with respect to the vector field~$L=x\indel{x}+2y\indel{y}+3z\indel{z}$, $[L,X]=-X$. For~$C=3\indel{x}+x\indel{y}+3y\indel{z}$, $[L,C]=-C$ and~$[C,X]=L$, so~$X$ is indeed of Halphen type. For the adapted function \begin{gather*}\xi=\left(\frac{36-k^2}{k^2}\right)\frac{\big(x^3-9xy+9z\big)^2}{\big(6y-x^2\big)^3},\end{gather*} we have~(\ref{triangle}) for~$m_1=3$, $m_2=2$ and~$m_3=k$. The equation is semicomplete if~$k\in\mathbb{Z}$, $k\geq 2$.

\subsubsection{The Chazy equations as quotients of the quadratic symmetric ones} \label{chazy-sym} If~$X$ is a quadratic homogeneous vector field on~$\mathbb{C}^3$ which is invariant under the group~$S_3$ of permutations of the three variables, for a solution~$(z_1(t),z_2(t),z_3(t))$ of~$X$, $\phi(t)=z_1(t)+z_2(t)+z_3(t)$ is a solution to an equation of the form~(\ref{eqchazy}) which depends only on~$X$. For instance, the equations in the Chazy~XII family are the quotients of the symmetric instances of equations~\ref{halpheneq} (Halphen's equations). The quotient of a univalent quadratic vector field will be a univalent equation of the form~(\ref{eqchazy}).

\subsection{Lins Neto's vector fields}\label{sec:linsnetovf}
They are those of equation~\ref{linsnetovf}. We introduced them in~\cite{guillot-lins}, where their integration was studied. The foliations that these vector fields induce on~${\mathbb{CP}}^2$ give one of the remarkable families of foliations that Lins Neto constructed in relation to the Poincar\'e problem~\cite{linsneto-pp}.

For~$\rho^3=1$, let~$\Lambda=\langle 1,\rho\rangle$ and let~$E_\rho$ be the elliptic curve~$\mathbb{C}/\Lambda$. It has an automorphism of order six induced by multiplication by~$-\rho$ on~$\mathbb{C}$. The Abelian variety~$E_\rho\times E_\rho$ has a cyclic automorphism of order six given by the diagonal action of the latter. Every constant vector field on~$\mathbb{C}^2$ induces a vector field on~$E_\rho\times E_\rho$ which is preserved, up to a constant factor, by this cyclic automorphism. All the vector fields of equations~\ref{linsnetovf} commute. They have a common polynomial homogeneous first integral of degree six. The generic level surface of this first integral can be compactified into~$E_\rho\times E_\rho$ by an embedding that maps every vector field in the family into a constant one. See~\cite{guillot-lins} for details.

\section{First family}

We now begin the classification of the nondegenerate vector fields belonging to the first family, which is characterized by the existence of a radial orbit whose exponents are either~$(1,1)$ or~$(-1,-1)$. A rough classification of the nondegenerate semicomplete vector fields within the first family was carried out in~\cite[Proposition~3.8]{guillot-fourier}.

\subsection[Exponents~$(-1,-1)$: Halphen's equations]{Exponents~$\boldsymbol{(-1,-1)}$: Halphen's equations} If a nondegenerate vector field in~$V_3$ has a radial orbit at~$[1:1:1]$ with exponents~$(-1,-1)$ and radial orbits at~$[1:0:0]$, $[0:1:0]$ and~$[0:0:1]$, it belongs to the family~\ref{halpheneq} (briefly discussed in Section~\ref{sec:halphen}); see~\cite[Proposition~3.8]{guillot-fourier}.

\subsection[Exponents~$(1,1)$: reduction to two-dimensional equations]{Exponents~$\boldsymbol{(1,1)}$: reduction to two-dimensional equations}\label{im1f}

A nondegenerate vector field in~$V_3$ with a radial orbit at~$[0:0:1]$ having exponents~$(1,1)$ is of the form
\begin{gather}\label{im2f} P(x,y)\del{x}+Q(x,y)\del{y}+\big[z^2+2z\ell(x,y)+R(x,y)\big]\del{z},\end{gather}
for~$\ell$ a linear homogeneous function and~$R$ a quadratic homogeneous one (upon replacing~$z$ by~$z+\ell(x,y)$ we may suppose that~$\ell\equiv0$). The linear projection onto~$(x,y)$ gives a quadratic homogeneous vector field, which must be semicomplete and nondegenerate. According to Section~2 in~\cite{ghys-rebelo} (see also~\cite{briot-bouquet}), this reduced equation is, up to a linear change of coordinates, one of the following ones:
\begin{enumerate}\itemsep=0pt
 \item \label{q333} $x(x-2y)\indel{x}+y(y-2x)\indel{y}$,
 \item\label{q442} $x(x-3y)\indel{x}+y(y-3x)\indel{y}$,
 \item\label{q236} $x(2x-5y)\indel{x}+y(y-4x)\indel{y}$,
 \item\label{qn} $x^2\indel{x}+y[m y-(m-1)x]\indel{y}$, $m\in\mathbb{Z}\setminus\{0\}$.
\end{enumerate}
We must, in each case, determine the polynomials~$R$ that make the vector field~(\ref{im2f}) semicomplete.

\subsubsection{Case~(\ref{q333}): equation~\ref{fq333}} By setting
\begin{gather*}R=\frac{1}{4}\big(1-p^2\big)x^2+\frac{1}{4}\big(p^2+q^2-1 -r^2\big)xy+\frac{1}{4}\big(1-q^2\big)y^2,\end{gather*}
with~$p,q,r\in\mathbb{C}$, \looseness=-1 we get the most general quadratic polynomial of the form under consideration. The exponents of the radial orbits are~$(1,1)$, $(3,p)$, $(3,-p)$, $(3,q)$, $(3,-q)$, $(3,r)$, $(3,-r)$, so we must have that~$p,q,r\in\mathbb{Z}$. The involution~$(x,y,z)\mapsto (y,x,z)$ exchanges the roles of~$p$ and~$q$, while the involution~$(x,y,z)\mapsto (x-y,-y,z)$ exchanges those of~$q$ and~$r$. (The roles played by $p$, $q$ and~$r$ are thus completely symmetric.) The reduced equation~(\ref{q333}) has the first integral~$c=xy(x-y)$, which is also a first integral of the full system. By the homogeneity of the equation we need only consider one nonzero level curve. A solution to the reduced equation is
\begin{gather*}x=\frac{\sqrt{-g_3}}{\wp},\qquad y=\frac{1}{2}\left(\frac{\wp'}{\wp}+\frac{\sqrt{-g_3}}{\wp}\right),\end{gather*} for the Weierstrass elliptic function~$\wp$ such that~$(\wp')^2=4\wp^3-g_3$ (see~\cite[Chapter~6]{lawden} for general facts about these). Substituting in equation~(\ref{im2f}), we obtain the Riccati equation with elliptic coefficients
\begin{gather*}z'=z^2-\frac{g_3}{16}\big(3-2p^2-2r^2+q^2\big)\frac{1}{\wp^2} +\frac{\sqrt{-g_3}}{8}\big(p^2-r^2\big)\frac{\wp'}{\wp^2}+\frac{1}{16}\big(1-q^2\big)\left(\frac{\wp'}{\wp}\right)^2.\end{gather*}
This is equation~$\mathrm{II}$ in~\cite{guillot-riccati}, where we have established that the equation is semicomplete whenever
\begin{itemize}\itemsep=0pt \item neither~$p$ nor~$q$ nor~$r$ are multiples of~$3$, or when
\item $3\mid p$ but~$6\nmid p$, $q^2=r^2$, $3\nmid q$, $3\nmid r$ (and the analogous conditions obtained by permu\-ting~$p$,~$q$ and~$r$).
\end{itemize}

\subsubsection{Case~(\ref{q442}): equation~\ref{fq442}} The change of coordinates~$(u,v)=(x+y,x-y)$ maps the reduced vector field to
$\big(2v^2-u^2\big)\indel{u}+uv\indel{v}$, which has the first integral~$c=v^2\big(v^2-u^2\big)$. Let
\begin{gather*}R=\frac{1}{4}\big(1-p^2\big)u^2+\frac{1}{8}\big(q^2-r^2\big)uv+ \frac{1}{8}\big(2p^2-q^2-r^2\big)v^2.\end{gather*}
The exponents of the radial orbits are~$(1,1)$, $(2,p)$, $(2,-p)$, $(4,q)$, $(4,-q)$, $(4,r)$, and~$(4,-r)$. The parameters~$q$ and~$r$ play symmetric roles, for the involution~$(u,v,z)\mapsto(u,-v,z)$ exchanges them while preserving~$p$. A solution to the reduced equation for~$c=1$ is given by~$v={\rm i}\,\mathsf{sn}(t)$, $u=v'/v$, where~$\mathsf{sn}(t)$ is Jacobi's \emph{sinus amplitudinis} elliptic function of modulus~$k^2=-1$, the one such that~$(\mathsf{sn}'(t))^2=1-\mathsf{sn}^4(t)$ (see~\cite[Chapter~2]{lawden} for details about these functions). Substituting, we obtain the Riccati equation with elliptic coefficients
\begin{gather*}w'=w^2+\frac{1}{8}\big(q^2+r^2-2\big)\mathsf{sn}^2(t)+\frac{{\rm i}}{8}\big(q^2-r^2\big)\mathsf{sn}'(t)+ \frac{1}{4}\big(1-p^2\big)\mathsf{sn}^{-2}(t).\end{gather*}
This is equation~$\mathrm{V}$ in~\cite{guillot-riccati}, where we have established that the equation is semicomplete whenever~$q$ and~$r$ are odd and either
\begin{itemize}\itemsep=0pt \item $p$ is odd or
 \item $p$ is even but not a multiple of~$4$ and~$q^2=r^2$.
\end{itemize}

\subsubsection{Case~(\ref{q236}): equation~\ref{fq236}} Set
\begin{gather*}R=\big(1-q^2\big)x^2+\frac{1}{4}\big(4+4q^2+r^2-9p^2\big)xy+\frac{1}{4}\big(1-r^2\big)y^2.\end{gather*}
The exponents of the radial orbits are~$(1,1)$, $(2,p)$, $(2,-p)$, $(3,q)$, $(3,-q)$, $(6,r)$ and~$(6,-r)$. The reduced equation has the first integral~$g_3=-4xy^2(x-y)^3$ and the solutions
\begin{gather*}x=\frac{1}{2}\frac{g_3}{\wp\wp'}, \qquad y=2\frac{\wp^2}{\wp'},\end{gather*}
for the Weierstrass elliptic function~$\wp$ such that~$(\wp')^2=4\wp^3-g_3$. Substituting, we obtain the Riccati equation
\begin{gather*}z'=z^2+\frac{1}{(\wp')^2}\left[\frac{g_3^2 }{4}\big(1-q^2\big) \wp^{-2}+\frac{g_3}{4}\big(4+4q^2+r^2-9p^2\big)\wp +\big(1-r^2\big)\wp^4\right].\end{gather*}
This is equation~$\mathrm{III}$ in~\cite{guillot-riccati}. It is semicomplete whenever~$6 \nmid r$, $3\nmid q$ and~$2\nmid p$.

\subsubsection{Case~(\ref{qn}): equation~\ref{fq-1m-m}}\label{1ricrat} Setting
\begin{gather*}R=\frac{1}{4}\big[\big(1-p^2\big)x^2- \big[r^2-p^2+m^2\big(1-q^2\big)\big]xy+m^2\big(1-q^2\big)y^2\big],\end{gather*}
the exponents of the radial orbits are found to be~$(1,1)$, $(1,q)$, $(1,-q)$, $(m,p)$, $(m,-p)$, $(-m,r)$, and~$(-m,-r)$. The reduced equation has the first integral~$x^my/(x-y)$ and the solution
\begin{gather*}x=-\frac{1}{t}, \qquad y=-\frac{t^{m-1}}{t^m-1},\end{gather*}
corresponding to a nonradial level curve. Substituting these in the original equation, we obtain the Riccati equation with rational coefficients
\begin{gather*}z'=z^2+\frac{1}{4t^2}\left[m^2\big(1-q^2\big)\left(\frac{t^{m}}{t^{m}-1}\right)^2-\big[r^2-p^2+m^2\big(1-q^2\big)\big] \left(\frac{t^{m}}{t^{m}-1}\right)+\big(1-p^2\big) \right].\end{gather*}
It will be integrated in Section~\ref{sec:ricrac1}. The results there obtained show that this Riccati equation will have exclusively single-valued solutions if~$m\nmid p$, $m\nmid r$ and if, for some~$j\in\{0,\ldots,q-1\}$, either~$p-r=(q-1-2j)m$ or~$p+r=(q-1-2j)m$. Moreover, in these cases, all the solutions will be rational ones.

\section{Second family}
This second family is characterized by the existence of two radial orbits such that the exponents of each one of them are either~$(1,2)$ or~$(-1,-2)$. We will refer to these two radial orbits as the \emph{distinguished} ones. We begin by establishing two lemmas.

\begin{Lemma}\label{sing12} In a degree two foliation on~${\mathbb{CP}}^2$ a singularity with exponents~$[1:2]$ is linearizable if and only if the line tangent to the exponent~$1$ is invariant.\end{Lemma}
\begin{proof} The most general foliation of degree two on~${\mathbb{CP}}^2$ having a nondegenerate singularity at $(0,0)$, whose linearization has exponents~$[1:2]$ and such that the coordinate axes are tangent to the eigenspaces of the linear part is given by
\begin{gather*}\big(2y+a_1x^2+a_2xy+a_3y^2+y\big[b_1x^2+b_2xy+b_3y^2\big]\big){\rm d}x \\
\qquad{}-\big(x+a_4x^2+a_5xy+a_6y^2+x\big[b_1x^2+b_2xy+b_3y^2\big]\big){\rm d}y.
\end{gather*}
The term~$a_1x^2$ is a resonant one and we must have~$a_1=0$ if the singularity is linearizable. In this case~$y=0$ is an invariant line.
\end{proof}

\begin{Lemma}\label{notwopos} In a nondegenerate semicomplete vector field in~$V_3$, if the exponents of two radial orbits are both~$(-1,-2)$, there exists a third radial orbit whose exponents are either~$(1,2)$ or~$(1,1)$.
\end{Lemma}
\begin{proof} Under the hypothesis, there are seven nondegenerate radial orbits, whose exponents, $(u_i,v_i)$, are nonzero integers. Relation $\mathrm{R}_1$ reads
\begin{gather*}-\frac{1}{1}-\frac{1}{2}-\frac{1}{1}-\frac{1}{2}+\sum_{i=3}^7
\left(\frac{1}{u_i} +\frac{1}{v_i}\right)=4,\end{gather*}
and thus $\sum\limits_{i=3}^7 (1/u_i+1/v_i) =7$. If among the ten numbers~$u_3,\ldots, u_7,v_3,\ldots,v_7$ we have~$n_1$ times~$1$ and~$n_2$ times~$2$ then, since for the remaining ones~$u_i\geq 3$ if~$u_i>0$ and since the sum of the positive terms in~$\sum\limits_{i=3}^7 (1/u_i+1/v_i)$ is at least~$7$,
\begin{gather*}\frac{n_1}{1}+\frac{n_2}{2}+\frac{10-(n_1+n_2)}{3}\geq 7,\end{gather*}
and thus $4n_1+n_2\geq 22$. Since~$n_1+n_2\leq 10$, $n_1\geq 4$. Suppose that~$(u_i,v_i)\neq (1,1)$ for all~$i$. We have that~$n_1<6$ (for otherwise one radial orbit will have its two exponents equal to~$1$). If~$n_1=5$ then~$n_2\geq 2$, and the~$n_1$ exponents~$1$ must belong to different radial orbits: one of them must pair with one of the~$2$'s. If~$n_1=4$ then~$n_2=6$ and a similar reasoning applies.
\end{proof}

Let us now begin the classification in the second family. Cases appear according to the signs of the exponents of the two distinguished radial orbits and to the configuration of the invariant planes guaranteed by Lemma~\ref{sing12}.

\subsection[The plane associated to one of the distinguished radial orbits does not contain the other radial orbit and vice-versa]{The plane associated to one of the distinguished radial orbits\\ does not contain the other radial orbit and vice-versa}\label{2fnc}

We will now study the cases when, for one of the two distinguished radial orbits, the invariant plane given by Lemma~\ref{sing12} does not contain the other radial orbit and vice-versa (in particular, these invariant planes are different).

\begin{figure}
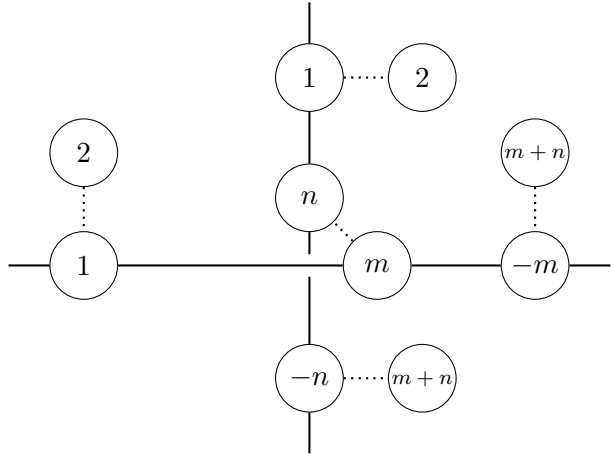
\centering
\tikz{
\tikzstyle{mycir}=[circle,minimum size=0.9cm,draw]
\node[mycir, label=center:$1$](u1) at (0,0) {};
\node[mycir, label=center:$m$](u2) at (3.9,0) {};
\node[mycir, label=center:$-m$](u3) at (6,0) {};
\node[mycir, label=center:$2$](v1) at (0,1.5) {};
\node[mycir, label=center:$n$](v2) at (3,0.9) {};
\node[mycir, label=center:\scriptsize{$m+n$}](v3) at (6,1.5) {};
\node[mycir, label=center:$1$](a1) at (3,2.5) {};
\node[mycir, label=center:$-n$](a2) at (3,-1.5) {};
\node[mycir, label=center:$2$](b1) at (4.5,2.5) {};
\node[mycir, label=center:\scriptsize{$m+n$}](b2) at (4.5,-1.5) {};
\draw[thick] (u1)--(u2);
\draw[thick] (u2)--(u3);
\draw[thick] (a1)--(v2);
\draw[thick] (v2)--(3,0.15);
\draw[thick] (3,-0.15)--(a2);
\draw[thick] (u1)--++(-1,0);
\draw[thick] (u3)--++(1,0);
\draw[thick] (a1)--++(0,1);
\draw[thick] (a2)--++(0,-1);
\draw[thick, dotted] (v1)--(u1);
\draw[thick, dotted] (v3)--(u3);
\draw[thick, dotted] (a1)--(b1);
\draw[thick, dotted] (a2)--(b2);
\draw[thick, dotted] (u2)--(v2);
}
\caption{The configuration of Section~\ref{2neg2lines}.}\label{fig:222}
\end{figure}

\subsubsection{Positive exponents for both distinguished radial orbits} \label{2neg2lines}

We have the diagram in Fig.~\ref{fig:222} with $m$, $n\in\mathbb{Z}$. By taking~$\{x=0\}$ and $\{y=0\}$ as invariant planes and fixing the two radial orbits with~$\xi_i=2$ at~$[0:1:0]$ and $[1:0:0]$, we find the vector field
\begin{gather}\label{genrtwod}x\left(x-y+[1-n]z\right)\del{x}
+y\left(y-x+[1-m]z\right)\del{y} +\left(z^2+\alpha xy\right)\del{z}, \end{gather}
for some~$\alpha\in\mathbb{C}$. There are two further radial orbits whose exponents are, respectively, $(p,q)$ and~$(p,-q)$, for~$p$ and~$q$ given by
\begin{gather*}\frac{1}{\kappa}+\frac{1}{p}=\frac{1}{2},\qquad q^2=\frac{\alpha (m-n)^2-8\kappa}{\alpha(2-\kappa)^2},\end{gather*}
for~$\kappa=m+n$. In particular,
\begin{gather}\label{whichkappa}(\kappa,p)\in\{(1,-2), (-2,1), (4,4), (3,6), (6,3)\}.\end{gather}
These will give, respectively, equations~\ref{ricell-1}, \ref{ricell+2}, \ref{ricell-4}, \ref{ricell-3} and~\ref{ricell-6}. The image of~(\ref{genrtwod}) under~$(x,y,z)\mapsto (xy,z)=(w,z)$ is
\begin{gather*}(2-\kappa)wz\del{w}+\big(z^2+\alpha w\big)\del{z}.\end{gather*}
In particular, we have the second-order equation
\begin{gather}\label{lydg2}z''=(4-\kappa) zz'-(2-\kappa) z^3,\end{gather}
which is independent of~$\alpha$. From the solutions to these equations, the solutions to the original ones may be obtained by solving the (independent) Riccati equations
\begin{gather} \nonumber x' = x^2+(1-n)z(t)x-w(t),\\ \label{dosrics2} y' = y^2+(1-m)z(t)y-w(t). \end{gather}
For the first one, setting~$v=x+\frac{1}{2}(1-n)z(t)$ we have~$v'=v^2+A$, for
\begin{gather}\label{riccA} A= \frac{1}{4}\big(1-n^2\big)z^2+\left(\frac{1}{2}\alpha[1-n]-1\right)w;
\end{gather} there is an analogous equation for~(\ref{dosrics2}). Let us integrate the equations corresponding to the values of~$\kappa$ in~(\ref{whichkappa}).
\paragraph{Equation~\ref{ricell-3}, $\kappa=3$.} A solution to~(\ref{lydg2}) is~$z=-\wp'/\wp$ for~$\wp$ the Weierstrass elliptic function~$\wp$ such that~$(\wp')^2=4\wp^3-g_3$. For~(\ref{riccA}),
\begin{gather*}A=\frac{1}{4}\big(1-q^2\big)\wp-\frac{1}{4}\big(1-n^2\big)\frac{g_3}{\wp^2}.\end{gather*}
The corresponding Riccati equation is equation~$\mathrm{I}$ in~\cite{guillot-riccati}. It is semicomplete if~$6\nmid q$ and~$3\nmid n$. Since~$n+m=3$, condition~$3\nmid n$ implies that~$3\nmid m$ and thus equation~(\ref{dosrics2}) is also semicomplete under these conditions.
\paragraph{Equation~\ref{ricell-6}, $\boldsymbol{\kappa=6}$.} A solution to~(\ref{lydg2}) is~$z=\frac{1}{2}\wp'/\wp$, $(\wp')^2=4\wp^3-g_3$, and thus, for~(\ref{riccA}),
\begin{gather*}A=\frac{1}{4}\big(1-n^2\big)\wp-\frac{1}{4}\big(1-q^2\big)\frac{g_3}{\wp^2}.\end{gather*}
The corresponding Riccati equation is again equation~$\mathrm{I}$ in~\cite{guillot-riccati}. It is semicomplete if~$6\nmid n$ and~$3\nmid q$. Since~$n+m=6$, condition~$6\nmid n$ implies that~$6\nmid m$ and thus equation~(\ref{dosrics2}) is also semicomplete under these conditions.
\paragraph{Equation~\ref{ricell-4}, $\boldsymbol{\kappa=4}$.} A solution to~(\ref{lydg2}) is~$z={\rm i} \mathsf{sn}( t)$, for the Jacobi elliptic function~$\mathsf{sn}(t)$ of modulus~$k^2=-1$, the one such that~$(\mathsf{sn}'(t))^2=1-\mathsf{sn}^4(t)$. Accordingly, in~(\ref{riccA}),
\begin{gather*} A = \frac{1}{8}\big[\big(n^2+q^2-2\big)\mathsf{sn}^2( t)+\big(q^2-n^2\big){\rm i}\,\mathsf{sn}'( t)\big] \\
\hphantom{A}{} =\frac{1}{16}\big(1-n^2\big)(\mathsf{cn}(t)+{\rm i}\,\mathsf{dn}(t))^2 +\frac{1}{16}\big(1-q^2\big)(\mathsf{cn}(t)-{\rm i}\,\mathsf{dn}(t))^2,\end{gather*}
for the other Jacobi elliptic functions~$\mathsf{cn}(t)$ and~$\mathsf{dn}(t)$ of the same modulus. This is equation~$\mathrm{IV}$ in~\cite{guillot-riccati}. It is semicomplete in all the cases where both~$n$ and~$q$ are odd and in the case where~$4\nmid q$, $4\nmid n$ and~$q^2=n^2$. Likewise, equation~(\ref{dosrics2}) will be semicomplete in all the cases where both~$m$ and~$q$ are odd and in the case where~$4\nmid q$, $4\nmid m$ and~$q^2=m^2$. Thus, equation~\ref{ricell-4} is semicomplete in the cases where~$n$, $m$ and~$q$ are all odd and in the special case~$m=2$, $n=2$ and~$q=\pm 2$.

\paragraph{Equation~\ref{ricell-1}, $\boldsymbol{\kappa=1}$.} For any quadratic polynomial~$u$, $z=-u'/u$ is a solution to~(\ref{lydg2}) and hence, for~(\ref{riccA}),
\begin{gather}\label{tojs}A=\frac{1}{4}\big(1-n^2\big)\left[\frac{(u')^2-2uu''}{u^2}\right]+\frac{1}{8}\big(1-q^2\big)\frac{u''}{u}.\end{gather}
If~$n=-1$ then~$(-n,m+n)=(1,1)$, bringing us to back to Section~\ref{im1f}; if~$n=1$ then~$m=0$, which does not correspond to a nondegenerate vector field. We may thus suppose that~$n^2\neq 1$. This equation will be studied in Section~\ref{sec:ricrac2}, where we will see that it has exclusively single-valued solutions (actually, rational ones) if~$q$ is odd and~$q<2n$.

\paragraph{Equation~\ref{ricell+2}, $\boldsymbol{\kappa=-2}$.} For a quadratic polynomial~$u$, $z=-\frac{1}{2}u'/u$ is a solution to~(\ref{lydg2}). For~(\ref{riccA}),
\begin{gather*}A=\frac{1}{4}\big(1-q^2\big)\left[\frac{(u')^2-2uu''}{u^2}\right]+\frac{1}{8}\big(1-n^2\big)\frac{u''}{u}.\end{gather*}
Since~$p=1$, $q^2=1$ would imply the existence of a radial orbit with exponents~$(1,1)$, placing us in the setting of Section~\ref{im1f}, so we may suppose that~$q^2\neq 1$. The Riccati equation we obtain is the very same one associated to~(\ref{tojs}), that will be analyzed in Section~\ref{sec:ricrac2}. From the analysis there done, the Riccati equation has only single-valued solutions (actually, rational ones) whenever~$n$ is odd and~$n<2q$.

\begin{figure}
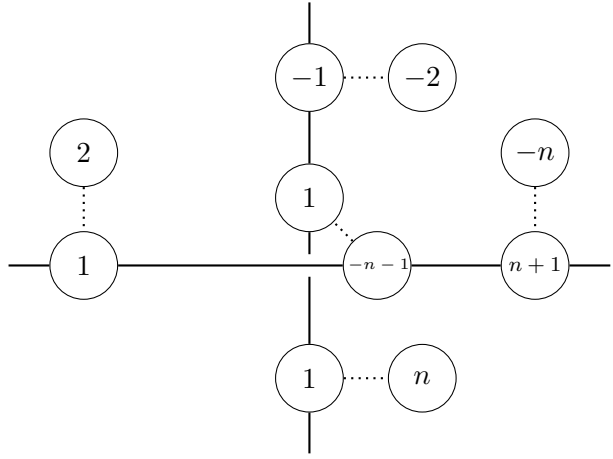
\centering
\tikz{
\tikzstyle{mycir}=[circle,minimum size=0.9cm,draw]
\node[mycir, label=center:$1$](u1) at (0,0) {};
\node[mycir, label=center:\tiny{\hspace{0.5ex}$-n-1$}](u2) at (3.9,0) {};
\node[mycir, label=center:\scriptsize{$n+1$}](u3) at (6,0) {};
\node[mycir, label=center:$2$](v1) at (0,1.5) {};
\node[mycir, label=center:$1$](v2) at (3,0.9) {};
\node[mycir, label=center:$-n$](v3) at (6,1.5) {};
\node[mycir, label=center:$-1$](a1) at (3,2.5) {};
\node[mycir, label=center:$1$](a2) at (3,-1.5) {};
\node[mycir, label=center:$-2$](b1) at (4.5,2.5) {};
\node[mycir, label=center:$n$](b2) at (4.5,-1.5) {};
\draw[thick] (u1)--(u2);
\draw[thick] (u2)--(u3);
\draw[thick] (a1)--(v2);
\draw[thick] (v2)--(3,0.15);
\draw[thick] (3,-0.15)--(a2);
\draw[thick] (u1)--++(-1,0);
\draw[thick] (u3)--++(1,0);
\draw[thick] (a1)--++(0,1);
\draw[thick] (a2)--++(0,-1);
\draw[thick, dotted] (v1)--(u1);
\draw[thick, dotted] (v3)--(u3);
\draw[thick, dotted] (a1)--(b1);
\draw[thick, dotted] (a2)--(b2);
\draw[thick, dotted] (u2)--(v2);
}
\caption{The configuration of Section~\ref{gendiffsng}.}\label{fig:221}
\end{figure}

\subsubsection{The exponents of the two distinguished radial orbits have different signs}\label{gendiffsng} We have, for some~$n\in\mathbb{Z}$, the diagram in Fig.~\ref{fig:221}. If~$n=-1$ the vector field is degenerate (for~$n+1=0$) and if~$n=1$ it corresponds to one of the vector fields studied in Section~\ref{im1f}; we may thus suppose that~$n^2\neq 1$. Placing the horizontal line at the plane~$\{x=0\}$, the vertical at $\{y=0\}$, the radial orbit of exponents $(1,2)$ at $[0:1:0]$ and the one with exponents~$(-2,-1)$ at $[1:0:0]$, we find the vector field~$Y$ of equation~\ref{pseudobemol} for some~$\gamma\in\mathbb{C}$. Let
\begin{gather*}\gamma=\frac{8m^2n}{(mn+2n+2m)(mn-2n+2m)}.\end{gather*}
The other two radial orbits have exponents~$(-1,m)$ and~$(-1,-m)$; we may suppose that~$m^2\neq 1$ for otherwise we would be in the setting of Section~\ref{im1f}. These vector fields are of Halphen type with respect to
$C=\indel{x}-y/x\indel{y}$, for~$C$ is homogeneous of degree~$0$ and~$[C,Y]=2E$. For the adapted function~$\xi=(2\gamma/n) xy/z^2$,
\begin{gather*}-\frac{1}{\xi^2}\{\xi,t\}=\frac{1}{2}\frac{\big(1-\frac{1}{4}\big)(\xi-1)\xi-\big(1-\frac{1}{n^2}\big)(\xi-1) +\big(1-\frac{1}{m^2}\big)\xi}{\xi^2(\xi-1)^2}.\end{gather*}
Following the discussion in Section~\ref{sec:halphen}, equation~\ref{pseudobemol} is associated to the triangle group~$T(2,m,n)$. It is univalent when~$n,m\in\mathbb{Z}$, $|n|\geq 2$ and~$|m|\geq 2$ (recall that we have dismissed the vector fields where~$n^2=1$ or~$m^2=1$). These vector fields had been previously spotted in~\cite{guillot-fourier}, where it was shown that they are quotients of Halphen's systems~\ref{halpheneq} having two coincident parameters.\footnote{By setting
$m=(\alpha+2\beta-2)/\beta$, $n = -2(\alpha+2\beta-2)/\alpha$ ($\gamma = \alpha(\alpha+2\beta-2)/(2\beta-1)$) we find, in the coordinates
$(\zeta_1,\zeta_2,\zeta_3)=\big(x+\alpha^{-1}z, 2x, \frac{1}{2}x+(2[2\beta-1])^{-1}y\big)$,
the vector field~$\mathcal{H}_2(\alpha,\beta)$ in formula~(3.1) of~\cite{guillot-fourier}.}

By Lemma~\ref{notwopos}, there is no need to consider the case where the exponents of both distinguished radial orbits are positive.

\subsection[The plane associated to one distinguished radial orbit contains the other one (but the two planes are different)]{The plane associated to one distinguished radial orbit\\ contains the other one (but the two planes are different)}\label{connect}

\subsubsection{One radial orbit with positive exponents and one with negative ones}\label{oneposoneneg} Two out of the three eigenvalues of the restriction of the vector field to the invariant plane must be either~$(-1,2)$ or~$(1,-2)$. In the first case the restriction of the vector field to this plane cannot be, by the enumeration in~(\ref{tess}), univalent. In the second one we have the diagram in Fig.~\ref{fig:onepos1}.
Since the plane tangent to the exponent~$-1$ in the radial orbit with exponents~$(-2,-1)$ is invariant, the above diagram extends into a diagram of the form portrayed in Fig.~\ref{fig:211-2} for some~$n\in\mathbb{Z}$. We may assume that~$n\geq 2$ since~$n-1$ and~$-n$ play symmetric roles. If~$n=2$, $(1,n-1)=(1,1)$, and we would be back to Section~\ref{im1f}; if~$n=3$, $(1,n-1)=(1,2)$ and we would be in the setting of Section~\ref{2fnc}. We will thus suppose that~$n>3$. Relation~$\mathrm{R}_{0}$ reads
\begin{figure}
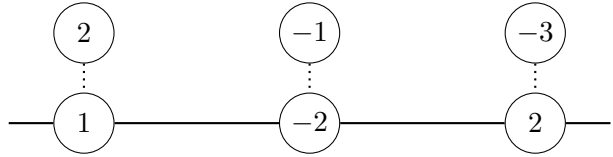
\centering
\oneline{$1$}{$-2$}{$2$}{$2$}{$-1$}{$-3$}
\caption{The invariant line for Section~\ref{oneposoneneg}.}\label{fig:onepos1}
\end{figure}
\begin{figure}
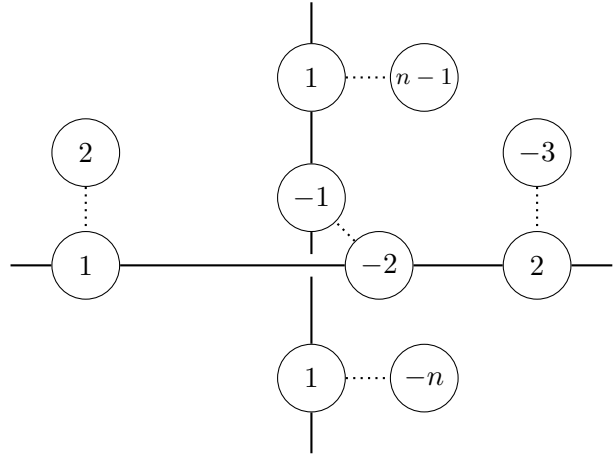
\centering
\tikz{
\tikzstyle{mycir}=[circle,minimum size=0.9cm,draw]
\node[mycir, label=center:$1$](u1) at (0,0) {};
\node[mycir, label=center:$-2$](u2) at (3.9,0) {};
\node[mycir, label=center:$2$](u3) at (6,0) {};
\node[mycir, label=center:$2$](v1) at (0,1.5) {};
\node[mycir, label=center:$-1$](v2) at (3,0.9) {};
\node[mycir, label=center:$-3$](v3) at (6,1.5) {};
\node[mycir, label=center:$1$](a1) at (3,2.5) {};
\node[mycir, label=center:$1$](a2) at (3,-1.5) {};
\node[mycir, label=center:\footnotesize{$n-1$}](b1) at (4.5,2.5) {};
\node[mycir, label=center:$-n$](b2) at (4.5,-1.5) {};
\draw[thick] (u1)--(u2);
\draw[thick] (u2)--(u3);
\draw[thick] (a1)--(v2);
\draw[thick] (v2)--(3,0.15);
\draw[thick] (3,-0.15)--(a2);
\draw[thick] (u1)--++(-1,0);
\draw[thick] (u3)--++(1,0);
\draw[thick] (a1)--++(0,1);
\draw[thick] (a2)--++(0,-1);
\draw[thick, dotted] (v1)--(u1);
\draw[thick, dotted] (v3)--(u3);
\draw[thick, dotted] (a1)--(b1);
\draw[thick, dotted] (a2)--(b2);
\draw[thick, dotted] (u2)--(v2);
}
\caption{The configuration in Section~\ref{oneposoneneg}.}\label{fig:211-2}
\end{figure}
\begin{gather}\label{3cintri}
\frac{1}{\xi_6}+\frac{1}{\xi_7}+\frac{1}{n(n-1)}=\frac{1}{6}.
\end{gather}
Several cases arise when considering the solutions to this equation.

\paragraph{First case, $\boldsymbol{\xi_6=6}$, $\boldsymbol{\xi_7=-n(n-1)}$.} We have the three equations~$\mathrm{R}_i$ for the three unknowns~$v_6$, $v_7$, $n$. From the solutions to these,
\begin{enumerate}\itemsep=0pt
\item for~$(u_6,v_6)=(6,1)$, $u_7=\frac{1}{2}\big({-}7+\sqrt{97}\big)$ and~$v_7=\frac{1}{2}\big({-}7-\sqrt{97}\big)$, which are not integers;
\item for~$(u_6,v_6)=(3,2)$, $n=2$ (but we had previously supposed that~$n> 3$);
\item for~$(u_6,v_6)=(-2,-3)$, $n=3/2$ or~$n=-1/2$ (which is not an integer);
\item for~$(u_6,v_6)=(-1,-6)$, $n=5/3$ or~$n=-2/3$.
\end{enumerate}
We can thus discard all these cases.

\paragraph{Second case, $\xi_6\neq 6$, $\xi_7\neq 6$.} Since~$n>3$, there are only finitely many solutions to~(\ref{3cintri}). We have used the computer to find them.\footnote{See the ancillary files \texttt{p1a.sage} and \texttt{p1b.sage}.} Let us list these solutions, in order to illustrate the nature of the problems that we are dealing with. The solutions are: for~$n=4$, $(3,-4)$, $(4,-6)$ $(8, -24)$, $(9,-36)$, $(10,-60)$, $(11,-132)$, $(13,156)$, $(14,84)$, $(15,60)$, $(16,48)$, $(18,36)$, $(20,30)$, $(21,28)$, $(24,24)$; for~$n=5$, $(5,-12) $, $(8,-120)$, $(9,180)$, $(10,60)$, $(12,30)$, $(15,20)$; for~$n=6$, $(3,-5)$, $(5,-15)$, $(7,-105)$, $(8,120)$, $(9,45)$, $(10,30)$, $(12,20)$, $(15,15)$; for~$n=7$, $(8,56 )$, $(14,14)$; for~$n=8$, $(7,168)$, $(8,42)$; for~$n=9$, $(8,36)$, $(9,24)$; for~$n=10$, $(10,18)$; for~$n=11$, $(11,15)$; for~$n=12$, $(4,-11)$; for~$n=13$, $(12,13)$; for~$n=16$, $(10,16)$; for~$n=19$, $(9,19)$; for~$n=21$,
$(5,-28)$; for~$n=25$, $(8,25)$ ; for~$n=30$, $(5,-29)$; for~$n=43$, $(7,43)$. In particular, there is no solution for~$n>43$.

Given~$n$, for each one of these solutions~$(\xi_6,\xi_7)$ we have finitely many choices for~$u_6$, $v_6$, $u_7$ and~$v_7$, which can be again obtained with the help of the computer. With these values we would have all the potential values of~$u_i$ and~$v_i$ for all~$i$, and we can check if they satisfy the equations~$\mathrm{R}_i$. We have done this and have found that in there is no set of~$u_i$'s and~$v_i$'s that is compatible with these relations in the above cases.

\subsubsection{Positive exponents for the two radial orbits}\label{sec:twopos}

These are the vector fields having a diagram of the form portrayed in Fig.~\ref{fig:twopos}. By Lemma~\ref{sing12}, another invariant plane passes through one of the radial orbit having exponents~$(2,1)$, and we have a diagram like the one in Fig.~\ref{dia214-2} for~$n,m\in\mathbb{Z}$. We may suppose that~$m>0$. Let~$k=n(m+n)$. Equation~$\mathrm{R}_0$ reads
\begin{figure}
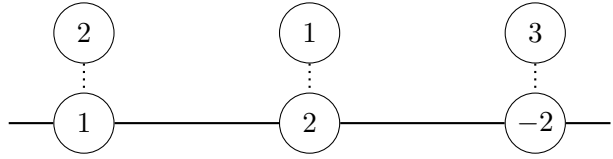
\centering
 \oneline{$1$}{$2$}{$-2$}{$2$}{$1$}{$3$}
\caption{The invariant line in Section~\ref{sec:twopos}.}\label{fig:twopos}
\end{figure}
\begin{figure}
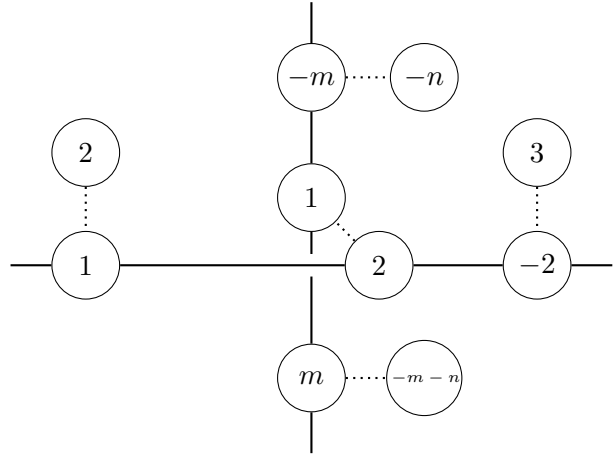
\centering
\tikz{
\tikzstyle{mycir}=[circle,minimum size=0.9cm,draw]
\tikzstyle{mycir2}=[circle,minimum size=1.0cm,draw]
\node[mycir, label=center:$1$](u1) at (0,0) {};
\node[mycir, label=center:$2$](u2) at (3.9,0) {};
\node[mycir, label=center:$-2$](u3) at (6,0) {};
\node[mycir, label=center:$2$](v1) at (0,1.5) {};
\node[mycir, label=center:$1$](v2) at (3,0.9) {};
\node[mycir, label=center:$3$](v3) at (6,1.5) {};
\node[mycir, label=center:$-m$](a1) at (3,2.5) {};
\node[mycir, label=center:$m$](a2) at (3,-1.5) {};
\node[mycir, label=center:$-n$](b1) at (4.5,2.5) {};
\node[mycir2, label=center:\tiny{\hspace{0.5ex}$-m-n$}](b2) at (4.5,-1.5) {};
\draw[thick] (u1)--(u2);
\draw[thick] (u2)--(u3);
\draw[thick] (a1)--(v2);
\draw[thick] (v2)--(3,0.15);
\draw[thick] (3,-0.15)--(a2);
\draw[thick] (u1)--++(-1,0);
\draw[thick] (u3)--++(1,0);
\draw[thick] (a1)--++(0,1);
\draw[thick] (a2)--++(0,-1);
\draw[thick, dotted] (v1)--(u1);
\draw[thick, dotted] (v3)--(u3);
\draw[thick, dotted] (a1)--(b1);
\draw[thick, dotted] (a2)--(b2);
\draw[thick, dotted] (u2)--(v2);
}
\caption{The master configuration of Section~\ref{sec:twopos}.} \label{dia214-2}
\end{figure}
\begin{gather}\label{anogtheregyp}\frac{1}{\xi_6}+\frac{1}{\xi_7}+\frac{1}{k}=\frac{1}{6}.\end{gather}
We have infinitely many solutions for it if one of the involved integers is~$6$, finitely many otherwise. The most general vector field having these data is
\begin{gather}\label{twonegative}x(x -y)\del{x}+y(y+[n+1]x-z)\del{y}+\big(z^2+\alpha xy +[m+1]xz\big)\del{z}\end{gather}
for some~$\alpha\in\mathbb{C}$. If we set~$\alpha=-\big(6\sigma^2+ 2m\sigma + 5n\sigma + mn + n^2\big)/\sigma$, $\sigma\in\mathbb{C}$, there is a radial orbit~$\rho_6$ at~$[1:\sigma:2\sigma+n]$ and another one, $\rho_7$, at~$[6\sigma:n(n+m):2n(n+m+3\sigma)]$. They are exchanged by the involution~$\sigma\mapsto \frac{1}{6}n(n+m)/\sigma$, so they are equivalent for most purposes. The exponents of~$\rho_6$ are the roots of
\begin{gather}\label{carac6}\zeta^2-\left(7+\frac{7+m+2n}{\sigma-1}\right)\zeta+\frac{6\sigma^2-n^2-mn}{(\sigma-1)^2},\end{gather}
while those of~$\rho_7$ are the roots of
\begin{gather}\label{carac6-bis}\zeta^2-\left(7+\frac{6\sigma[7+m+2n]}{n^2+mn-6\sigma}\right)\zeta +\frac{6n(n+m)\big(n^2+mn-6\sigma^2\big)}{\big(n^2+mn-6\sigma\big)^2}.\end{gather}
Notice that the coefficient of the linear term is~$-7$ in the first if and only if it is~$-7$ in the second ($\sigma\in\{0,\infty\}$ corresponds to degenerate vector fields) and that this happens if and only if~$7+m+2n=0$. We have three subcases according to the nature of the solutions of~(\ref{anogtheregyp}).

\paragraph{First subcase: in~(\ref{anogtheregyp}), $\boldsymbol{k=6}$.} Since~$n$ divides~$6$, the values that~$n$ may take are only finitely many; in turn, the value of~$n$ determines that of~$m$, for~$m=k/n-n$. Among the divisors of~$6$, if~$n\in\{6,3,-1,-2\}$, $m<0$. We have the following remaining cases.
\subparagraph{When~$\boldsymbol{(n,m)=(2,1)}$.} Since there is a radial orbit with exponents~$(-n,-m)$ $=(-2,-1)$, we are in the setting of Section~\ref{gendiffsng}.
\subparagraph{When~$\boldsymbol{(n,m)\!=\!({-}3,1)}$.} Since there is a radial orbit with exponents~$(m,{-}m-n)=(1,2)$, we are in the setting of Section~\ref{2neg2lines}.
\subparagraph{When~$\boldsymbol{(n,m)=(1,5)}$.} For the partial sum of~$\mathrm{R}_1$ we have~$\sum\limits_{i=1}^5 (1/u_i+1/v_i)=5/3$ and thus~$1/u_6+1/v_6+1/u_7+1/v_7=7/3$. If the number~$1$ appeared only once among $u_6$, $v_6$, $u_7$, $v_7$, say, $u_6=1$, we would have~$1/v_6+1/u_7+1/v_7=4/3$. Up to ordering, the only way to express~$4/3$ as a sum of inverses of integers without the summand~$1$ is~$4/3=1/2+1/2+1/3$ and we would either have a radial orbit with exponents $(1,2)$, going back to Section~\ref{2neg2lines}, or one with exponents~$(2,2)$. In this last case, through this radial orbit with exponents~$(2,2)$ we would have three invariant planes~\cite[Lemma~3.6]{guillot-fourier}. One of them, $P$, intersects the invariant plane associated with the vertical line in the diagram of Fig.~\ref{dia214-2} at the radial orbit with exponents~$(-m,-n)$. Two of the three exponents of the restriction to~$P$ would be~$2$ and~$-n=-1$, but this is impossible by the enumeration in~(\ref{tess}). Hence, the number~$1$ appears at least twice as an exponent among~$u_6$, $v_6$, $u_7$, $v_7$ (but not within the same radial orbit since that would take us back to Section~\ref{im1f}), say~$u_6=1$ and~$u_7=1$. From~$\mathrm{R}_0$, $v_6=-v_7$, but from~$\mathrm{R}_1$, $1/v_6+1/v_7=1/3$, which is impossible. (There is no vector field having these data.)
\subparagraph{When~$\boldsymbol{(n,m)=(-6,5)}$.} From~$\mathrm{R}_{0}$, $\xi_7=-\xi_6$. Since~$m+2n=-7$, from~(\ref{carac6}), $u_6$ and~$v_6$ are the roots of~$x^2-7x+d_6$ (whose discriminant is~$49-4d_6^2$), while~$u_7$ and~$v_7$ are the roots of~$x^2-7x-d_6$ (whose discriminant is~$49+4d_6^2$). Both discriminants must be squares of integers. Their sum is~$98$, but~$98$ cannot be written as the sum of two different squares of integers.

\paragraph{Second subcase: in~(\ref{anogtheregyp}), $\boldsymbol{\xi_6=6}$, $\boldsymbol{\xi_7=-k}$.} Given~$u_6$ and~$v_6$ (integers such that~$u_6v_6=6$), we have four unknowns ($m,n,u_7,v_7$) and the three equations~$\mathrm{R}_i$, which we can actually solve. The solutions, factoring out the symmetries of the problem, are the following.

\subparagraph{When~$\boldsymbol{(u_6,v_6)=(-2,-3)}$.}
We have~$n=-6$, $(u_7,v_7)=(6,q)$, $m=q+6$, $\sigma=\frac{1}{2}(1-q)$. From the vector field~(\ref{twonegative}), after rescaling~$y$ by~$t\mapsto(1-q)y$, we find the family~\ref{t3}. The image of the vector field under
\begin{gather*}(x,y,z)\mapsto\big(x[x-2y],2x\big[x^2+(q+3)xy+yz\big]\big)=(u,v)\end{gather*}
is~$v\indel{u}+6u^2\indel{v}$, which has the first integral~$g_3=4u^3-v^2$ (it induces a homogeneous first integral of degree~$6$ for the original vector field). The reduced equation has the solution~$(u,v)=(\wp,\wp')$,
for the Weierstrass elliptic function~$\wp$ such that~$(\wp')^2=4\wp^3-g_3$. Solving for~$y$ and~$z$ we find
\begin{gather*}y=\frac{x^2-\wp}{2x},\qquad z=\frac{\wp'+(q+3)x\wp-(5+q)x^3}{x^2-\wp}.\end{gather*}
Substituting in~(\ref{twonegative}) reduces the system of equations to~$x'=\frac{1}{2}(q+1)x^2+\frac{1}{2}(1-q)\wp$. If~$q=-1$, $x$ is the Weierstrass~$\zeta$ function, up to the addition of a constant. If~$q\neq -1$, for~$w=\frac{1}{2}(q+1)x$, $w'=w^2+\frac{1}{4}\big(1-q^2\big)\wp$: the vector field is semicomplete whenever~$q\in\mathbb{Z}\setminus6\mathbb{Z}$~(see~\cite{chazy,guillot-riccati}). These vector fields appeared in~\cite{these} as the family~$\mathcal{T}_3$.

\subparagraph{When~$\boldsymbol{(u_6,v_6)=(2,3)}$.}
We have~$n=-1$, $m=1-k$, $(u_7,v_7)=(1,-k)$, $\sigma=\frac{1}{12}(k+6)$. From the vector field~(\ref{twonegative}), upon rescaling the variables by~$(x,y,z)\mapsto(2x,[k+6]y,kz)$, we find the family of vector fields~$Y(k)$ of equation~\ref{t4}. (The data for~$k=-6$ is not realizable.) When~$k\neq 6$ these vector fields are of Halphen type: for
\begin{gather*}C=3\del{x}+\left(1-3\frac{y}{x}\right)\del{y}+\left(3-\frac{z}{y}\right)\del{z},\end{gather*}
we have that~$[C,Y]=(6-k)E$. They all have the common adapted function
\begin{gather*}\xi=\frac{x(x-6y)^3}{\big(x^2-9xy+9yz\big)^2},\end{gather*}
for which
\begin{gather*}-\frac{1}{\xi^2}\{\xi,t\}=\frac{1}{2}\frac{\big(1-\frac{1}{4}\big)(\xi-1)\xi-\big(1-\frac{1}{9}\big)(\xi-1)
+\big(1-\frac{1}{k^2}\big)\xi}{\xi^2(\xi-1)^2},\end{gather*}
corresponding, as described in Section~\ref{sec:halphen}, to the triangle group~$T(2,3,k)$. Thus, when~$k\neq 6$, the vector field is semicomplete whenever~$k\in\mathbb{Z}\setminus\{-1,0,1\}$. It is, after exchanging~$y$ and~$z$, the vector field~$k\mathcal{H}_{3a}(1/k)$ in~formula~(3.5a) of~\cite{guillot-fourier} and it is thus a quotient of the symmetric case of the Halphen system~\ref{halpheneq} when the three parameters coincide. These vector fields appeared previously in~\cite{these} as the family~$\mathcal{T}_4$.

When~$k=6$ the vector field commutes with~$C$. The projection onto the orbit space of~$C$
\begin{gather*}(x,y,z)\mapsto\big(4x[x-6y],16x\big[x^2-9xy+9yz\big]\big)=(u,v)\end{gather*} maps the vector field to
$v\indel{u}+6u^2\indel{v}$; the solution of this reduced equation is given by~$(u,v)=(\wp,\wp')$, for the Weierstrass elliptic function~$\wp$ such that~$(\wp')^2=4\wp^3-g_3$. Sol\-ving for~$y$ and~$z$ we obtain
\begin{gather*}y(t)=\frac{4x^2-\wp}{24x},\qquad z(t)=\frac{8x^3-6x\wp+\wp'}{6\big(4x^2-\wp\big)},\end{gather*}
and substituting in~(\ref{twonegative}), we have~$x'(t)=\frac{1}{2}\wp(t)$, and thus~$x(t)=-\frac{1}{2}\zeta(t)+c$, for the Weierstrass~$\zeta$ function. The vector field is also semicomplete for~$k=6$.

\subparagraph{When~$\boldsymbol{(u_6,v_6)=(-1,-6)}$.} The equations impose the condition $1/n+1/(m+n)=-7/6$, and thus either~($m+n=-6$ and~$n=-1$) or ($m+n=-1$ and~$n=-6$). In the first case we have~$m=-5$ (but we supposed~$m>0$); in the second, since~$k=6$, we are back to the first subcase.

\subparagraph{When~$\boldsymbol{(u_6,v_6)=(1,6)}$.} Since~$u_6+v_6=7$, from~(\ref{carac6}) and~(\ref{carac6-bis}), $u_7+v_7=7$ and~$m+2n=-7$ (we may set~$v_7=7-u_7$ and~$m=-7-2n$). The relation~$\mathrm{R}_0$ implies that~$n^2+7n-7u_7+u_7^2=0$ and thus~$u_7=\frac{7}{2}\pm\frac{1}{2}\sqrt{49-28n-4n^2}$, which is an integer only if~$n=0$ or~$n=-7$ (which implies~$-m-n=0$). These correspond to degenerate vector fields.

\paragraph{Third subcase: in~(\ref{anogtheregyp}), $\boldsymbol{\kappa\neq 6}$, $\boldsymbol{\xi_6\neq 6}$, $\boldsymbol{\xi_7\neq 6}$.} There are 140 unordered triples of integers that solve the equation~$1/n_1+1/n_2+1/n_3=6$ under the condition that~$n_i\neq 6$ for all~$i$. From the list of such triples we can deduce all the possible values of~$\xi_6$, $\xi_7$ and~$k$. From the values of the first two, we may deduce all potential values of~$(u_6,v_6)$ and~$(u_7,v_7)$, as well as all factorizations of~$\xi_6$ and~$\xi_7$, and, from the latter, the putative values of~$n$ (among the factors of~$k$) and~$m$ (under the assumption that~$m>0$). We thus get possible values for~$u_i$, $v_i$, which can then be tested for compliance with the relations~$\mathrm{R}_i$. We obtain in this way the following five solutions:\footnote{See the files \texttt{p2a.sage} and \texttt{p2b.sage}.}
\begin{enumerate}\itemsep=0pt
\item $k=10$, $n=-5$, $m=3$, $(u_6,v_6)=(-3,10)$, $(u_7,v_7)=(2,5)$,
\item $k=-30$, $n=-10$, $m=13$, $(u_6,v_6)=(2,5)$, $(u_7,v_7)=(2,5)$,
\item $k=10$, $n=-5$, $m=3$, $(u_6,v_6)=(-5,12)$, $(u_7,v_7)=(3,4)$,
\item $k=12$, $n=-4$, $m=1$, $(u_6,v_6)=(-5,12)$, $(u_7,v_7)=(2,5)$,
\item $k=-60$, $n=-12$, $m=17$, $(u_6,v_6)=(3,4)$, $(u_7,v_7)=(2,5)$.
\end{enumerate}
In all of them, $m+2n+7=0$. Substituting these values of~$(u_6,v_6)$ as roots of~(\ref{carac6}), we obtain a quadratic equation for~$\sigma$; after substituting one of the roots in~(\ref{twonegative}), we obtain, respectively, the equations~\ref{chazyalt2}, \ref{chazyalt4}, \ref{chazyalt1}, \ref{chazyalt3} and~\ref{chazyalt5}. The equation corresponding to the other value of~$\sigma$ is obtained by changing the determination of the square root (except in equation~\ref{chazyalt4}, where the two values of~$\sigma$ give linearly equivalent equations).

Let us integrate them by showing that each one of them has a polynomial first integral and that its restriction to a particular level set of it is birationally equivalent to the restriction of either the Chazy~IX or the Chazy~X equation to a level set of the corresponding first integral.

\subparagraph{Equation~\ref{chazyalt2}.}
It is birationally equivalent to the Chazy~IX equation~(\ref{chazyix}). If~$\phi$ is a solution to the latter, a solution to this equation is given by
\begin{gather*}x=-\frac{3}{2}\big(\sqrt{5}-1\big)\phi,\qquad y=-\frac{1}{4}\big(\sqrt{5}-1\big)\frac{6\phi^2+\big(\sqrt{5}+1\big)\phi'}{\phi}, \\ z=\frac{1}{2}\big(\sqrt{5}-1\big)\frac{54\phi^3-\big(3+\sqrt{5}\big)\phi''+3\big(\sqrt{5}+1\big)\phi\phi'}{6\phi^2+\big(\sqrt{5}+1\big)\phi'}.\end{gather*}
Reciprocally, from a solution to~\ref{chazyalt2}, a solution to Chazy~IX is given by the birational inverse of the above,
\begin{gather*}\phi=-\frac{1}{6}\big(1+\sqrt{5}\big)x, \qquad \phi'=-\frac{1}{6}\big(1+\sqrt{5}\big)x(x-y), \\ \phi''=-\frac{1}{6}\big(1+\sqrt{5}\big)x\big(2x^2+xy+yz\big).\end{gather*}
This proves that the vector field is semicomplete.\footnote{In order to integrate equation~\ref{chazyalt2}, we started by looking (with the help of the computer) for homogeneous polynomial first integrals and homogeneous polynomial vector fields that commuted with the vector field we wished to integrate. Having established that the vector field had a homogeneous polynomial first integral of degree~$10$ and commuted with a homogeneous polynomial vector field of degree~$4$ having the same first integral, just like the Chazy~IX equation, it began to seem likely for both equations to be birationally equivalent (by some birational mapping that preserved the homogeneity). We then looked for linear functions of~$x$, $y$ and~$z$ such that, when evaluated at a solution to equation~\ref{chazyalt2}, solved the Chazy~IX equation, and obtained the present result. Equations~\ref{chazyalt4}--\ref{chazyalt5} were integrated in the same way (taking into account that Chazy~X has a polynomial first integral of degree~$12$ and commutes with a homogeneous polynomial vector field of degree~$6$ having the same first integral), considering the possibility that the homogeneous vector fields that commute with them could be rational, and not necessarily polynomial.} It has a homogeneous polynomial first integral of degree~$10$ and commutes with a homogeneous polynomial vector field of degree~$4$ having the same first integral.

\subparagraph{Equation~\ref{chazyalt4}.}
It is also birationally equivalent to Chazy~IX. If~$\phi$ is a~solution to the Chazy~IX equation, a solution to this equation is given by
\begin{gather*}x=-\phi,\qquad y=-\frac{\phi^2+\phi'}{\phi},\qquad z=\frac{8\phi^3+6\phi\phi'-
\phi''}{\phi^2+\phi'}.\end{gather*}
Reciprocally, from a solution~$(x,y,z)$ to this equation we may build a solution
\begin{gather*}\phi=-x, \qquad \phi'= x(y-x), \qquad \phi''=-x\big(2x^2+6yx+yz\big),\end{gather*}
to Chazy~IX (this is the birational inverse of the previous map). Equation~\ref{chazyalt4} has a homogeneous polynomial first integral of degree~$10$ and commutes with a homogeneous polynomial vector field of degree~$14$ (which, after dividing by the first integral, becomes a rational homogeneous vector field of degree~$4$ which commutes with the original one) having the same first integral.

\subparagraph{Equation~\ref{chazyalt1}.}
It is birationally equivalent to the Chazy~X equation~(\ref{chazyx}). If~$\phi$ is a solution to the latter, a solution to this equation is given by
\begin{gather*}x=\frac{1}{2}\big(3+\sqrt{3}\big)\phi,\qquad y=\frac{1}{6}\big(3+\sqrt{3}\big)\frac{\big({-}3+\sqrt{3}\big)\phi'+3\phi^2}{\phi}, \\ z=-\frac{1}{2}\big(3+\sqrt{3}\big)\frac{\big({-}4+2\sqrt{3}\big)\phi''+9\phi^3 +\big({-}3+\sqrt{3}\big)\phi\phi'}{\big({-}3+\sqrt{3}\big)\phi'+3\phi^2}.\end{gather*}
Reciprocally, from a solution to~\ref{chazyalt1} a solution to Chazy~X is given by
\begin{gather*}\phi=\frac{1}{3}\big(3-\sqrt{3}\big)x , \qquad \phi'=\frac{1}{3}\big(3-\sqrt{3}\big)x(x-y), \qquad \phi''=\frac{1}{3}\big(3-\sqrt{3}\big)x\big(2x^2+xy+yz\big).\end{gather*}
These two are birational inverses of one another. Equation~\ref{chazyalt1} has a homogeneous polynomial first integral of degree~$12$ and commutes with a homogeneous polynomial vector field of degree~$6$ having the same first integral.

\subparagraph{Equation~\ref{chazyalt3}.}
It is birationally equivalent to Chazy~X. If~$\phi$ is a solution to the Chazy~X equation, a solution to this equation is given by
\begin{gather*}x= -\phi,\qquad y=-\frac{\phi^2+\phi'}{\phi},\qquad z=\frac{2\phi^3-\phi''}{\phi^2+\phi'}.\end{gather*}
Reciprocally, from a solution to this equation, a solution to Chazy~X may be given by the birational inverse
\begin{gather*}\phi= -x,\qquad \phi'=x(y-x),\qquad \phi''=-x\big(2x^2+yz\big)\end{gather*}
of the above map. Equation~\ref{chazyalt3} has a homogeneous polynomial first integral of degree 12 and commutes with a homogeneous polynomial vector field of degree~$6$ having the same first integral.

\subparagraph{Equation~\ref{chazyalt5}.}
It is birationally equivalent to Chazy~X. If~$\phi$ is a solution to the Chazy~X equation, a solution to this equation is given by
\begin{gather*}x=-\frac{1}{11}\big(1+2\sqrt{3}\big)\phi,\qquad y=\frac{1}{11}\big(1+2\sqrt{3}\big)\frac{\big(1-2\sqrt{3}\big)\phi'-\phi^2}{\phi}, \\ z=\frac{1}{11}\big(1+2\sqrt{3}\big)\frac{\big(13-4\sqrt{3}\big)\phi''-10\phi^3+8\big(1-2\sqrt{3}\big)\phi\phi'}{\big(1-2\sqrt{3}\big)\phi'-\phi^2}.
\end{gather*}
Further, from a solution to this equation a solution to the Chazy~X equation may be given by
\begin{gather*}\phi=\big(1-2\sqrt{3}\big)x, \qquad \phi'= \big(1-2\sqrt{3}\big)x(x-y), \qquad \phi''=\big(1-2\sqrt{3}\big)x\big(2x^2+8xy+yz\big),\end{gather*}
which is the birational inverse of the previous map. It has a homogeneous polynomial first integral of degree 12 and commutes with a homogeneous polynomial vector field of degree~$18$ (which becomes rational of degree~$6$ after dividing by the first integral) having the same first integral.

By Lemma~\ref{notwopos}, there is no need to consider the case where the exponents of both radial orbits are positive.

\subsection{Same invariant plane}\label{sec:sameplane} We have one of the diagrams, (A) or~(B), appearing in Fig.~\ref{fig:sameplane}. In both cases, for~$\mathrm{R}_0$ to hold one must have
\begin{figure}
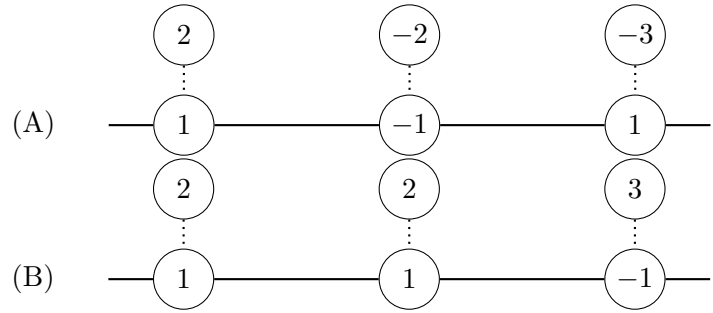
\centering
\onelinwlab{$1$}{$-1$}{$1$}{$2$}{$-2$}{$-3$}{(A)}

\onelinwlab{$1$}{$1$}{$-1$}{$2$}{$2$}{$3$}{(B)}

\caption{The two possible configurations for Section~\ref{sec:sameplane}.} \label{fig:sameplane}
\end{figure}
\begin{gather}\label{form:sil}\frac{1}{\xi_4}+\frac{1}{\xi_5}+\frac{1}{\xi_6}+\frac{1}{\xi_7}=\frac{1}{3}.\end{gather}
As before, special situations arise whenever the equality is already realized by a proper subset of the summands.

\subsubsection{In~(\ref{form:sil}), $\boldsymbol{\xi_4=3}$} This is the case where, in equation~(\ref{form:sil}), we have~$\xi_4=3$. We either have~$(u_4,v_4)=(1,3)$ or~$(u_4,v_4)=(-1,-3)$. In both cases, the exponents are resonant, and there is a further obstruction on the vector fields for semicompleteness. This obstruction can be easily read in the foliation induced on~${\mathbb{CP}}^2$ at the corresponding point~$p$: the singularity of the foliation, which is a resonant node, must be linearizable. This is equivalent to having, in the induced foliation, an invariant curve tangent to the exponent~$\pm 1$. One way to guarantee the existence of this curve is to impose the invariance of the plane that passes through~$p$ tangent to the exponent~$\pm 1$. This will be the assumption in Sections~\ref{slimmw} through~\ref{lsppmix}. The invariance of this plane is not the only way to overcome the linearizability obstruction, as we will show in Section~\ref{sliimmo}. The results there obtained will be used in Sections~\ref{slimmo} through~\ref{sliippo} to produce the classification in these cases.

\paragraph{Case (A), $\boldsymbol{(u_4,v_4)=(1,3)}$, with invariant plane.}\label{slimmw}
If the plane tangent to the exponent~$1$ is invariant, we have a diagram of one of the forms portrayed in Figs.~\ref{doslin01}, \ref{doslin02} or~\ref{doslin03}. For the first two, for the partial sum of~$\mathrm{R_0}$ we have that~$\sum\limits_{i=1}^5 1/\xi_i=9/10\neq 1$ and for last one~$\sum\limits_{i=1}^5 1/\xi_i=14/15\neq 1$. For each one of these cases we can (with the help of the computer) find all the possible (finitely many) nonzero integers~$\xi_6$ and~$\xi_7$ that satisfy, together with the data already obtained, equation~$\mathrm{R}_0$. For each value of~$\xi_6$ and~$\xi_7$ we can find all possible values of~$(u_6,v_6)$ and~$(v_7,u_7)$ and, with these, test for equations~$\mathrm{R}_1$ and~$\mathrm{R}_2$ (in order to find all seven couples of integers containing the above set of five pairs and satisfying the three relations~$\mathrm{R}_i$). We did not find any solutions in any of these three cases.\footnote{See the file \texttt{p3.sage}.}

\begin{figure}
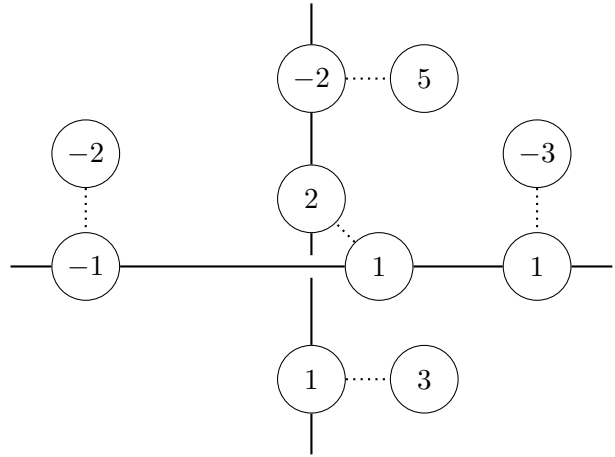
\centering
\tikz{
\tikzstyle{mycir}=[circle,minimum size=0.9cm,draw]
\node[mycir, label=center:$-1$](u1) at (0,0) {};
\node[mycir, label=center:$1$](u2) at (3.9,0) {};
\node[mycir, label=center:$1$](u3) at (6,0) {};
\node[mycir, label=center:$-2$](v1) at (0,1.5) {};
\node[mycir, label=center:$2$](v2) at (3,0.9) {};
\node[mycir, label=center:$-3$](v3) at (6,1.5) {};
\node[mycir, label=center:$-2$](a1) at (3,2.5) {};
\node[mycir, label=center:$1$](a2) at (3,-1.5) {};
\node[mycir, label=center:$5$](b1) at (4.5,2.5) {};
\node[mycir, label=center:$3$](b2) at (4.5,-1.5) {};
\draw[thick] (u1)--(u2);
\draw[thick] (u2)--(u3);
\draw[thick] (a1)--(v2);
\draw[thick] (v2)--(3,0.15);
\draw[thick] (3,-0.15)--(a2);
\draw[thick] (u1)--++(-1,0);
\draw[thick] (u3)--++(1,0);
\draw[thick] (a1)--++(0,1);
\draw[thick] (a2)--++(0,-1);
\draw[thick, dotted] (v1)--(u1);
\draw[thick, dotted] (v3)--(u3);
\draw[thick, dotted] (a1)--(b1);
\draw[thick, dotted] (a2)--(b2);
\draw[thick, dotted] (u2)--(v2);
}
\caption{The first configuration of Section~\ref{slimmw}.}\label{doslin01}
\end{figure}

\begin{figure}
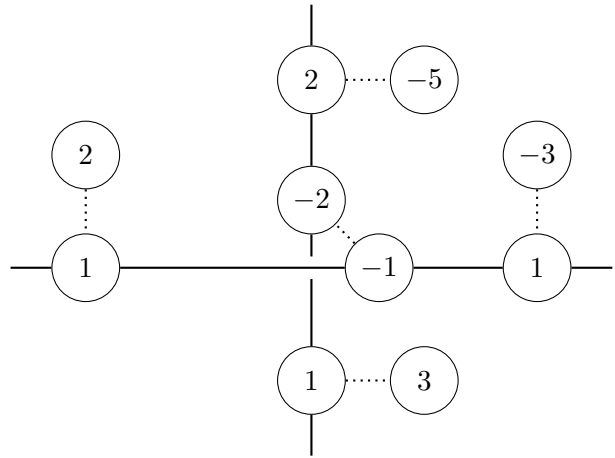
\centering
\tikz{
\tikzstyle{mycir}=[circle,minimum size=0.9cm,draw]
\node[mycir, label=center:$1$](u1) at (0,0) {};
\node[mycir, label=center:$-1$](u2) at (3.9,0) {};
\node[mycir, label=center:$1$](u3) at (6,0) {};
\node[mycir, label=center:$2$](v1) at (0,1.5) {};
\node[mycir, label=center:$-2$](v2) at (3,0.9) {};
\node[mycir, label=center:$-3$](v3) at (6,1.5) {};
\node[mycir, label=center:$2$](a1) at (3,2.5) {};
\node[mycir, label=center:$1$](a2) at (3,-1.5) {};
\node[mycir, label=center:$-5$](b1) at (4.5,2.5) {};
\node[mycir, label=center:$3$](b2) at (4.5,-1.5) {};
\draw[thick] (u1)--(u2);
\draw[thick] (u2)--(u3);
\draw[thick] (a1)--(v2);
\draw[thick] (v2)--(3,0.15);
\draw[thick] (3,-0.15)--(a2);
\draw[thick] (u1)--++(-1,0);
\draw[thick] (u3)--++(1,0);
\draw[thick] (a1)--++(0,1);
\draw[thick] (a2)--++(0,-1);
\draw[thick, dotted] (v1)--(u1);
\draw[thick, dotted] (v3)--(u3);
\draw[thick, dotted] (a1)--(b1);
\draw[thick, dotted] (a2)--(b2);
\draw[thick, dotted] (u2)--(v2);
}
 \caption{The second configuration of Section~\ref{slimmw}.}\label{doslin02}
\end{figure}

\begin{figure}
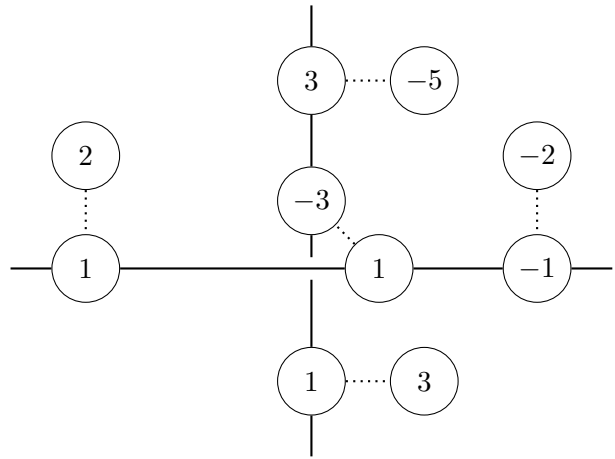
\centering
\tikz{
\tikzstyle{mycir}=[circle,minimum size=0.9cm,draw]
\node[mycir, label=center:$1$](u1) at (0,0) {};
\node[mycir, label=center:$1$](u2) at (3.9,0) {};
\node[mycir, label=center:$-1$](u3) at (6,0) {};
\node[mycir, label=center:$2$](v1) at (0,1.5) {};
\node[mycir, label=center:$-3$](v2) at (3,0.9) {};
\node[mycir, label=center:$-2$](v3) at (6,1.5) {};
\node[mycir, label=center:$3$](a1) at (3,2.5) {};
\node[mycir, label=center:$1$](a2) at (3,-1.5) {};
\node[mycir, label=center:$-5$](b1) at (4.5,2.5) {};
\node[mycir, label=center:$3$](b2) at (4.5,-1.5) {};
\draw[thick] (u1)--(u2);
\draw[thick] (u2)--(u3);
\draw[thick] (a1)--(v2);
\draw[thick] (v2)--(3,0.15);
\draw[thick] (3,-0.15)--(a2);
\draw[thick] (u1)--++(-1,0);
\draw[thick] (u3)--++(1,0);
\draw[thick] (a1)--++(0,1);
\draw[thick] (a2)--++(0,-1);
\draw[thick, dotted] (v1)--(u1);
\draw[thick, dotted] (v3)--(u3);
\draw[thick, dotted] (a1)--(b1);
\draw[thick, dotted] (a2)--(b2);
\draw[thick, dotted] (u2)--(v2);
}
\caption{The third configuration of Section~\ref{slimmw}.}\label{doslin03}
\end{figure}

\paragraph{Case (B), $\boldsymbol{(u_4,v_4)=(1,3)}$, with invariant plane.}\label{b13}

We have a diagram of one of the forms in Figs.~\ref{doslin04} or~\ref{doslin05}. For the first, $\sum\limits_{i=1}^5 1/\xi_i=9/10\neq 1$ and for the second one, $\sum\limits_{i=1}^5 1/\xi_i=14/15\neq 1$. By performing the test just described, we find no solutions in the first case. In the second case, we find that the solution can be completed by either~($(u_6,v_6)=(2,5)$, $(u_7,v_7)=(-3,10)$) or ($(u_6,v_6)=(3,4)$, $(u_7,v_7)=(-5,12)$)~-- which will give, respectively, equations~\ref{chazyixbis} and~\ref{chazyxbis}~-- and that these are the only solutions.\footnote{For the analysis of both cases see the file \texttt{p3.sage}.}

\begin{figure}
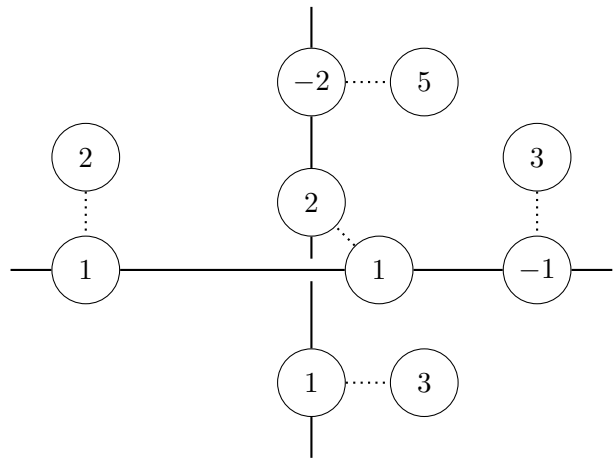
\centering
\tikz{
\tikzstyle{mycir}=[circle,minimum size=0.9cm,draw]
\node[mycir, label=center:$1$](u1) at (0,0) {};
\node[mycir, label=center:$1$](u2) at (3.9,0) {};
\node[mycir, label=center:$-1$](u3) at (6,0) {};
\node[mycir, label=center:$2$](v1) at (0,1.5) {};
\node[mycir, label=center:$2$](v2) at (3,0.9) {};
\node[mycir, label=center:$3$](v3) at (6,1.5) {};
\node[mycir, label=center:$-2$](a1) at (3,2.5) {};
\node[mycir, label=center:$1$](a2) at (3,-1.5) {};
\node[mycir, label=center:$5$](b1) at (4.5,2.5) {};
\node[mycir, label=center:$3$](b2) at (4.5,-1.5) {};
\draw[thick] (u1)--(u2);
\draw[thick] (u2)--(u3);
\draw[thick] (a1)--(v2);
\draw[thick] (v2)--(3,0.15);
\draw[thick] (3,-0.15)--(a2);
\draw[thick] (u1)--++(-1,0);
\draw[thick] (u3)--++(1,0);
\draw[thick] (a1)--++(0,1);
\draw[thick] (a2)--++(0,-1);
\draw[thick, dotted] (v1)--(u1);
\draw[thick, dotted] (v3)--(u3);
\draw[thick, dotted] (a1)--(b1);
\draw[thick, dotted] (a2)--(b2);
\draw[thick, dotted] (u2)--(v2);
}
\caption{One of the configurations of Section~\ref{b13}.}\label{doslin04}
\end{figure}

\begin{figure}
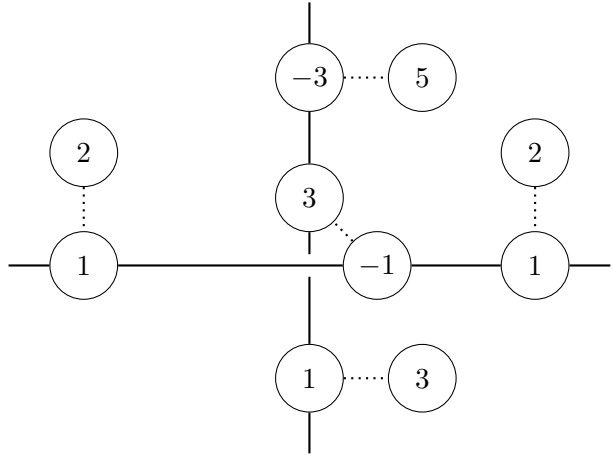
\centering
\tikz{
\tikzstyle{mycir}=[circle,minimum size=0.9cm,draw]
\node[mycir, label=center:$1$](u1) at (0,0) {};
\node[mycir, label=center:$-1$](u2) at (3.9,0) {};
\node[mycir, label=center:$1$](u3) at (6,0) {};
\node[mycir, label=center:$2$](v1) at (0,1.5) {};
\node[mycir, label=center:$3$](v2) at (3,0.9) {};
\node[mycir, label=center:$2$](v3) at (6,1.5) {};
\node[mycir, label=center:$-3$](a1) at (3,2.5) {};
\node[mycir, label=center:$1$](a2) at (3,-1.5) {};
\node[mycir, label=center:$5$](b1) at (4.5,2.5) {};
\node[mycir, label=center:$3$](b2) at (4.5,-1.5) {};
\draw[thick] (u1)--(u2);
\draw[thick] (u2)--(u3);
\draw[thick] (a1)--(v2);
\draw[thick] (v2)--(3,0.15);
\draw[thick] (3,-0.15)--(a2);
\draw[thick] (u1)--++(-1,0);
\draw[thick] (u3)--++(1,0);
\draw[thick] (a1)--++(0,1);
\draw[thick] (a2)--++(0,-1);
\draw[thick, dotted] (v1)--(u1);
\draw[thick, dotted] (v3)--(u3);
\draw[thick, dotted] (a1)--(b1);
\draw[thick, dotted] (a2)--(b2);
\draw[thick, dotted] (u2)--(v2);
}
\caption{The configuration of Section~\ref{b13} leading to equations~\ref{chazyixbis} and~\ref{chazyxbis}.}\label{doslin05}
\end{figure}

The vector fields
\begin{gather*}\big(x^2+[1-a]yz\big)\del{x}+y(2x+y-2z)\del{y}+z(z-2x-y)\del{z},\end{gather*}
for~$a\in\mathbb{C}$, are the most general ones representing the configuration of the diagram in Fig.~\ref{doslin05}. Two values of~$a$ differing by a sign give the same vector field: the involution~$(x,y,z)\mapsto (x+y,-y,z)$ preserves the family and acts upon~$a$ by~$a\mapsto -a$. By setting~$a=-\big(\sigma^2+15\big)/(8\sigma)$ we find a~radial orbit at~$[s+3:-2s:4]$ whose exponents are the roots of~$\zeta^2-7\zeta+15-\sigma^2$.

\subparagraph{Equation~\ref{chazyixbis}.}
The case~$(u_6,v_6)=(2,5)$, $(u_7,v_7)=(-3,10)$ can be realized by~$a=\frac{1}{2}\sqrt{5}$; in this way we find equation~\ref{chazyixbis}. It has a homogeneous polynomial first integral of degree 10 and commutes with a homogeneous polynomial vector field of degree~$4$ having the same first integral (like Chazy~IX). When restricted to a generic level surface of the first integral, the vector field is birationally equivalent to the restriction of Chazy~IX to a generic level surface of its first integral: if~$\phi$ is a solution to the Chazy~IX equation, a solution to equation~\ref{chazyixbis} is given by
\begin{gather*}x=\frac{3}{2}\big(1-\sqrt{5}\big)\phi, \qquad y=-\frac{3\big(1+\sqrt{5}\big)(6\phi^2+\big[1+\sqrt{5}\big]\phi')^2}{2\big(\big[1+\sqrt{5}\big]\phi''+12\phi\phi'\big)},\\
z=-\frac{\big(1+\sqrt{5}\big)\phi''+12\phi\phi'}{\big(1+\sqrt{5}\big)\phi'+6\phi^2};\end{gather*} reciprocally, from a solution~$(x,y,z)$ of equation~\ref{chazyixbis}, we have the solution of Chazy~IX
\begin{gather*}\phi=-\frac{1}{6}\big(1+\sqrt{5}\big)x,\qquad \phi'=-\frac{1}{12}\big(1+\sqrt{5}\big)\big(2x^2+\big[2-\sqrt{5}\big]zy\big), \\ \phi''=-\frac{1}{12}\big(1+\sqrt{5}\big)\big(4x^3+\big[4-2\sqrt{5}\big]xyz-\big[2-\sqrt{5}\big]z^2y\big) \end{gather*}
(these two transformations are birational inverses of one another).

\subparagraph{Equation~\ref{chazyxbis}.}
The case~$(u_6,v_6)=(3,4)$, $(u_7,v_7)=(-5,12)$ can be realized by~$a=-\frac{3}{4}\sqrt{3}$, obtaining equation~\ref{chazyxbis}. It has a homogeneous polynomial first integral of degree 12 and commutes with a homogeneous sextic polynomial vector field having the same first integral (like Chazy~X). When restricted to a generic level surface of the first integral, it is birationally equivalent to the restriction of Chazy~X to a generic level surface: for a solution~$\phi$ of Chazy~X, a solution to equation~\ref{chazyxbis} is given by
\begin{gather*}x=\frac{1}{2}\big(3+\sqrt{3}\big)\phi, \qquad y=\frac{2}{11}\big(1+2\sqrt{3}\big)\frac{\big(3\phi^2-\big[3-\sqrt{3}\big]\phi'\big)^2}{6\phi\phi'-\big(3-\sqrt{3}\big)\phi''}, \\ z=\frac{\big(3-\sqrt{3}\big)\phi''-6\phi\phi'}{3\phi^2-\big(3-\sqrt{3}\big)\phi'};\end{gather*}
reciprocally, for a solution~$(x,y,z)$ to~\ref{chazyxbis}, a solution to Chazy~X is given by
\begin{gather*}\phi=\frac{1}{3}\big(3-\sqrt{3}\big)x, \qquad \phi'=\frac{1}{12}\big(3-\sqrt{3}\big)\big(4x^2+\big[4+3\sqrt{3}\big]zy\big), \\ \phi''=\frac{1}{12}\big(3-\sqrt{3}\big)\big(8x^3+2\big[4+3\sqrt{3}\big]xzy-\big[4+3\sqrt{3}\big]z^2y\big)\end{gather*}
(these two transformations are birational inverses of one another).

\paragraph{Cases (A) \& (B), $\boldsymbol{(u_4,v_4)=(-1,-3)}$, with invariant plane.}\label{lsppmix} From the enumeration in~(\ref{tess}), the exponents of the only two-dimensional nondegenerate quadratic homogeneous vector field with one radial direction with exponent~$-1$ are~$(-1,1,1)$. The exponent~$-1$ belongs to the fourth radial orbit and the exponents~$1$ to the other radial orbits in the invariant plane. One of these must correspond to the exponent of a radial orbit in the original plane, but all the exponents~$1$ are tangent to the original invariant plane. These data cannot be realized.

\paragraph{Case (B), $\boldsymbol{(u_4,v_4)=(-1,-3)}$ without invariant plane.}\label{sliimmo}

Let us construct the most general vector field having the required data and such that there is neither an obstruction to semicompleteness given by the resonance at the fourth point nor an invariant plane making this obstruction vanish. The most general vector field having~$z=0$ as an invariant plane, having radial orbits with exponents~$(1,2)$ at~$[1:0:0]$ and~$[0:1:0]$ (both with the exponent~$1$ along the invariant plane) and a radial orbit at~$p=[0:0:1]$ with exponents whose sum is~$-4$ is, up to a rescaling of the variables,
\begin{gather}\label{yetanothervf}\left(x^2+\alpha xz+\beta yz\right)\del{x}+\left(y^2+\gamma xz-[\alpha+2] yz\right)\del{y}+z\left(z-x-y\right)\del{z},\end{gather}
for~$\alpha, \beta, \gamma\in\mathbb{C}$. If~$\beta$ and~$\gamma$ vanish simultaneously, there are two invariant lines through~$p$, so we may suppose that~$\beta\neq 0$. The exponents at~$p$ are the roots of~$\xi^2 - 4\xi+ 3 -\big(\alpha^2 + 2\alpha+\beta\gamma\big)$ and we may thus set~$\gamma=-\alpha(\alpha+2)/\beta$ so that the exponents at~$p$ are~$(1,3)$. At~$p$, the exponent~$1$ is tangent to the line~$\ell$ given by $\alpha x+\beta y=0$. Since this line is not invariant, $\alpha(\alpha+\beta)\neq 0$.

The induced foliation on~${\mathbb{CP}}^2$ is, in the affine chart~$[u:v:1]$, given by the kernel of the form
\begin{gather*}\big(u[\alpha-1]+\beta v+2u^2+uv\big){\rm d}v-\big({-}\alpha[\alpha+2]\beta^{-1}u-[3+\alpha]v+uv+2v^2\big){\rm d}u.\end{gather*}
Doing a first blowup by $v=\big(w-\alpha\beta^{-1}\big)u$ we find the foliation
\begin{gather*}u\big(1+\big[\alpha\beta^{-1}-2\big]u-\beta w-uw\big){\rm d}w\\
\qquad{} -\big(2w-\alpha\big[1+\alpha\beta^{-1}\big]\beta^{-1}u+\big[1+2\alpha\beta^{-1}\big]uw+\beta w^2-uw^2\big){\rm d}u.\end{gather*}
A second blowup by~$w=\big(t+\alpha(\beta+\alpha)\beta^{-2}\big)u$ yields the foliation
\begin{gather*}(u+\cdots){\rm d}t-\big(t+\alpha(\alpha+\beta)\big(2\alpha^2+\alpha+2\alpha\beta+3\beta\big)\beta^{-3}u+\cdots\big){\rm d}u.\end{gather*}
The condition for the original equation to induce a linearizable foliation at~$p$ is that
\begin{gather*}\alpha(\alpha+\beta)\big(2\alpha^2+\alpha+2\alpha\beta+3\beta\big)= 0.\end{gather*}
Since~$\ell$ is not invariant, we must have
\begin{gather}\label{beta}\beta=-\frac{\alpha(2\alpha+1)}{2\alpha+3}.\end{gather}
This gives the most general vector field having the sought conditions. Set~$\alpha=\frac{1}{4}(k-1)$ so that $\beta=-\frac{1}{4}(k+1)(k-1)/(k+5)$. After rescaling the variables by~$(x,y,z)\mapsto \big({-}[k+1]^{-1}x,[k+5]^{-1}y,\frac{1}{4}z\big)$, we obtain the vector field~$Y(k)$ in equation~\ref{t2}. This vector field is still the most general one having the sought conditions, for it is nondegenerate whenever~$\beta$ is in~$\mathbb{C}\setminus\{0\}$. Let us integrate equation~\ref{t2}. The image of~$Y(k)$ under
\begin{gather*}(x,y,z)\mapsto(4z[x+y],-16z[xy+xz+yz])=(u,v)\end{gather*}
is~$v\indel{u}+6u^2\indel{v}$. The latter has the first integral~$g_3=4u^3-v^2$, from which one can obtain a~first integral for the original equation. A solution of the reduced system is~$(u,v)=(\wp(t),\wp'(t))$, for the Weierstrass function~$\wp$ such that~$(\wp')^2=4\wp^3-g_3$. Solving for~$y$ and~$z$,
\begin{gather*}y(t)=-\frac{\wp^2+x\wp'}{4x\wp+\wp'}, \qquad z(t)=\frac{4x\wp+\wp'}{4\big(4x^2-\wp\big)};\end{gather*}
substituting into~(\ref{yetanothervf}), the system reduces to
$x'(t)=-(k+1)x^2+\frac{1}{4}(k-1)\wp(t)$. For~$w=-(k+1)x$, $w'=w^2+\frac{1}{4}\big(1-k^2\big)\wp$ and the equation is semicomplete if~$6\nmid k$. These vector fields appeared in~\cite{these} as the family~$\mathcal{T}_2$.

\paragraph{Case (A), $\boldsymbol{(u_4,v_4)=(-1,-3)}$ without invariant plane.}\label{slimmo} We do not need to go through a full analysis like the one of the preceding section for, as we previously mentioned, the obstruction lies at the level of the foliation, and the vector fields of this section and of those of the previous one induce the same foliations.

For~$X(m)=mY(6[1-m]/m)+2(m-6)yE$, for the vector field~$Y(\cdot)$ of equation~\ref{t2} we obtain equation~\ref{t1}. For~$\beta$ as in~(\ref{beta}) we have~$\beta=\frac{1}{4}(7m-6)(5m-6)m^{-1}(m-6)^{-1}$: the vector field is a nondegenerate one whenever~$\beta\in\mathbb{C}\setminus\{0\}$. This vector field has radial orbits with exponents~$(1,2)$ at~$[1:0:0]$, $(-1,-2)$ at~$[0:1:0]$, and one with exponents~$(1,3)$ at~$[0:0:1]$. It induces a foliation on~${\mathbb{CP}}^2$ which is linearizable at this point (but such that the line tangent to the exponent~$1$ is not invariant). These vector fields are of Halphen type. For
\begin{gather*}C=-\frac{x^2-xz-yz}{y^2-xz-yz}\del{x}+\del{y}+\frac{z(x-y)}{y^2-xz-yz}\del{z},\end{gather*}
we have~$[C,X]=2(m-6)E$. They have the adapted function
\begin{gather*}\xi=\frac{z(x+y)^3}{(xy+yz+zx)^2},\end{gather*}
for which
\begin{gather*}-\frac{1}{\xi^2}\{\xi,t\}=\frac{1}{2}\frac{\big(1-\frac{1}{4}\big)(\xi-1)\xi-\big(1-\frac{1}{9}\big)(\xi-1)
+\big(1-\frac{1}{m^2}\big)\xi}{\xi^2(\xi-1)^2}.\end{gather*}
Thus, the vector field is semicomplete if~$m\in\mathbb{Z}\setminus\{-1,0,1\}$. This family had previously appeared in~\cite{these}, as the family~$\mathcal{T}_1$, and in~\cite{guillot-fourier}, as the family~$\mathcal{H}_{3c}$.\footnote{In the coordinates
$(\zeta_1,\zeta_2,\zeta_3)=(3z,y-2z,x-3y+4z)$, we find the vector field $-k\mathcal{H}_{3c}\big(k^{-1}\big)$, for
\begin{gather*}\mathcal{H}_{3c}(\gamma)=\zeta_1[8\gamma\zeta_1+12\zeta_2-(6\gamma-5)\zeta_3]\del{\zeta_1}+ \big[(6\gamma-1)\zeta_2^2+(2\gamma-1)(8\zeta_2+3\zeta_3)\zeta_1\big]\del{\zeta_2} \\ \hphantom{\mathcal{H}_{3c}(\gamma)=}{} +\big[(6\gamma-5)\zeta_3^2+8(3\gamma-2)(3\zeta_2+2\zeta_3)\zeta_2\big]\del{\zeta_3}.\end{gather*} (This formula corrects formula~(3.5c) in~\cite{guillot-fourier}.)}

\paragraph{Case (A), $(u_4,v_4)=(-1,-3)$, without invariant plane.}
By a discussion similar to the one of the preceding section, these vector fields belong to the family~$X(m)-8mzE$ for~$m\in\mathbb{C}$, for~$X(m)$ in equation~\ref{t1}. The exponents at the radial orbits of the invariant line~$z=0$ are those of~$X(k)$ and those at~$[0:0:1]$ are~$(-1,-3)$. For this vector field, the exponents at the radial orbit~$[2:2:1]$ are~$(3/(m+3),3(m-1)/(m+3))$, and those at the radial orbit~$\big[2m(m-6):2m(m+6):(m-6)^2\big]$ are~$(3m/(m-9),-3/(m-9))$. If~$3/(m+3)=n\in\mathbb{Z}\setminus\{0\}$, $-3/(m-9)=n/(4n-1)$ is not a nonzero integer.

\paragraph{Case (B), $\boldsymbol{(u_4,v_4)=(-1,-3)}$, without invariant plane.} \label{sliippo}
These vector fields belong to the family~$Z(k)=Y(k)-8zE$ for~$k\in\mathbb{C}$, for~$Y(k)$ of equation~\ref{t2}. The exponents at the invariant line~$z=0$ are those of~$Y(k)$ and those at~$[0:0:1]$ are~$(-1,-3)$. The exponents of~$Z(k)$ at the radial orbit~$\rho_1=[2:2:1]$ are $(-k/3,2+k/3)$, and we thus need that~$k/3\in\mathbb{Z}$. The exponents of~$Z(k)$ at the radial orbit~$\big[2(1-k):2(1+k):(1+k)^2\big]$ are~$(6/(2k+3),k/(2k+3))$. Under the assumption that~$k/3\in\mathbb{Z}$, the first one is an integer only if~$k=0$ or~$k=-3$. In the first case~$\rho_1$ has a zero exponent, which is incompatible with the regularity hypothesis; in the second, the exponents of~$\rho_1$ are~$(1,1)$, bringing us to the setting of Section~\ref{im1f}.

\subsubsection{In~(\ref{form:sil}), $\boldsymbol{1/\xi_4+1/\xi_5=1/3}$.} Only three pairs of integers satisfy this equation: $(4,12)$, $(6,6)$, and~$(2,-6)$. Choosing one of these fixes essentially five pairs of exponents.
\paragraph{If~$\boldsymbol{(\xi_4,\xi_5)=(2,-6)}$.} A radial orbit with product of exponents equal to~$2$ appears, and places us in the setting of either Section~\ref{2fnc} or~\ref{connect}.
\paragraph{If~$\boldsymbol{(\xi_4,\xi_5)=(4,12)}$ and either~$\boldsymbol{(u_4,v_4)=(2,2)}$ or~$\boldsymbol{(u_4,v_4)=(-2,-2)}$.} \label{no22}
For a radial orbit with exponents~$(2,2)$, there exist three invariant planes containing this orbit (Lemma~3.6 in~\cite{guillot-fourier}). Each one of these three invariant planes contains one further radial orbit at its intersection with the original invariant plane. From~(\ref{tess}), the sets of exponents of a nondegenerate semicomplete quadratic homogeneous vector field on~$\mathbb{C}^2$ containing the exponent~$2$ are~$(-2,1,2)$, $(2,3,6)$ and~$(2,4,4)$. Thus, for each radial orbit in the invariant line, the exponent that is not tangent to this line must belong to~$S=\{-2,1,3,4,6\}$. This does not happen: both in cases~(A) and~(B), one of these exponents is~$2$, which is not in~$S$. Likewise, considering a~radial orbit with exponents~$(-2,-2)$, the only nondegenerate semicomplete quadratic homogeneous vector field on~$\mathbb{C}^2$ having an exponent~$2$ has exponents~$(-2,1,2)$, but the exponents of the original invariant plane are incompatible with this.

\paragraph{The ``special'' case.} This is case~(B) when~$(u_4,v_4)=(2,3)$ and~$(u_5,v_5)=(2,3)$. The two radial orbits with exponents~$(2,3)$ cannot be coplanar with one with exponents~$(1,2)$ in the invariant line, and we may thus assume that these four radial orbits are in general position. The vector fields that have~$\{z=0\}$ as an invariant line and radial orbits at~$[1:q:0]$ for~$q^2=-3$, each with exponents~$(1,2)$, as in diagram~(B), and that have radial orbits at~$[1:1:2]$ and~$[1:-1:2]$, with exponents~$(2,3)$ each, form two one-parameter families. (If the position of the other radial orbit within~$z=0$ is set to be at~$[1:\beta:0]$, $\beta$ becomes a parameter.) One of them is the family of Lins Neto's vector fields (equation~\ref{linsnetovf}, see Section~\ref{sec:linsnetovf}). Vector fields in the other family are linearly equivalent to vectors fields in this one.

\paragraph{The other cases.} For all the remaining cases we can list all the possible values of~$u_4$, $v_5$, $u_5$ and~$v_5$ and calculate the value of~$r_5=\sum\limits_{i=1}^5 (1/u_i+1/v_i)$. It is equal to~$4$ in the ``special'' case and is different from~$4$ in all the other cases. Leaving the special case aside, in the solutions of the equation
\begin{gather*}\frac{1}{u_6}+\frac{1}{v_6}+\frac{1}{u_7}+\frac{1}{v_7}=4-r_5,\end{gather*}
one number~(say~$u_6$) is determined up to a finite choice by the above equation: for example, if~$r_5<4$, at least one of the summands must be positive and must belong to the interval~$[1,4/(4-r_5)]$, and so on. Once the value of~$u_6$ is determined, we may solve the three equations~$\mathrm{R}_i$ in terms of the three unknowns~$v_6$, $u_7$ and~$v_7$. We obtain the following integer solutions.\footnote{See the file \texttt{p4.sage}.} In case~(B),
\begin{gather*}(u_4,v_4)=(2,3), \qquad (u_5,v_5)=(-2,-3), \qquad (u_6,v_6)=(2,3), \qquad (u_7,v_7)=(1,-6),\end{gather*}
which, up to renumbering, falls within the setting defining the special case; in case~(A),
\begin{gather*}(u_4,v_4)=(2,3),\qquad (u_5,v_5)=(2,3),\qquad (u_6,v_6)=(2,3), \qquad (u_7,v_7)=(1,-6).\end{gather*}
This last case is not realizable. If it were, following the discussion in Section~\ref{sec:fol}, it would differ from one of the vector fields of equation~\ref{linsnetovf} by the addition of multiple of Euler's vector field~$\sum_iz_i\indel{z_i}$ by a homogeneous linear form, which must vanish at the three radial orbits with exponents~$(2,3)$~-- which cannot be coplanar~-- but not at one of the radial orbits with exponents~$(1,2)$. This is impossible.

\subsubsection{The remaining cases in~(\ref{form:sil})} We are in this case whenever~$\xi_4$, $\xi_5$, $\xi_6$ and~$\xi_7$ solve equation~(\ref{form:sil}) but where the this equation is not satisfied for any subset of the summands. The solutions to these equations are finite in number. For each solution one may determine the finitely many values that~$(u_4,v_4)$, $(u_5,v_5)$, $(u_6,v_6)$, and~$(u_7,v_7)$ may take and test, for each case, if the relations~$\mathrm{R}_i$ are fulfilled. In this way we obtain the following sets of four couples~$(u_4,v_4)$, $(u_5,v_5)$, $(u_6,v_6)$, and~$(u_7,v_7)$:\footnote{See the file \texttt{p5.sage}.}
\begin{enumerate}\itemsep=0pt
\item\label{lastcase1} $(2,2)$, $(1,6)$, $(1,-10)$, $(6,10)$,
\item $(2,2)$, $(3,5)$, $(2,12)$, $(-2,20)$,
\item $(2,2)$, $(4,4)$, $(4,8)$, $(-4,24)$,
\item\label{lastcase4} $(2,2)$, $(2,4)$, $(-6,8)$, $(-6,8)$,
\item\label{not5} $(1,4)$, $(1,20)$, $(-4,-7)$, $(-4,105)$,
\item \label{t51}$(1,4)$, $(2,3)$, $(-3,8)$, $(-3,8)$,
\item\label{t52} $(1,4)$, $(2,3)$, $(-2,7)$, $(-7,12)$.
\end{enumerate}

The first one belongs to case~(A), the rest to case~(B). For each one of these cases, we must first determine if the corresponding sets of exponents are realizable.

The argument given in Section~\ref{no22} proves that cases~(\ref{lastcase1}) through~(\ref{lastcase4}) cannot be realized. We are left with the exponents in~(\ref{not5}), (\ref{t51}) and~(\ref{t52}), all of which belong to case~(B) and have a radial orbit with eigenvalues~$(1,4)$. Let us construct vector fields having these data. Choose~$z=0$ as the invariant line, and place the radial orbits~$\rho_1$ and~$\rho_2$ with exponents~$(1,2)$ at~$[1:0:0]$ and~$[0:1:0]$. Let~$\rho_4$ and~$\rho_5$ be, respectively, radial orbits at~$[0:0:1]$ and~$[1:1:1]$. The most general vector field satisfying these conditions is
\begin{gather*}\big(ax^2+dyz+ezx\big)\del{x}+\big(by^2+fyz+gzx\big)\del{y}+z(cz-ax-by)\del{z},\end{gather*}
with~$a=h-d-e$, $b=h-f-g$, $c=h-a-b$. We would like~$\rho_4$ to be the radial orbit with exponents~$(1,4)$. The plane tangent to the exponent~$4$ cannot be invariant and contain one of the radial orbits~$\rho_1$ or~$\rho_2$, for in that case the invariant plane would need to have a third radial orbit with exponent~$4$, which cannot happen. In particular, the lines~$x=0$ and~$y=0$ cannot be simultaneously invariant. We impose the condition that~$u_4+v_4=5$ by setting $e=\frac{1}{2}(9h-3d-2f-3g)$. An explicit calculation shows that~$u_5+v_5=5$. In particular, case~(\ref{not5}) cannot be realized. The condition for~$\xi_4=4$ is that
\begin{gather}\label{unamas}2 g d+3 f g + 3 f d - 9 h f + 2 f^2 =0.\end{gather}

If~$d=-\frac{3}{2}f$, the above equation reduces to~$ f(5 f + 18 h)=0$. If~$f=0$, the condition for~$\xi_5=6$ reads~$13g^2-42hg+33h^2=0$; by solving it, we obtain equation~\ref{abel:square12-3} up to a permutation of~$x$ and~$y$ and a choice of the square root. If~$f=\frac{18}{5}h$, the condition for~$\xi_5=6$ reads~$325g^2-204gh+463h^2=0$; after solving it we find vector fields which have two nondegenerate radial orbits whose exponents are not integers.

If~$2d+3f\neq 0$, we may set $g =f(9h-3d - 2f)/ (2d + 3f)$ in order to have~(\ref{unamas}). The resulting vector fields form a variety that embeds into~${\mathbb{CP}}^2$, with homogeneous coordinates~$[d:f:h]$. The condition for the radial orbit~$\rho_4$ to be nondegenerate is that $d^2-3hd-f^2\neq 0$. We must now impose two further conditions. The first: that the exponents at~$\rho_5$ be~$(2,3)$. This can be done since the radial orbits with exponents~$(1,4)$ and~$(2,3)$ cannot be coplanar with a radial orbit in the invariant plane. The condition for this is that
\begin{gather*}36hd^2f+33h^2d^2+2f^2d^2-30f^2dh-90h^2df-2f^3d-2fd^3\\
\qquad{} -6hd^3+54h^2f^2+d^4+f^4=0.\end{gather*}
This is a rational quartic. Under the quadratic transformation
\begin{gather}\label{preratquad}d = qr + 27pr,\qquad f = qr - 18pr,\qquad h = 5pr +pq,\end{gather}
the equation becomes the quadratic one
\begin{gather}\label{ratquad}4050rp-350rq+324pq-3125r^2+q^2-28431p^2=0.\end{gather}
The second: that the induced foliation on~${\mathbb{CP}}^2$ has a linearizable singularity at~$\rho_4$. The condition for this is, after calculating it following a procedure analogous to the one done in Section~\ref{sliimmo},
\begin{multline*}(d-9h-f)(3d+2f-9h)\big(6hd+df-f^2\big)\big(d^2-f^2+3hf-hd\big) \\ \qquad{}\times \big(f^3+15f^2h-df^2+16dfh-27fh^2-d^2f-6hd^2+d^3+9h^2d\big)\\
\qquad{}\times \big(d^2-3hd+9hf-df\big)\big(27h^2d-16hd^2+d^3-9dfh-d^2f-df^2+f^3\big)=0.\end{multline*}
We may parametrize rationally equation~(\ref{ratquad}), substitute into~(\ref{preratquad}) and then substitute into this last equation which then becomes a polynomial equation in a single variable. Upon solving this equation, we obtain the vector fields in equations \ref{abel:square12-1} through~\ref{abel:burn8-2}, as well as other vector fields obtained from these by either exchanging~$x$ and~$y$ or by replacing one of the square roots in their expression with its Galois conjugate. A posteriori, equations~\ref{abel:square12-1} through~\ref{abel:square12-1} correspond to case~(\ref{t52}), equations~\ref{abel:burn8-1} through~\ref{abel:square12-2} to case~(\ref{t51}). Their integration will be carried out in the next section.

\section[Some Abelian surfaces with automorphisms and the associated equations]{Some Abelian surfaces with automorphisms\\ and the associated equations}\label{sec:abelian}

The Chazy~IX and Chazy~X equations and Lins Neto's vector fields~(equation~\ref{linsnetovf}) give examples of polynomial vector fields on affine surfaces (generic level surfaces of their homogeneous polynomial first integrals) which can be embedded as Zariski-open subsets of Abelian surfaces and such that, under this embedding, the vector fields become constant ones. The homogeneity of the first integral is associated to an automorphism of the Abelian surface, and this automorphism preserves the vector field up to a constant factor.

Other special quasihomogenous vector fields are bound to arise from Abelian surfaces having interesting cyclic automorphisms. Fujiki classified the finite automorphism groups of complex tori of dimension two~\cite{fujiki}. Table~6 in Fujiki's work features the classification of \emph{nonspecial} cyclic automorphism groups of complex tori, which will be relevant to our results.

\subsection{The Jacobian of Burnside's curve and its order eight automorphism} Consider the genus two curve~$C$ of equation~$y^2=x\big(x^4-1\big)$. It is usually known as \emph{Burnside's curve}, following Burnside's work on its uniformization~\cite{burnside}. It has an automorphism group of order~$48$; an order eight automorphism is given by~$(x,y)\mapsto \big(\lambda^2x,\lambda y\big)$, for~$\lambda$ an eighth root of unity. The Jacobian of~$C$ inherits an automorphism of order eight induced by this one. Let us follow Mumford's account on Jacobi's work concerning the algebraic construction of the Jacobian of~$C$~\cite[Chapter~IIIa]{mumford}. Consider~$\mathbb{C}^4$ with the coordinates~$(u_1,u_2,v_1,v_2)$. For
\begin{gather*}H=u_2u_1^3+u_2v_1^2-v_2^2-2u_2^2u_1,\\
K=u_1^4-2v_1v_2+u_2^2-3u_1^2u_2+u_1v_1^2,\end{gather*}
an affine chart of the Jacobian~$J$ of~$C$ is the affine variety defined by~$H\equiv 0$, $K\equiv 1$. These relations ensure that, for the polynomials~$f=x\big(x^4-1\big)$, $U=x^2+u_1x+u_2$ and~\smash{$V=v_1x+v_2$}, $U$~divides~$f-V^2$. The above order eight automorphism of~$C$ induces the order eight automorphism~$\ell$ on~$J$
\begin{gather}\label{autburnj}(u_1,u_2,v_1,v_2)\mapsto\big(\lambda^2 u_1,\lambda^4 u_2,\lambda^3 v_1,\lambda^5 v_2\big),\end{gather}
for which~$\ell^*H=\lambda^2H$, $\ell^*K=K$. The vector fields
\begin{gather}\label{vf:jb}
X= v_1\del{u_1}+v_2\del{u_2}+\frac{1}{2}\big(2u_2-3u_1^2\big)\del{v_1}+\frac{1}{2}\big(u_1^3+v_1^2-4u_2u_1\big)\del{v_2},\\
Y=v_2\del{u_1}+(v_2u_1-u_2v_1)\del{u_2}+\frac{1}{2}\big(u_1^3-4u_2u_1+v_1^2\big)\del{v_1}\nonumber\\
\hphantom{Y=}{} +\frac{1}{2}\big(u_1^4-2u_2^2-u_1^2u_2+u_1v_1^2\big)\del{v_2}\nonumber
\end{gather}
of~$J$ are preserved by the cyclic automorphism~(\ref{autburnj}) up a constant factor: $\lambda$ for the first, $\lambda^3$ for the second. When~$u_1$, $u_2$, $v_1$ and~$v_2$ are respectively given the weights~$2$, $4$, $3$ and~$5$, these vector fields become quasihomogeneous of degrees two and four. Both~$H$ and~$K$ are common first integrals of~$X$ and~$Y$ and are, respectively, quasihomogeneous of degrees~$10$ and~$8$. By homogeneity considerations, all the varieties given by~$H\equiv 0$ and~$K$ a nonzero constant, together with their vector fields, are equivalent.

We will now integrate equations~\ref{abel:burn8-1} and~\ref{abel:burn8-2} by showing that, when restricted to a generic level surface of a first integral, they are birationally equivalent to~$J$, and that this birational equivalence maps the vector field of the equation to the vector field~(\ref{vf:jb}). These equations have appeared in the analysis of other ones: if~$u_1(t)$ is part of a solution of~$X$, $y(t)=-\frac{1}{4}u_1(t)$ solves Cosgrove's reduced F-V equation~$y^{\mathrm{(iv)}}=20yy''+10(y')^2-40y^3$ (see~\cite[Section~5]{cosgrovep2}) in the particular case
\begin{gather*}(y''')^2-24y'''y'y-192y^5+80y''y^3+120(y')^2y^2+4y''(y')^2-8(y'')^2y=0,\end{gather*}
which has the first integral~$20(y')^2y-2y'''y-20y^4+(y'')^2$.

\subsubsection{Equation~\ref{abel:burn8-1}}
It has the polynomial first integral of degree~$8$
\begin{gather*}Q=z^3\big[zx-\big(1+\sqrt{2}\big)y^2+\sqrt{2}yz\big]\big[x^3+\big(\sqrt{2}-1\big)y^2z+\big(3\sqrt{2}-4\big)yz^2\\
\hphantom{Q=}{} +\big(3-2\sqrt{2}\big)xz^2+\big(3\sqrt{2}-5\big)x^2z +\big(\sqrt{2}-2\big)x^2y+\big(8-6\sqrt{2}\big)yzx\big],\end{gather*}
and commutes with a homogeneous polynomial vector field of degree~$4$ having the same first integral. From a solution~$(x,y,z)$ to this equation, setting
\begin{gather*} u_1 =4\big(1-\sqrt{2}\big)\big(x+\sqrt{2}y\big)z, \\
v_1 =8\big(\sqrt{2}-1\big)\big[\big(2-2\sqrt{2}\big)yz+\big(\sqrt{2}-2\big)zx-xy\big]z, \\
u_2 =16\big(7-5\sqrt{2}\big)\big[zx-\big(1+\sqrt{2}\big)y^2+\sqrt{2}yz\big]z^2, \\
v_2 =32\big(7-5\sqrt{2}\big)\big[zx-\big(1+\sqrt{2}\big)y^2+\sqrt{2}yz\big]xz^2\end{gather*}
we obtain a solution to the vector field~$X$ in~(\ref{vf:jb}) taking values in~$H=0$, $K=\big(4352-3072\sqrt{2}\big)Q$. An inverse to this map is given by
\begin{gather*}x=\frac{u_1^4-2u_2u_1^2+u_1v_1^2+\sqrt{2}v_1v_2}{2\big(u_1v_2+\sqrt{2}v_1u_2\big)},\qquad
y=\frac{\sqrt{2}v_2+v_1u_1}{2\big(u_1^2-2u_2\big)}, \\ z=-\frac{\big(1+\sqrt{2}\big)\big(u_1v_2+\sqrt{2}v_1u_2\big)\big(u_1^2-2u_2\big)}
{2\big(2v_2^2-4u_1^3u_2+4u_2^2u_1+u_1^5+u_1^2v_1^2+2\sqrt{2}u_1v_1v_2\big)}.\end{gather*}
To integrate this equation, based on the homogeneity properties of both equations, we looked for solutions to~(\ref{vf:jb}) of the form~$u_1=P(x,y,z)$, with~$P$ a quadratic homogeneous polynomial. (The vector field~$X$ determines~$v_1$ from the derivative of~$u_1$, $u_2$ from the derivative of~$v_1$, $v_2$ from the derivative of~$u_2$, and establishes an algebraic equation for~$P$ from the derivative of~$u_2$.)

\subsubsection{Equation~\ref{abel:burn8-2}}
The vector field has the homogeneous first integral
\begin{gather*}
Q=\big[x^3z-8\big(1+\sqrt{2}\big)x^2y^2-\big(2\sqrt{2}+3\big)y^3z-4yz^3+\big(12+8\sqrt{2}\big)xz^3-\big(6\sqrt{2}+10\big)x^2z^2\\
\hphantom{Q=}{} +\big(6+2\sqrt{2}\big)y^2z^2+\big(12\sqrt{2}+11\big)zxy^2+\big(15+14\sqrt{2}\big)zx^2y-20\big(1+\sqrt{2}\big)z^2xy\big] \\
\hphantom{Q=}{}\times \big[x+\big(2\sqrt{2}-3\big)y\big]z^3
\end{gather*}
and commutes with a homogeneous polynomial vector field of degree~$4$. From a solution~$(x,y,z)$ to the equation, setting
\begin{gather*}
u_1 =16\big(17-12\sqrt{2}\big)\big[\big(2\sqrt{2}-2\big)y^2+\big(3-2\sqrt{2}\big)yz -zx\big], \\
u_2 =512\big(\sqrt{2}-1\big)\big[x+\big(2\sqrt{2}-3\big)y\big](y-z)^2z, \\
v_1 = \frac{64}{7}\big(6-5\sqrt{2}\big)\big[7\big(3-2\sqrt{2}\big)y^2z-7yzx+ \big(5\sqrt{2}-8\big)yz^2\\
\hphantom{v_1 =}{} + \big(4+\sqrt{2}\big)z^2x+\big(8\sqrt{2}-10\big)y^3\big], \\
v_2 =1024\big(3\sqrt{2}-4\big)\big(2y^2-3yz+zx\big)\big[x+\big(2\sqrt{2}-3\big)y\big](y-z)z,\end{gather*}
we obtain a solution to~(\ref{vf:jb}) for~$H=0$ and~$K=65536\big(17-12\sqrt{2}\big)Q$. This solution has an inverse: by setting
\begin{gather*}y=\frac{\big(2+\sqrt{2}\big)\big(u_1v_1-\sqrt{2}v_2\big)}{8\big(u_1^2-2u_2\big)},\end{gather*}
we can easily find~$x$ and~$z$ from~$u_1$ and~$v_1$. (The integration of this equation was obtained in the same way as the previous one.)

\subsection{The anharmonic square with an automorphism of order twelve}
For~$\rho^3=1$, let~$\Lambda=\langle 1,\rho\rangle$ and let~$E_\rho$ be the elliptic curve~$\mathbb{C}/\Lambda$. It has an automorphism of order three given by multiplication by~$\rho$. The Abelian variety~$E_\rho\times E_\rho$ has the cyclic automorphism of order twelve induced by the linear mapping~$\ell$ of~$\mathbb{C}^2$, $\ell(u,v)=\big(\rho^2 v,-\rho^2 u\big)$, the composition of the commuting mappings of orders three and four~$(u,v)\mapsto\big(\rho^2 u, \rho^2 v\big)$ and~$(u,v)\mapsto(v, -u)$. Let~$\wp$ be the Weierstrass elliptic function of~$E_\rho$, and let~$g_3\in \mathbb{C}^*$ be such that~$(\wp')^2=4\wp^3-g_3$. The field of meromorphic functions of~$E_\rho\times E_\rho$ is generated by~$\wp(u)$, $\wp'(u)$, $\wp(v)$ and $\wp'(v)$. Since~$\wp(\rho t)=\rho \wp(t)$ and~$\wp'(\rho t)=\wp'(t)$, the cyclic group generated by~$\ell$ acts linearly upon the vector space generated by these functions. A basis of eigenvectors is given by
\begin{gather*}a=\wp(u)+\wp(v), \qquad b=\wp(u)-\wp(v), \qquad c=\wp'(u)+{\rm i}\,\wp'(v), \qquad d=\wp'(u)-{\rm i}\,\wp'(v).\end{gather*}

We have that~$\ell^*(a,b,c,d)=\big(\lambda^8 a,\lambda^2 b,\lambda^9 c,\lambda^3 d\big)$ for~$\lambda=-{\rm i}\rho$.

We can embed a Zariski-open subset of~$E_\rho\times E_\rho$ into~$\mathbb{C}^4$ via~$(a,b,c,d)$. From the relation~$(\wp')^2=4\wp^3-g_3$, $g_3=\frac{1}{2}(a+b)^3-\frac{1}{4}(c+d)^2$ and~$g_3=\frac{1}{2}(a-b)^3+\frac{1}{4}(c-d)^2$. From the difference and sums of the right-hand sides we have the polynomials
\begin{gather*}\Delta=3a^2b+b^3-\frac{1}{2}\big(c^2+d^2\big), \qquad 2G_3=a^3+3ab^2-cd.\end{gather*}
The image of the Zariski-open subset of~$E_\rho\times E_\rho$ under the above-defined map is the variety~$A$ given by~$\Delta\equiv 0$ and~$G_3\equiv g_3$. The order 12 automorphism of~$E_\rho\times E_\rho$ extends to~$\mathbb{C}^4$ as the linear automorphism~$(a,b,c,d)=\big(\lambda^8 a,\lambda^2 b,\lambda^9 c,\lambda^3 d\big)$, which preserves~$A$ since~$\ell^*\Delta=-\Delta$ and~$\ell^*G_3=G_3$. The vector fields~$D_{\pm}=\indel{u}\pm {\rm i}\indel{v}$ of~$\mathbb{C}\times\mathbb{C}$ are preserved, up to a constant multiplicative factor, by~$\ell$: $\ell_*D_+=\lambda D_+$, $\ell_*D_-=\lambda^7 D_-$. Through the embedding, the vector fields corresponding to~$D_+$ and~$D_-$ are, respectively,
\begin{gather*}X_+=c\del{a}+d\del{b}+6ab\del{c}+3\big(a^2+b^2\big)\del{d},\\
 X_-=d\del{a}+c\del{b}+3\big(a^2+b^2\big)\del{c}+6ab\del{d}.\end{gather*}
Both~$\Delta$ and~$G_3$ are first integrals of these. If we give~$a$ and~$b$ the weight~$2$ and~$c$ and~$d$ the weight~$3$, $\Delta$ and~$G_3$ are quasihomogeneous of degree six, and~$D_+$ and~$D_-$ are quasihomogeneous of degree two. There is another interesting graduation for these objects taking values in~$\mathbb{Z}/12\mathbb{Z}$. If we give~$a$, $b$, $c$ and~$d$ the weights~$8$, $2$, $9$ and~$3$ in~$\mathbb{Z}/12\mathbb{Z}$, respectively, $X_+$ (resp.~$X_-$) is quasihomogeneous of degree~$2$ (resp.~$8$), while~$\Delta$ and~$G_3$ are quasihomogeneous of degrees~$6$ and~$0$.

We will integrate equations~\ref{abel:square12-1}, \ref{abel:square12-2} and~\ref{abel:square12-3} by showing that, when restricted to a~level set of a polynomial first integral, they are birationally equivalent to this algebraic model of~$E_\rho\times E_\rho$ with the vector field~$X_+$.

\subsubsection{Equation~\ref{abel:square12-1}}
It has the homogeneous polynomial first integral of degree~12
\begin{gather*}Q=\big[\big(20+12\sqrt{3}\big)xz^3-\big(30+19\sqrt{3}\big)x^2z^2+\big(6+6\sqrt{3}\big)x^3z- \big(36+20\sqrt{3}\big)z^3y \\
\hphantom{Q=}{} -\big(2+\sqrt{3}\big)z^2y^2+\big(66+36\sqrt{3}\big) xz^2y-\big(14\sqrt{3}+26\big)x^2zy+2x^4\big]\\
\hphantom{Q=}{}\times \big[zx-\sqrt{3}yz+\big(\sqrt{3}-1\big)y^2\big]^2z^4,\end{gather*}
and commutes with a vector field of degree eight having the same first integral.\footnote{While the degree of the first integral and the fact that the vector field commuted with a homogeneous polynomial pointed towards a birational equivalence with the Chazy~X equation, the degree of the commuting polynomial ruled this out.} From a~solution~$(a,b,c,d)$ to~$X_+$,
\begin{gather*} x=\frac{c}{2a},\qquad y=-\frac{\big(3+\sqrt{3}\big)c^4+48b^2a^4-6\big(4+\sqrt{3}\big)a^2bc^2+2\sqrt{3}da^3c}{2a\big[\big(18+4\sqrt{3}\big)cba^2
-\big(3+\sqrt{3}\big)c^3-6da^3\big]}, \\
z=\frac{\big(3-\sqrt{3}\big)\big[\big(18+4\sqrt{3}\big)cba^2-\big(3+\sqrt{3}\big)c^3-6da^3\big]}{12 \big(6ba^2-c^2\big)a}\end{gather*}
gives a solution to equation~\ref{abel:square12-1}. Reciprocally, from a solution to this equation in the level set~$Q\equiv Q_0$, for~$\theta^2=\big(26+15\sqrt{3}\big)/Q$ ($\theta$ is quasihomogeneous of degree~$-6$ and is constant along the solution),
\begin{gather*}a = \frac{\theta}{2}\big(5-3\sqrt{3}\big)\big[6\big(11+6\sqrt{3}\big)z^2xy-2\big(13+7\sqrt{3}\big)x^2zy+ 6\big(\sqrt{3}+1\big)x^3z\\
\hphantom{a=}{}- \big(19\sqrt{3}+30\big)z^2x^2 +2x^4-\big(\sqrt{3}+2\big)z^2y^2+ 4\big(3\sqrt{3}+5\big)xz^3- 4\big(9+5\sqrt{3}\big)z^3y\big]\\
\hphantom{a=}{}\times \big[zx-\sqrt{3}yz+\big(\sqrt{3}-1\big)y^2\big]z^2,\\
b = \frac{1}{2}\big[2x^2-\big(\sqrt{3}+3\big)yz+\big(\sqrt{3}+1\big)zx\big], \\
c = 2xa,\\
d = 2x^3-\big(11+5\sqrt{3}\big)xyz+3\big(1+\sqrt{3}\big)x^2z+3\big(4+2\sqrt{3}\big)yz^2-2\big(3+2\sqrt{3}\big)xz^2,\end{gather*}
we obtain a solution to~$X_+$ taking values in~$\Delta\equiv 0$, $G_3^2\equiv \big( 26+15\sqrt{3}\big)Q_0$.\footnote{In order to integrate this equation, from a solution~$(x,y,z)$ to equation~\ref{abel:square12-1}, we looked for a solution of~$X_+$ where~$b$ was a quadratic polynomial~$P(x,y,z)$. Through~$X_+$, this determined~$d$, which, in turn, determined~$a^2$. The logarithmic derivative~$\xi=a'/a$ (which can be calculated from~$a^2$) satisfies the equation~$\xi'=6b-\xi^2$, and this established an algebraic relation for~$P$. From~$2G_3=a\big(a^2+3b^2+\xi\big)$ we determined a value of~$a$, completing the solution.}

\subsubsection{Equation~\ref{abel:square12-2}} The integration is similar to that of the previous equation. It has the homogeneous polynomial first integral of degree~12
\begin{gather*}Q=\big[4\sqrt{3}zx+8\big(1-\sqrt{3}\big)yz-x^2\big] \\
\hphantom{Q=}{}\times \big[6zyx-2\big(1-\sqrt{3}\big)yz^2+\big(\sqrt{3}-2\big)xy^2-2x^2z -\sqrt{3}xz^2-2\sqrt{3}y^2z\big]^2z^4\end{gather*}
and commutes with a vector field of degree~8 having the same first integral. For a solution~$(x,y,z)$ to equation~\ref{abel:square12-2} and for~$\theta^2=-1/Q$, by setting
\begin{gather*}a =- \theta\big[8\big(1-\sqrt{3}\big)yz+4\sqrt{3}zx-x^2\big](x-z) \\ \hphantom{a=}{}\times\big[6zyx-2\big(1-\sqrt{3}\big)yz^2-2\sqrt{3}y^2z - \sqrt{3}xz^2+\big(\sqrt{3}-2\big)xy^2-2x^2z\big]z^2, \\
b=x^2-\big(3+2\sqrt{3}\big)zx+2\big(1+\sqrt{3}\big)yz,\\
c=-2\theta \big[\big(3+\sqrt{3}\big)zx-\big(2+\sqrt{3}\big)yz-x^2\big] \big[8\big(1-\sqrt{3}\big)yz+4\sqrt{3}zx-x^2\big]\\
\hphantom{c=}{}\times \big[2\big(\sqrt{3}-1\big)yz^2-2\sqrt{3}y^2z-\sqrt{3}xz^2+6zyx+ \big(\sqrt{3}-2\big)xy^2-2x^2z\big]z^2,\\
d=2\big[3\big(1+\sqrt{3}\big)x^2z-3\big(1+\sqrt{3}\big)z^2x+4\sqrt{3}yz^2+\big(1-4\sqrt{3}\big)zyx- x^3\big],\end{gather*}
we obtain a solution of~$X_+$. For the inverse, we can easily express~$y$ and~$z$ in terms of~$b$, $d$ and~$x$. Setting
\begin{gather*}x= \frac{\big(1+\sqrt{3}\big)\big(bd-\sqrt{3}ca\big)}{3a^2+b^2},\end{gather*}
completes the solution.

\subsubsection{Equation~\ref{abel:square12-3}} It has a homogeneous polynomial first integral of degree 12 and commutes with a homogeneous polynomial vector field of degree~$8$ having the same first integral. Under the birational mapping (homogeneous of degree one)
\begin{gather*}(x,y,z)\mapsto \left( \frac{x^2z+(2-\sqrt{3})(2z-6x)yz+(3-2\sqrt{3}
)(xz-y^2)z-(4\sqrt{3}-7)xy^2}{[x+\big(1-\sqrt{3}\big)y-(2-\sqrt{3})z](y-z)}\right.,\\ \left.
\hphantom{(x,y,z)\mapsto}{} \frac{(\sqrt{3}-3)yz+xz+(2-\sqrt{3})y^2}{y-z}, \frac{z[x+\big(1-\sqrt{3}\big)y+\big(\sqrt{3}-2\big)z]}{y-z}\right)\end{gather*}
the vector field of equation~\ref{abel:square12-3} gets mapped to the one of equation~\ref{abel:square12-1}. This proves that its restriction to a level surface of its first integral is birationally equivalent to~$D_+$ on~$A$.

\section{Some Riccati equations with rational coefficients}
We will now study the Riccati equations with rational coefficients appearing in our investigations. They are of the form~$y'=y^2+A(t)$ with~$A$ a rational function such that each one of its poles is a double one. In~\cite[Section~3]{guillot-riccati}, we have stated conditions that, in some cases, guarantee that all the solutions to such equations are meromorphic in the neighborhood of these double poles. We will use these conditions to investigate the settings corresponding to equations having exclusively single-valued solutions.

\subsection{First equation}\label{sec:ricrac1} We begin with the equation $z'=z^2+t^{-2}F(t)$ for
\begin{gather}\label{1striccati}F(t)=\frac{1}{4}\left[m^2\big(1-q^2\big)\left(\frac{t^{m}}{t^{m}-1}\right)^2-\big[r^2-p^2+m^2 \big(1-q^2\big)\big]\left(\frac{t^{m}}{t^{m}-1}\right)+\big(1-p^2\big) \right],\end{gather}
which appeared in Section~\ref{1ricrat}. We affirm that~\emph{if~$m\nmid p$ and for some~$j\in\{0,\ldots,q-1\}$, either~$p-r=(q-1-2j)m$ or~$p+r=(q-1-2j)m$, all the solutions are meromorphic} (notice that if both conditions are satisfied we also have that~$m\nmid r$).

We begin by investigating the behavior of the solutions in the neighborhood of~$t=0$. Let~$\rho$ be an~$m$\textsuperscript{th} root of unity. Since~$F(\rho t)=F(t)$ and~$F(0)=\frac{1}{4}\big(1-p^2\big)$, Lemma~3.1 in~\cite{guillot-riccati} allows us to conclude that \emph{if~$m\nmid p$ all the solutions are meromorphic in the neighborhood of~$t=0$.}

We now study the solutions in the neighborhood of the values of~$t$ which are~$m$\textsuperscript{th} roots of unity. The mapping~$(t,z)\mapsto\big(\rho^{-1} t,\rho z\big)$ is a symmetry of the equation. Since it permutes the~$m$\textsuperscript{th} roots of unity, the behavior of the solutions at all of them is the same: if in the neighborhood of~$t=1$ all the solutions to the equation are meromorphic then in the neighborhood of the other~$m$\textsuperscript{th} roots of unity all the solutions to the equation will be meromorphic as well. In order to consider the quotient by this symmetry, let~$w=tz$, $\tau=t^m$ (they are both invariant under the above symmetry). Let~$G$ be the function such that~$G(t^m)=F(t)$. We have
\begin{gather*}\frac{{\rm d}w}{{\rm d}\tau}=\frac{1}{m\tau}\big(w^2+w+G(\tau)\big).\end{gather*}
Notice that since the mapping~$(t,z)\mapsto(t^m,tz)$ does not ramify along~$t=1$, the solutions of this last equation in the neighborhood of~$\tau=1$ correspond to those of~(\ref{1striccati}) at~$t=1$. For~$s=1-\tau$, $q>0$ and
\begin{gather}y=\frac{1}{m(s-1)}\left(w+\frac{1}{2}[1-m]\right)s+\frac{1}{2}(1-q),\nonumber\\
 \label{lastform}s\frac{{\rm d}y}{{\rm d}s}=y^2+qy-\frac{1}{4m^2}\left(\frac{m^2q^2-r^2}{1-s}+\frac{p^2-m^2}{(1-s)^2}\right)s.\end{gather}
In order to prove that in the neighborhood of~$s=0$ all the solutions to this last equation are meromorphic, it suffices to show that this equation has a formal solution vanishing at~$0$~\cite[Section~3.1]{guillot-riccati}.
By setting~$r=mR$, $p=mP$ and~$y=-s(v'/v+\frac{1}{2}(P+1)/(s-1))$ for a solution~$y$ to this last equation, we find that~$v$ is a solution of the hypergeometric equation
\begin{gather*}s(1-s)v''+[c-(a+b+1)s]v'-abv=0,\end{gather*}
for~$a = \frac{1}{2}(P+R +1- q)$, $b = \frac{1}{2}(P-R+1 - q)$ and~$c = 1 - q$. If~$v=\sum\limits_{i=0}^\infty v_is^i$ is a solution to this equation then for every~$i$
\begin{gather*}(i+1)(i+c)v_{i+1}=(i+a)(i+b)v_i.\end{gather*}
The condition to have a non-identically zero formal solution is that for some~$j\in\{0,\ldots,q-1\}$, either~$q-1-P+R=2j$ or~$q-1-P-R=2j$. In such cases, which correspond to the second one of the required conditions, we obtain a formal solution to~(\ref{lastform}) vanishing at~$0$.

Thus, when the stated conditions are satisfied, all the solutions to~(\ref{1striccati}) will be meromorphic in the neighborhood of the poles of the equation, and all the solutions will be single-valued. In order to prove that the solutions are, moreover, rational, we must understand the behavior of the solutions of~(\ref{1striccati}) at~$t=\infty$. By setting~$t=1/\tau$ and~$w=-t^2z-t$ in~(\ref{1striccati}), we have that~${\rm d}w/{\rm d}\tau=w^2+\tau^{-2}F\big(\tau^{-1}\big)$. Since~$F\big(\tau^{-1}\big)|_{\tau=0}=\frac{1}{4}\big(1-r^2\big)$ and~$F\big((\rho\tau)^{-1}\big)=F\big(\tau^{-1}\big)$, by Lemma~3.1 in~\cite{guillot-riccati}, the solutions are meromorphic in the neighborhood of~$\tau=0$ if~$m\nmid r$, a~condition implied by the two conditions that we already have. The solutions of~(\ref{1striccati}) are rational, for they are meromorphic at~$t=\infty$.

\subsection{Second equation}\label{sec:ricrac2}
We will now study the equations that appeared in Section~\ref{2neg2lines} (cases~$\kappa=-1$ and~$\kappa=2$), the Riccati equations
\begin{gather}\label{secondriccati}y'=y^2+\frac{1}{4}\big(1-q^2\big)\left[\frac{(u')^2-2uu''}{u^2}\right]+\frac{1}{8}\big(1-n^2\big)\frac{u''}{u},\end{gather}
where~$q,n\in\mathbb{Z}$, $q^2\neq 1$, with~$u$ a quadratic polynomial in~$t$. We claim that \emph{all the solutions of these equations are meromorphic if~$n$ is odd and~$n<2q$}.

The right-hand side the above equation has double poles at the zeros of~$u$ (even if they are double ones). In order to prove that all the solutions to the above equation are single-valued, it is sufficient to guarantee the existence of many local solutions defined in a neighborhood of the poles.

Multiplying~$u$ by a constant leaves the equation unchanged, so the parameters of the equation are actually the roots of~$u$. If~$y(t)$ is a solution to the above equation, $ay(at+b)$ will be a solution to the equation corresponding to replacing~$u$ by~$u(at+b)$. There are two cases to consider, the one where the roots of~$u$ are simple and the one where they are double.

If~$u$ has two different roots, we may suppose that~$u=t(t-1)$, so that the equation becomes
\begin{gather}\label{segundaultima}y'=y^2+\frac{1}{4}\big(1-q^2\big)\frac{1}{t^2(t-1)^2}+\frac{1}{4}\big(1-n^2\big)\frac{1}{t(t-1)}.\end{gather}
The situations at both poles are equivalent (for we may replace~$t$ by~$1-t$ following the previous discussion), so we may concentrate exclusively in the one at~$t=0$. For~$q>0$, in the variable~$z=ty+\frac{1}{2}(1-q)$, we have the equation
 \begin{gather}\label{riclast}t z' =z^2+q z -\frac{1}{4}\left(\frac{q^2-n^2}{1-t}+\frac{q^2-1}{(1-t)^2}\right)t.\end{gather}
Again, in order to prove that in the neighborhood of~$t=0$ all the solutions to equation~(\ref{riclast}) are meromorphic, it suffices to show that there is a formal solution vanishing at~$0$~\cite[Section~3.1]{guillot-riccati}. But this equation is a special case of equation~(\ref{lastform}) for~$m=1$, $p=q$ and $r=n$. Our previous analysis proves that this equation has a formal solution if~$n$ is odd and~$n<2q$ (supposing that both of them are positive). In these situations, all the solutions to~(\ref{segundaultima}) are meromorphic in the neighborhood of~$t=0$ and~$t=1$ and are thus single-valued.

In order to prove that the solutions are rational, we need to understand the behavior of the solutions of~(\ref{segundaultima}) at~$t=\infty$. By setting~$t=\frac{1}{2}(s-1)/s$ and~$w=\frac{1}{2}(1-2t)^2y+(2t-1)$, we find
\begin{gather*}\frac{dw}{ds}=w^2+\frac{\big(3-4q^2+n^2\big)s^2+\big(1-n^2\big)}{4s^2\big(1-s^2\big)^2}.\end{gather*}
The right-hand side is invariant under the involution~$s\mapsto -s$. Lemma~3.1 in~\cite{guillot-riccati} allows us to conclude that this last equation has meromorphic solutions at~$s=0$ whenever~$n$ is odd. This proves that, under the stated conditions, the solutions to the original equations are rational, since they are meromorphic at infinity.

If in~(\ref{secondriccati}) $u$ has a double root, the solutions will also be single-valued under the required conditions, for the set of polynomials~$u$ giving exclusively single-valued solutions, which is a~closed one, contains, as we already proved, the polynomials with different roots. We can also establish this directly: for~$u=t^2$, the equation becomes simply
\begin{gather*}y'=y^2+\frac{1-n^2}{4t^2},\end{gather*}
which has, for~$c\neq 0$, the solutions
\begin{gather*}y=\frac{1-n-c(1+n)t^n}{2t\big(ct^n-1\big)},\end{gather*}
which are single-valued whenever~$n\in\mathbb{Z}$.

\subsection*{Acknowledgements}
The author gratefully acknowledges support from grant PAPIIT~IN102518 (UNAM, Mexico). He thanks the referees for their thorough reading and helpful comments and suggestions.

\LastPageEnding

\end{document}